\documentclass{article}
\usepackage[utf8]{inputenc}
\usepackage{tikz}
\usetikzlibrary{decorations.pathreplacing, calligraphy, positioning, graphs, decorations.markings, arrows.meta}
\usepackage{amsfonts, amsthm, amsmath, amssymb, apptools, array, comment, enumerate, esvect, float, mathdots, mathtools, placeins, setspace}
\usepackage{blindtext, subcaption, makecell, multicol, pgfplots}
\usepackage[symbol]{footmisc}

\makeatletter
\def\@seccntformat#1{\@ifundefined{#1@cntformat}%
   {\csname the#1\endcsname\quad}  
   {\csname #1@cntformat\endcsname}
}
\let\oldappendix\appendix 
\renewcommand\appendix{%
    \oldappendix
    \newcommand{\section@cntformat}{\appendixname~\thesection\quad}
}
\makeatother

\usepackage{geometry}
\usepackage{hyperref}
\usepackage{cleveref}

\newtheorem{lemma*}{Lemma}[section]
\newtheorem{lemma}{Lemma}[subsection]
\newtheorem{theorem*}[lemma*]{Theorem}
\newtheorem{theorem}[lemma]{Theorem}

\newtheorem{corollary}[lemma]{Corollary}
\newtheorem{proposition}[lemma]{Proposition}
\newtheorem{proposition*}[lemma*]{Proposition}
\theoremstyle{definition}
\newtheorem{definition}[lemma]{Definition}
\newtheorem{definition*}[lemma*]{Definition}

\newcommand{\F}{\mathbb{F}}

\newcommand{\Aut}{\operatorname{Aut}}
\newcommand{\Inn}{\operatorname{Inn}}
\newcommand{\ol}{\overline}
\newcommand{\PSL}{\operatorname{PSL}}
\newcommand{\SL}{\operatorname{SL}}

\newcommand{\Syl}{\operatorname{Syl}}
\newcommand{\Out}{\operatorname{Out}}
\newcommand{\Fit}{\operatorname{Fit}}
\newcommand{\Sz}{\operatorname{Sz}}

\newcommand{\divides}{\bigm|}
\newcommand{\ndivides}{%
  \mathrel{\mkern.5mu 
    \ooalign{\hidewidth$\big|$\hidewidth\cr$\nmid$\cr}%
  }%
}

\newcommand{\kyle}{
\begin{tikzpicture}[scale=0.12]
    \draw (-1, 0) -- (0, -1.2) -- (1, 0) -- (0, 1.2) -- (-1, 0) -- (1, 0);
    \draw[fill=black] (-1, 0) circle (0.25);
    \draw[fill=black] (1, 0) circle (0.25);
    \draw[fill=black] (0, -1.2) circle (0.25);
    \draw[fill=black] (0, 1.2) circle (0.25);
\end{tikzpicture}
}
\newcommand{\kfour}{
\begin{tikzpicture}[scale=0.12]
    \draw (-1, 0) -- (0, -1.2) -- (1, 0) -- (0, 1.2) -- (-1, 0) -- (1, 0);
    \draw (0,1.2) -- (0,-1.2);
    \draw[fill=black] (-1, 0) circle (0.25);
    \draw[fill=black] (1, 0) circle (0.25);
    \draw[fill=black] (0, -1.2) circle (0.25);
    \draw[fill=black] (0, 1.2) circle (0.25);
\end{tikzpicture}
}
\newcommand{\stargraph}{
\begin{tikzpicture}[scale=0.12]
    \draw (0,0) -- (1.386, 0.8);
    \draw (0,0) -- (-1.386, 0.8);
    \draw (0,0) -- (0,-1.6);
    \draw[fill=black] (0,0) circle (0.25);
    \draw[fill=black] (1.386, 0.8) circle (0.25);
    \draw[fill=black] (-1.386, 0.8) circle (0.25);
    \draw[fill=black] (0,-1.6) circle (0.25);
\end{tikzpicture}
}
\newcommand{\spoon}{
\begin{tikzpicture}[scale=0.12]
    \draw (0,0) -- (0,2.4) -- (1.2, 1.2) -- (0,0);
    \draw (1.2,1.2) -- (2.75, 1.2);
    \draw[fill=black] (0,0) circle (0.25);
    \draw[fill=black] (0,2.4) circle (0.25);
    \draw[fill=black] (1.2,1.2) circle (0.25);
    \draw[fill=black] (2.75,1.2) circle (0.25);
\end{tikzpicture}
}
\newcommand{\Aten}{
\begin{tikzpicture}[scale=0.12]
    \draw (0,0) -- (1.2,0) -- (2.4, 0);
    \draw[fill=black] (0,0) circle (0.25);
    \draw[fill=black] (1.2,0) circle (0.25);
    \draw[fill=black] (2.4,0) circle (0.25);
    \draw[fill=black] (1.2,1.2) circle (0.25);
\end{tikzpicture}
}

\hypersetup{%
	colorlinks,%
	linkcolor={blue!60!white},%
	citecolor={blue!60!white},%
	urlcolor={blue!60!white}%
}

\AtAppendix{\counterwithin{lemma}{section}}

\title{Classifying prime graphs of finite groups -- a methodical approach}
\author{Thomas Michael Keller, Gavin Pettigrew, Saskia Solotko, Lixin Zheng}
\date{}

\begin{document}

\maketitle

\begin{abstract}
    For a finite group $G$, the vertices of the prime graph $\Gamma(G)$ are the primes that divide $|G|$, and two vertices $p$ and $q$ are connected by an edge if and only if there is an element of order $pq$ in $G$. Prime graphs of solvable groups as well as groups whose noncyclic composition factors have order divisible by exactly three distinct primes have been classified in graph-theoretic terms. In this paper, we begin to develop a general theory on the existence of edges in the prime graph of an arbitrary $T$-solvable group, that is, a group whose composition factors are cyclic or isomorphic to a fixed nonabelian simple group $T$. We then apply these results to classify the prime graphs of $T$-solvable groups for, in a suitable sense, most $T$ such that $|T|$ has exactly four prime divisors. We find that these groups almost always have a 3-colorable prime graph complement containing few possible triangles.

\end{abstract}

\section{Introduction}

In this paper, we continue the investigation of prime graphs (Gruenberg-Kegel graphs) of finite groups. For a finite group $G$, the prime graph $\Gamma(G)$ is defined as follows: the vertices are the prime divisors of $|G|$, and there is an edge between primes $p$ and $q$ if and only if there is an element of order $pq$ in $G$. Gruenberg and Kegel introduced this notion in the 1970s to study certain cohomological questions of group rings, but it has been studied for its own sake ever since. 

For a long time, the research on prime graphs focused on simple groups and groups whose prime graphs have a specific graph-theoretic property (such as being acyclic, see \cite{lucido}). Until the last decade, classifying prime graphs of large families of groups in purely graph-theoretical terms -- that is, starting with a group-theoretic property and then classifying all the possible prime graphs of the groups having this property -- seemed to be out of reach. Nevertheless, the authors of \cite{siberian} classified prime graphs with at most five vertices, while \cite{baechle} classified prime graphs of solvable cut groups with three or four vertices. In \cite{debon}, the authors classified prime graphs of solvable rational groups, though the only primes dividing the order of a solvable rational group are 2, 3, and 5.


Classifying prime graphs of families of groups with many distinct prime divisors started in earnest with the 2015 paper by Gruber, Keller, Lewis, Naughton, and Strasser, where they classified the prime graphs of solvable groups. We state their result as a lemma.
\begin{lemma*}\label{ThomasMichaelKeller}
    \cite[Theorem 2.10]{2015} An unlabeled simple graph $\Gamma$ is isomorphic to the prime graph of a solvable group if and only if its complement is $3$-colorable and triangle-free.
\end{lemma*}
We say that a simple group is $K_n$ if its order is divisible by $n$ distinct primes. In \cite{2022} and \cite{A5}, the authors prove a broader classification than the one above, describing all groups whose composition factors are either cyclic or $K_3$. For groups possessing a unique $K_3$ composition factor (up to isomorphism), the prime graph complements often contain triangles, but remain 3-colorable. However, it was found that whenever two nonisomorphic $K_3$ groups appear as composition factors, the prime graph complements are always triangle-free and 3-colorable. The following terminology will be helpful.
\begin{definition*}
    Let $G$ be a group and $\mathcal{T}$ a set of distinct nonabelian simple groups. Then $G$ is {\it $\mathcal{T}$-solvable} if each of its composition factors are cyclic or in $\mathcal{T}$. Furthermore, $G$ is said to be {\it strictly $\mathcal{T}$-solvable} if $G$ is $\mathcal{T}$-solvable and has at least one composition factor from $\mathcal{T}$. If $\mathcal{T}$ has only one element $T$, we will refer to $\{T\}$-solvable simply as $T$-solvable. If $\mathcal T$ is the set of all $K_n$ groups (up to isomorphism), we will write $K_n$-solvable instead of $\mathcal T$-solvable.
\end{definition*}
In these terms, we can say that \cite{2022} and \cite{A5} classified the prime graphs of $K_3$-solvable groups. In this paper, we take the next natural step and study prime graphs of $T$-solvable groups, where $T$ is a $K_4$ group. Doing this for all $K_4$ groups is a near-Herculean task: the list of $K_4$ groups is quite lengthy (possibly infinite), and the groups are more structurally and computationally complex than $K_3$ groups. Therefore, while previous work on prime graphs used many ad hoc arguments, this paper builds towards a more general framework to streamline our classifications and help support future work on this subject. The development of this framework is as important, if not more so, than the classification results in Section \ref{section:allresults}, as it gives a methodical way to approach this problem that applies much more generally than previous arguments.

In Section \ref{section:preliminaries} of this paper, we begin to develop a general theory of the prime graphs of $T$-solvable groups for some nonabelian simple group $T$. A typical result from this section looks like this:
\begin{proposition*}[See 
\Cref{funny}]
    For each $p \in \pi(K) \setminus \pi(T)$, one of the following holds:
    
    \begin{enumerate}
    	\item For each odd prime $r \in \pi(T)$ satisfying $(r, |M(T)|) = 1$ (where $M(T)$ is the Schur multiplier of $T$), we have $r-p \notin \ol \Gamma(G).$
    	\item $G$ has a section $V.E$, where $V$ is a nontrivial elementary abelian $p$-group and $E$ is a perfect central extension of $T$.
    \end{enumerate}
\end{proposition*}

We use these fundamental results throughout Section \ref{section:allresults} of the paper to obtain numerous classification results when $T$ is $K_4$. Some computational aspects of the paper involving the program GAP \cite{GAP4} are discussed in Section \ref{section:GAP}. Then, in Appendix A, we further demonstrate the power of the theoretical foundation by revisiting $A_5$-solvable groups. Those prime graphs were classified in \cite[Section 6]{REU} and \cite{A5} on a total of more than 35 pages, but with our general techniques, we can obtain the same result on less than a third of that. In doing so, we also correct a mistake in the proof of Lemma 4.5 in \cite{A5}. Finally, Appendix B contains important tables of group-theoretic and representation-theoretic information used throughout the paper.

The following result of \cite{K4} lists every $K_4$ group.

\begin{lemma*}
    \cite[Lemma 2.2]{K4} Let $G$ be a $K_4$ group. Then $G$ is isomorphic to one of the following groups: $A_7$, $A_8$, $A_9$, $A_{10}$, $M_{11}$, $M_{12}$, $J_2$, $\PSL(3,4)$, $\PSL(3,5)$, $\PSL(3,7)$, $\PSL(3,8)$, $\PSL(3,17)$, $\PSL(4,3)$, $O_5(4)$, $O_5(5)$, $O_5(7)$, $O_5(9)$, $O_7(2)$, $O_8^+(2)$, $\Sz(8)$, $\Sz(32)$, $U_3(4)$, $U_3(5)$, $U_3(7)$, $U_3(8)$, $U_3(9)$, $U_4(3)$, $U_5(2)$, ${}^2F_4(2)'$, $G_2(3)$, ${}^3D_4(2)$, and $\PSL(2,q)$, where $q$ is a prime power satisfying
    \[q(q^2-1) = (2,q-1) \cdot 2^{\alpha_1} \cdot 3^{\alpha_2} \cdot r^{\alpha_3} \cdot s^{\alpha_4}\]
    for distinct primes $r,s$ and positive integers $\alpha_i$, $i \in \{1,2,3,4\}$.
\end{lemma*}
Aside from the additional complexity that comes from introducing a fourth prime divisor, these groups introduce many obstacles that did not exist when studying the prime graphs of $K_3$ groups. In \cite{2022} and \cite{A5}, the authors constantly exploited the fact that every $K_3$ group $T$ satisfies $\pi(T) = \pi(\Aut(T))$, a property which fails for some $K_4$ groups. This property is also important for us because a meaningful analysis of a strictly $T$-solvable group $G$ with the tools of \Cref{section:preliminaries} typically requires that $G$ contain a subgroup $K \cong N.T$, where $N$ is solvable and $\pi(G) = \pi(K)$. The best known method for producing such a $K$ is \Cref{2.4 2022}, which requires $\pi(T) = \pi(\Aut(T))$. Hence, although the tools of \Cref{section:preliminaries} are presented with greater generality when possible, the classification results of this paper only consider $T$ of this type. 
\begin{lemma*}\label{Aut}
    All $K_4$ groups $T$ satisfy $\pi(\Aut(T)) = \pi(T)$ except:
    \begin{enumerate}
        \item[(1)] The Suzuki group $\Sz(8)$
        \item[(2)] The $K_4$ groups $\PSL(2,p^f)$ when there exists a prime $r|f$ such that $r \ndivides |\PSL(2,p^f)|$.
    \end{enumerate}
\end{lemma*}

The smallest example of (2) is when $q = 32$, where we have $|\Out(\PSL(2,32))| = 5 \ndivides |\PSL(2,32)|$. We remark that $|\Out(\Sz(8))| = 3 \ndivides |\Sz(8)|$ as well. Verifying the remainder of this lemma is simply a matter of searching through the literature on the size of the outer automorphism group $|\Out(T)|$, which is widely available for the finite simple groups. We include some information in Appendix B, \Cref{table:1}.

An additional obstacle is that the number of $K_4$ groups is far greater than 8, the number of $K_3$ groups. In fact, it is unknown whether there are infinitely many $q$ for which $\PSL(2,q)$ is $K_4$. Thus a complete set of classifications of the prime graphs of $\PSL(2, q)$-solvable groups would have to rely on a very general relationship between the number $q$ and the structure of these groups. Taking note of their subgroups, central extensions, and representation data -- information that has proved integral to prior classification results -- for a couple dozen values of $q$, it appears that any useful relationship of this sort would have to be quite complex. For these reasons, our problem is significantly more difficult than the $K_3$ case. Nevertheless, in this paper, we obtain the following. \\

\noindent{\bf Main result.} We classify the prime graphs of $T$-solvable groups for all $K_4$ groups $T$ not isomorphic to $\Sz(8)$, $\Sz(32),$ or $\PSL(2,q)$ for any $q$. \\

The reason $\Sz(32)$ is in that collection, even though $\pi(\Sz(32))$ contains $|\Out(\Sz(32))| = 5$, has to do with computational issues regarding $K_4$-solvable groups. The sizes of $K_4$ groups are orders of magnitude larger than the sizes of $K_3$ groups, and this means that computations take much more time and memory to run. Specifically, computing the character table for $\Sz(32)$ is already a daunting task in GAP. Furthermore, there are memory limitations in finding extensions of $\Sz(32)$ that realize each possible configuration of the prime graph complement.

As an example of the flavor of our results, we provide a classification result from Section \ref{section:allresults} for $T = A_7$, which was one of the most difficult cases.

\begin{theorem*}[See Theorem \ref{a7bigproof}]
     Let $\Gamma$ be an unlabeled simple graph. Then $\Gamma$ is isomorphic to the prime graph of an $A_7$-solvable group if and only if one of the following is satisfied:
    \begin{enumerate}
    \item[(1)] $\ol \Gamma$ is triangle-free and 3-colorable.
    \item[(2)] There exist vertices $a,b,c,d$ in $\ol \Gamma$ and a $3$-coloring for which all vertices adjacent to but not included in $\{a,b,c,d\}$ have the same color. Additionally, one of the following holds:
    \begin{enumerate}
        \item[(2.1)] $\ol\Gamma$ contains exactly two triangles $\{a,b,c\}$ and $\{b,c,d\}.$ All edges incident to $a,b,d$ are included in these two triangles.
        \item[(2.2)] $\ol\Gamma$ contains exactly one triangle $\{a,b,c\}$ and a vertex $d$ adjacent only to vertex $c.$ Vertices $a$ and $b$ are not adjacent to any other vertices.
        \item[(2.3)] $\ol\Gamma$ contains exactly one triangle $\{a,b,c\}$ and a vertex $d$ that is adjacent only to vertex $b.$ Vertices $a$ and $b$ are not adjacent to any other vertices.
        \item[(2.4)] $\ol\Gamma$ contains exactly one triangle $\{a,b,c\}$ and a vertex $d.$  Vertex $a$ is not adjacent to any other vertices. Vertex $d$ is not adjacent to $a,b,$ or $c.$ Additionally, $N(b)\setminus\{a,c\}\subseteq N(d),$ and $N(c)\cap N(d)=\emptyset.$
        \item[(2.5)] $\ol\Gamma$ contains exactly one triangle $\{a,b,c\}$ and an isolated vertex $d$.The vertex $a$ is not adjacent to any other vertices.
    \end{enumerate}
    \end{enumerate} 
\end{theorem*}




It is also important to discuss some notation commonly used throughout this paper. Every group $G$ will be finite. Nonabelian simple groups will often be denoted by $T$, and the trivial group will be denoted by 1. Furthermore, $\pi(G)$ will refer to the set of primes that divide $|G|$. If $\pi \subset \pi(G)$, then $\pi'$ will denote the complement of $\pi$ in $\pi(G)$. In the case that $\pi = \{p\}$, we simply let $p' = \pi'$. When $a,b$ are integers, $(a,b)$ will refer to their greatest common divisor. Given a nonabelian simple group $T$, the letters $r$ and $s$ will usually denote primes that belong to $\pi(T)$, and $p$ and $q$ will usually denote primes that do not belong to $\pi(T)$. The commutator subgroup of a group $G$ will be denoted $G'$. Additionally, we will follow the ATLAS convention and write $G \cong N.M$ if $N \lhd G$ and $G/N \cong M$.

For a graph $\Gamma$, $\ol{\Gamma}$ will denote its complement, $V(\Gamma)$ its vertex set, and $E(\Gamma)$ its edge set. By abuse of notation, we will write $p \in \Gamma$ if $p \in V(\Gamma)$, and we will write $p-q \in \Gamma$ if $\{p,q\} \in E(\Gamma)$. If $E \subseteq E(\Gamma)$ and $V \subseteq V(\Gamma)$, then $\Gamma \setminus E$ will denote the graph $\Gamma$ without the edges of $E$ and $\Gamma \setminus V$ will denote the induced subgraph on $V(\Gamma) \setminus V$. If $p-q, p-r, q-r \in \Gamma$, then we will say that $\{p,q,r\}$ forms a triangle in $\Gamma$. Alternatively, we may say that $\Gamma$ has a $\{p,q,r\}$-triangle. We usually use $\vv{\Gamma}(G)$ to refer to an orientation of $\ol \Gamma(G)$. In this case, we write $p \rightarrow q \in \vv{\Gamma}(G)$ if $\vv\Gamma(G)$ contains the directed edge $(p,q)$. A directed path on $n$ edges will be called an $n$-path.

Throughout, $\Gamma$ will usually denote the prime graph of some group. Because the graph theoretic conditions of our classification results are more conveniently stated in terms of the complement $\ol \Gamma$ (see \Cref{ThomasMichaelKeller} for example), we will, for consistency's sake, state all graph theoretic properties of $\Gamma$ in terms of $\ol \Gamma$. Most importantly, we will write $p-q \notin \ol \Gamma$ instead of $p-q \in \Gamma$, and if $V \subseteq V(\Gamma)$, then $N(V)$ will always denote the neighborhood of $V$ in the complement $\ol \Gamma$.

\section{Prime graphs of general $T$-solvable groups}\label{section:preliminaries}

\indent The goal of this section is to develop a multitude of general tools for studying prime graphs of $T$-solvable groups. Some of these methods have been covertly applied in previous papers, such as \cite{2022} and \cite{A5}. However, few have been stated as important results in their own right. In our case, taking the time to establish these results will streamline the next section massively and offer a clear foundation upon which future work can be done. 

First, we present a series of algebraic lemmas, providing insight into the structure of a general $T$-solvable group $G$. Then, we focus our attention on edges of the form $r-p \in \ol \Gamma(G)$ where $r \in \pi(T)$ and $p \in \pi(G) \setminus \pi(T)$, which turn out to be the most important edges to understand when proving a classification result. Finally, we discuss orientations on prime graphs and applications in the problems of 3-coloring and constructing groups associated with a given graph.

\subsection{The structure of $T$-solvable groups}

To classify the prime graphs of $T$-solvable groups, we must be proficient at identifying whether a given edge $p-q$ belongs to $\ol \Gamma(G)$. The most common way to approach this problem is by considering a section of $G$. It is easy to see that if $\ol \Gamma(G)$ contains the edge $p-q$ for some $p,q \in \pi(G)$ and $L$ is any section of $G$ with order divisible by $p$ and $q$, then $p-q \in \ol \Gamma(L)$ as well. Thus, when proving that a generic edge $p-q$ does not exist in $\ol \Gamma(G)$, it is enough to exhibit a section $L$ of $G$ such that $p-q \notin \ol \Gamma(L)$. For example, if $G$ is strictly $T$-solvable and $r,s \in \pi(T)$, then $\ol \Gamma(G)$ contains the edge $r-s$ only if $\ol \Gamma(T)$ also does. In other words, the subgraph of $\ol \Gamma(G)$ induced by $\pi(T)$ is a subgraph of $\ol \Gamma(T)$. This is convenient, but edges of the form $r -p$ for $r \in \pi(T)$ and $p \in \pi(G) \setminus \pi(T)$ will prove much more troublesome to rule out. To study these edges, we must develop a deeper understanding of the structure of $T$-solvable groups. 

The goal of this subsection is to prove basic results on the structure of $T$-solvable groups and prime graphs, as well as some miscellaneous algebraic results that did not fit anywhere else. The first lemma shows how under certain conditions, $G$ can be split into two parts: one solvable and the other a subgroup of $\Aut(T)$ containing $T$.

\begin{lemma}\label{N.M}
    Let $G$ be strictly $T$-solvable for some nonabelian simple group $T$, and assume that there exist two vertices of $\pi(T)$ such that $\ol \Gamma(G)$ contains an edge between them. Then, there exists a subgroup $M \leq \Aut(T)$ containing $\Inn(T) \cong T$ and a solvable normal subgroup $N$ of $G$ such that $G \cong N.M$.
\end{lemma}

\begin{proof}
    Because the subgraph induced by $\pi(T)$ has at least one edge, we see from \cite[Lemma 2.5]{2022} that $T$ appears exactly once as a composition factor for $G$. Therefore, there exist solvable groups $N_0$ and $M_0$ such that $G \cong N_0.T.M_0.$ Now, let $H = G/N_0 \cong T.M_0$ and consider the centralizer $C=C_H(T)$. Taking $M=H/C$, we can find a solvable group $N \cong N_0.C$ such that $G \cong N.M$. The fact that $T$ is nonabelian and simple implies $T \cap C = 1$, so $M$ has a normal subgroup $TC/C \cong T/(T \cap C) \cong T$. Hence, the natural embedding $M \hookrightarrow \Aut(T)$ associated with the faithful action of $M$ on $T$ maps $TC/C \cong T$ isomorphically onto $\Inn(T)$.
\end{proof}

One simple consequence of this lemma is that when $\ol \Gamma(G)$ contains an edge between any two vertices of $\pi(T)$, then $G$ contains a subgroup $K \cong N.T$. Moreover, if $T$ satisfies $\pi(T) = \pi(\Aut(T))$ (see \Cref{Aut}), then we have $\pi(G) = \pi(K)$, and it follows that $\ol \Gamma(G)$ is obtained by removing edges from $\ol \Gamma(K)$. This is in fact true regardless of the edges that exist between vertices of $\pi(T)$. We quote a lemma from \cite{2022}.

\begin{lemma}\label{2.4 2022}
    \cite[Corollary 2.4]{2022} Let $G$ be strictly $T$-solvable, where $T$ is a nonabelian simple group and $\pi(T) = \pi(\Aut(T))$. Then $G$ has a subgroup $K \cong N.T$ with $N$ solvable and $\pi(G)=\pi(K).$
\end{lemma}

Because we will not be classifying the prime graphs of $\Sz(8)$-solvable or $\PSL(2,q)$-solvable groups in this paper, all $T$-solvable groups under consideration will satisfy the hypotheses of \Cref{2.4 2022}. Let us provide further information on the structure of $K$.

\begin{lemma}\label{chief series}
    Let $K \cong N.T$ for some solvable group $N$ and nonabelian simple group $T$. Let
    \[1 = K_0 \unlhd K_1 \unlhd \cdots \unlhd K_n = K\]
    be a chief series such that $K_{n-1} = N$, $K_{n}/K_{n-1} \cong T$. Then each $K_\ell/K_{\ell-1}$ is an elementary abelian $p_\ell$-group for some prime $p_\ell$, $1 \leq \ell \leq n-1$. If $i,j \in \{1,\ldots, n-1\}$ are distinct integers such that $p_i$ is odd, $p_i-p_j \in \ol \Gamma(G)$, and $p_i - s \in \ol \Gamma(G)$ for some $s \in \pi(T)$, then $i< j$.
\end{lemma}

\begin{proof}
    For each $1 \leq \ell \leq n-1$, $K_\ell/K_{\ell-1}$ is a solvable minimal normal subgroup of $K/K_{\ell-1}$ and hence is an elementary abelian $p_\ell$-group for some prime $p_\ell$. To prove the last statement of the lemma, fix $i,j \in \{1, \ldots, n\}$. It suffices to show that if $p_i$ is odd and divides $|K_n/K_{j}|$, then either $p_i$ is not adjacent to $p_j$ or $p_i$ is not adjacent to any elements of $\pi(T)$.

    Let $\ell$ be the maximal index such that $K_\ell/K_{\ell-1}$ is a $p_i$-group, and set $V = K_\ell/K_{\ell-1}$, $H = K_{n-1}/K_{\ell-1}$, and $L \cong K/K_{\ell-1}$. Since $T$ is simple and $C_L(V)H/H \unlhd L/H \cong T$, we either have $C_L(V)H = L$ or $C_L(V) \leq H$. In the former case, we must have $\pi(T) \subseteq \pi(C_L(V))$ so that, in particular, an element of order $s$ commutes with an element of order $p_i$ in $L$ for each $s \in \pi(T)$. This implies that $p_i$ is not adjacent to any vertices in $\pi(T)$, so we may assume $C_L(V) \leq H$. Hence $L/C_L(V)$, which acts faithfully on $V$, has a nonabelian quotient:
    \[\frac{L/C_L(V)}{H/C_L(V)} \cong T.\]
    This means $\Aut(V)$ is nonabelian, and it follows that $V$ is noncyclic. In particular, the Sylow $p_i$-subgroups of $K_\ell/K_{j-1}$ are noncyclic. Moreover, these subgroups have odd order by hypothesis, so they cannot act Frobeniusly on $K_j/K_{j-1}$. Thus an order $p_i$ element commutes with an order $p_j$ element in $K_\ell/K_{j-1}$, giving $p_i-p_j \notin \ol \Gamma(G)$ and completing the proof.
\end{proof}

The most useful sections of $G$ for eliminating edges are those of the form $L \cong V.W.T$, where $V$ is a nontrivial elementary abelian $p$-group and $W$ is a solvable $p'$-group. The above lemma shows that such sections exist, and moreover that $W$ satisfies the following condition: If $r\divides |W|$ is an odd prime adjacent to some vertex in $\pi(T)$, then $r-p\notin \ol \Gamma(G)$. However, it is sometimes necessary to place even further restrictions on these sections.

\begin{lemma}\label{self-central}
    Let $T$ be a nonabelian simple group, $V$ an elementary abelian $p$-group, $M$ a solvable $p'$-group, and $L \cong V.M.T$. If $p$ is adjacent to some vertex in $\pi(T)$, then $G$ has a quotient $\ol L \cong V.W.T$ for some quotient $W$ of $M$ with $C_{\ol L}(V) = V$. Furthermore, if $r$ is a prime and there exists an odd prime $s \in \pi(T)$ such that $s-p,s-r \in \ol \Gamma(L)$, then $r \ndivides |W|$.
\end{lemma}

\begin{proof}
    Let $H$ denote the subgroup $V.M \leq L$. Because $T$ is simple and 
    \[C_L(V)H/H \unlhd L/H \cong T,\]
    we have $C_L(V)H = L$ or $C_L(V) \leq H$. In the former case, $\pi(T) \subseteq \pi(C_L(V))$, contradicting the fact that $p$ is adjacent to some vertex in $\pi(T)$. Hence $C_L(V) \leq H$. Furthermore, we can write $H \cong V \rtimes M$ by the Schur-Zassenhaus theorem, and since $C_L(V)$ contains $V$, we have $C_L(V) \cong V \rtimes A$ for some $A \unlhd H$. But since $A$ acts trivially on $V$, this product must be direct. Thus $A$ is a normal Hall $p'$-subgroup of $C_L(V)$, which means it is characteristic in $C_L(V)$, which in turn means it is normal in $G$. Now let $\ol L = L/A$, and observe that $\ol L \cong V.W.T$ where $W = M/A.$ Moreover, $C_{\ol L}(V) = V$.

    Finally, suppose $r \in \pi(T)$ is a prime adjacent to $p$ and some odd prime $s \in \pi(T)$. Refining $\ol L = V.W.T$ into a chief series, we see from \Cref{chief series} that $s \ndivides |W|$. Let $B\leq \ol L$ be a subgroup $V.W.C_r$, and $B_{psr}$ a Hall $\{p,s,r\}$-subgroup of $B$. Since $W$ is an $\{s,p\}'$-group, \cite[Lemma 2.1]{2022} and the Schur-Zassenhaus theorem gives $B_{psr} \cong V \rtimes R\rtimes C_s$ for some Sylow $r$-subgroup $R \leq W$. But since $s-p,s-r \in \ol \Gamma(\ol L)$, any cyclic subgroup of order $s$ acts Frobeniusly on $V \rtimes R$, so $V \rtimes R$ must be nilpotent, i.e., $R$ must centralize $V$. However, $V$ is self-central in $\ol L$, so this implies $R=1$. In other words, $r \ndivides |W|.$
\end{proof}

Let us summarize. We would like to find criteria for eliminating certain edges from $\ol \Gamma(G)$, and we have already hinted that edges of the form $r-p$ for $r \in \pi(T)$ and $p \in \pi(G) \setminus \pi(T)$ will be the main point of consideration. Using the above tools, we can produce sections of $G$ of the form $V.W.T$, where $V$ is a nontrivial self-centralizing elementary abelian $p$-group and $W$ is a $p'$-group satisfying various other conditions regarding its prime divisors. To prove that $r-p \notin \ol \Gamma(G)$, then, it suffices to show that $r-p \notin \ol \Gamma(L)$. We are not quite ready to prove criteria for this being the case -- that will be in the next subsection. First, we must prove two major lemmas. The first cites a powerful result of \cite{Flavell}.

\begin{definition}
    Let $G$ and $H$ be finite groups with $G$ acting via automorphisms on $H$, where the action is written exponentially. Given $g \in G$, an element $h \in H$ is said to be a \textit{fixed point} of $g$ if it is nontrivial and $h^g = h$.
\end{definition}

\begin{lemma}\label{truelizard}
    Let $r$ and $p$ be distinct primes, $R$ a cyclic group of order $r$ acting nontrivially on an $\{r,p\}'$-group $W$ by automorphisms, and $H \cong W \rtimes R$ the associated semidirect product. Let $H$ act faithfully on an elementary abelian $p$-group $V$, and assume that one of the following conditions hold.
    \begin{enumerate}
        \item[(1)] $r$ is not a Fermat prime.
        \item[(2)] $2 \ndivides |W|.$
        \item[(3)] There exists an odd prime $q$ and an $R$-invariant Sylow $q$-subgroup $Q$ of $W$ on which $R$ acts nontrivially.
    \end{enumerate}
    Then, $R$ has a fixed point in $V$.
\end{lemma}

\begin{proof}
    Because $(|H||R|, |V|) = 1,$ we can apply Maschke's theorem to write $V$ as a direct sum of irreducible $H$-modules over $\F_p$: 
    \[V = V_1 \oplus \cdots \oplus V_\ell.\] 
    For $1 \leq i \leq \ell$, let $K_i = C_{H}(V_i)$. Then $K_iW/W \unlhd H/W \cong R$, so either $K_iW = H$ or $K_i \leq W$. In the former case, we have $r \divides |K_i|$, so an element of order $r$ acts trivially on $V_i$. In particular, this element has a fixed point in $V$, as needed. Hence, we may assume that $K_i \leq W$ for each index $i$. Because $V$ is a faithful $H$-module, we have $\bigcap_{i=1}^\ell K_i = 1$, and it follows that
    \[W \cong W/\bigcap_{i=1}^\ell K_i \lesssim \prod_{i=1}^\ell W/K_i,\]
    where the last expression is a direct product. Then, since $R$ acts nontrivially on $W$, it must act nontrivially on some $W/K_i$, say $W/K_1$. Hence $[R, W/K_1] \ne 1$ and \[H/K_1 \cong (W/K_1) \rtimes R\] acts faithfully on $V_1$. 
    
    If (1) or (2) holds, then it follows immediately from \cite[Theorem A]{Flavell} that $C_{V_1}(R) \neq 0,$ and we are done. So assume (3) holds, and let $Q_i = QK_i/K_i$ for each index $i$. Since $R$ acts nontrivially on $Q$ and
    \[Q \cong \frac{Q}{Q \cap \bigcap_{i=1}^\ell K_i} \lesssim \prod_{i=1}^\ell Q_i,\]
    $R$ must act nontrivially on some $Q_i$, say $Q_1.$ Hence $R$ acts nontrivially on $W/K_1$, or equivalently $[R, W/K_1] \ne 1$. In addition, we see that $Q_1 \ne 1$ and since $q\ne 2$, it follows that $W/K_1$ is not a 2-group. Because $H/K_1 \cong (W/K_1) \rtimes R$ acts faithfully on $V_1$, \cite[Theorem A]{Flavell} again implies that $C_{V_1}(R) \ne 0$, completing the proof.
\end{proof}

The next lemma makes use of techniques from representation theory. In the next subsection, when attempting to eliminate edges of the form $r-p$ for $r \in \pi(T)$ and $p \in \pi(G) \setminus \pi(T)$, we must consider the case in which the group $W$ is trivial. In other words, $G$ has a section $L\cong V.T$ for some nontrivial elementary abelian $p$-group $V$. In fact, because $(|V|, |T|) = 1$, the Schur-Zassenhaus theorem gives $L \cong V \rtimes T$, which means $r-p \in \ol \Gamma(L)$ if and only if an order $r$ element of $T$ has fixed points on $V$. Now, as an elementary abelian $p$-group, $V$ may be regarded as an $\mathbb F_pT$-module, and since $p \ndivides |T|$, it is shown in \cite{Curtis+Reiner} (as explained in \cite{Meierfrankenfeld}) that $V$ is the reduction modulo $p$ of some complex representation $\widetilde V$ of $T$. In particular, it is shown 
there that an element of $T$ acts fixed point freely on $V$ if and only if it acts fixed points freely on $\widetilde V$. Hence, to prove $r-p \notin \ol \Gamma(L)$ (and thus $r-p \notin \ol \Gamma(G)$), it suffices to show that in every complex irreducible representation of $T$, there exists an element of order $r$ with fixed points. More generally, any relationship between fixed points of elements of distinct orders in complex irreducible representations of $T$ may be exploited. For example, suppose that in every irreducible $\mathbb CT$-module $\widetilde V$ with an order $r$ element acting with fixed points, there exists an order $s$ element acting with fixed points as well. Then, if it is known that an order $r$ element of $T$ has fixed points in $V$, we can be certain that an order $s$ element does too. We have recorded several facts of this flavor for a variety of different groups in \Cref{table:3} of the Appendix and will make frequent use of them. For now, though, we conclude this subsection by recording the above claims in a lemma.

\begin{lemma}\label{reps}
    Let $L \cong V.T$, where $T$ is any finite group and $V$ is a nontrivial elementary abelian $p$-group for some $p \ndivides |T|$. Then, there exists a $\mathbb{C}T$-module $\widetilde V$ such that for all $r \in \pi(T)$, there exists an order $r$ element of $T$ with fixed points in $\widetilde V$ if and only if $L$ has an element of order $rp$.
\end{lemma}

\subsection{Edges of the form $r-p$, $r \in \pi(T)$, $p \in \pi(G) \setminus \pi(T)$}

We now shift to the main goal of this section: establishing criteria for the nonexistence of edges $r-p$ in the prime graph complement of a $T$-solvable group $G$, where $r \in \pi(T)$ and $p \in \pi(G) \setminus \pi(T)$. Such results will be referred to as \textit{$r-p$ criteria}. Our main tools will be the lemmas from the previous section and some theory of central extensions. 

We fix throughout this section a nonabelian simple group $T$, a strictly $T$-solvable group $G$, and a subgroup $K \cong N.T$ of $G$ where $N$ is solvable. The soon-to-be presented $r-p$ criteria will typically apply for all $p \in \pi(K)\setminus \pi(T)$. Thus if $\pi(K)$ can be chosen to contain most or all of the primes of $\pi(G)$, then these results can be very effective. Fortunately for us, every simple group in \Cref{section:allresults} satisfies $\pi(T) = \pi(\Aut(T))$, so \Cref{2.4 2022} gives us the best case scenario of $\pi(G) = \pi(K)$ in all of our applications. Still, these results are likely still useful when $\pi(T)  \ne \pi(\Aut(T)).$ For example, one could use them alongside \Cref{N.M} or some generalization thereof to study edges $r-p$ when $r \in \pi(T)$ and $p \in \pi(G) \setminus \pi(\Aut(T))$.


The first $r-p$ criterion is relatively simple and often the first technique one tries when attempting to eliminate edges between $\pi(T)$ and $\pi(K) \setminus \pi(T)$. It is most conveniently stated in terms of the following definition.
\begin{definition}
    Let $P$ be a $p$-group. $P$ is said to satisfy the {\it Frobenius Criterion} if $P$ is cyclic, dihedral, Klein-4, or generalized quaternion.
\end{definition}
Note that every quotient of a generalized quaternion or cyclic group satisfies the Frobenius Criterion. 

\begin{proposition}\label{FC}
    Assume the Sylow $r$-subgroups of $T$ do not satisfy the Frobenius criterion for some $r \in \pi(T)$. For all $p\in \pi(K) \setminus \pi(T)$, we have $r-p \notin \ol \Gamma(G)$.
\end{proposition}
\begin{proof}
    By \Cref{chief series}, $G$ has a section $L \cong V.W.T$, where $V$ is a nontrivial elementary abelian $p$-group and $W$ is a $p'$-group. Then $L$ acts on $V$ by conjugation, and since $V$ is abelian, so does $L/V \cong W.T$. Let $S \in \Syl_r(L/V)$ and $\ol S \in \Syl_r(T)$, so that $S/(S \cap W) \cong \ol S$. Since $\ol S$ does not satisfy the Frobenius criterion, $S$ is not cyclic nor generalized quaternion and thus cannot act Frobeniusly on $P$ by \cite[Corollary 6.17]{FGT}. Accordingly, there exist nontrivial elements $s \in S$ of $r$-power order and $x \in V$ of $p$-power order such that $s^{-1}xs = x$. Letting $\widetilde s$ denote a preimage of $s$ in $L$, we see that $\widetilde s^{-1}x \widetilde s = x$ and so some power of $\widetilde s x$ has order $rp$, forcing $r-p \notin \ol \Gamma(G)$.
\end{proof}

Typically, the above lemma is used to rule out edges originating from the vertices 2 and 3. The only exception is $A_7$-solvable groups, as 2 can in fact be adjacent to vertices $\pi(G) \setminus \pi(A_7)$.

Several of the remaining $r-p$ criteria require another, more complicated, definition.
\begin{definition}\label{p-special}
    A prime $r \in \pi(T)$ is called {\it special} if at least one of the following conditions hold.
    \begin{enumerate}
        \item[(a)] $r \ndivides |N|$.
        \item[(b)] $r$ is odd and $r-s \in \ol \Gamma(G)$ for some $s \in \pi(T)$.
    \end{enumerate}
    In this subsection, $G,K,N$ are fixed. In the future, to avoid ambiguity, we will say that $r$ is special relative to the triple $(G,K,N)$.
\end{definition}

\begin{proposition}\label{Thanks TMK}
    Assume that the subgraph of $\ol \Gamma(G)$ induced by $\pi(T)$ has no isolated vertices. If $r \in \pi(T)$ is a special prime such that in every complex irreducible representation of $T$ there exists an element of order $r$ with fixed points, then $r-p \notin \ol \Gamma(G)$ for all $p \in \pi(K) \setminus \pi(T)$.
\end{proposition}

\begin{proof}
    By \Cref{chief series}, $K$ has a quotient $L\cong V.W.T$ where $V$ is a nontrivial elementary abelian $p$-subgroup and $W$ is a solvable $p'$-group. If $r \divides |W|$, then $r \divides |N|$, and because $r$ is special it must be odd and adjacent to some vertex in $\pi(T)$. It then follows from \Cref{chief series} that $r-p \notin \ol \Gamma(G)$, as desired. Now suppose $r \ndivides |W|$. By \Cref{self-central}, we may assume in addition that $V$ is self-central. 
    
    Since $L$ is $(\pi(T) \cup \{p\})$-separable, it has a Hall $(\pi(T) \cup \{p\})$-subgroup $H$. Then, applying \cite[Lemma 2.1]{2022}, we can write $H \cong VM.T$ for some Hall $(\pi(T) \cap \pi(W))$-subgroup $M$ of $W$. If $M$ is trivial, then since $(|V|,|T|)=1,$ \cite[Lemma 2.1.7]{2015} 
    or \cite[Remark 2 following Theorem C]{Meierfrankenfeld} produces elements of order $rp$ in $H$ and it follows that $r-p\notin \ol\Gamma(G),$ as desired. We may assume then that $M$ is nontrivial, so that it contains a minimal normal subgroup $M_0$ of $H/V \cong M.T$. Because $M$ is solvable, $M_0$ is a nontrivial elementary abelian $s$-group for some $s \in \pi(T) \cap \pi(W)$. Now let $R \leq H$ be a cyclic subgroup of order $r$ and observe that $R \cong RV/V$ acts coprimely on $M_0 \cong M_0V/V$ by conjugation, so there is the Zassenhaus decomposition $M_0 = [R,M_0] \times C_{M_0}(R).$ In particular, $R$ acts nontrivially on $[R, M_0]$, and since $R$ is cyclic of prime order, this action must in fact be Frobenius. 
    
    Assume, for a contradiction, that $r-p \in \ol \Gamma(L)$. Then $R$ acts Frobeniusly on $V$ and thus also acts Frobeniusly on $V[R, M_0].$ Therefore, \cite[Theorem 6.24]{FGT} shows that $V[R, M_0]$ is nilpotent, and it follows that $V[R, M_0] = V \times [R, M_0].$ In other words, $[R, M_0]$ centralizes $V$, and the fact that $C_L(V) = V$ forces $[R, M_0] = 1.$ Therefore, $R$ acts trivially on $M_0$, and it follows that $r \divides |C_H(M_0)|$. Now, since 
    \[\frac{C_H(M_0) VM}{VM} \unlhd \frac{H}{VM }\cong T,\]
    we must have $C_H(M_0)VM = H$ or $C_H(M_0) \leq VM.$ However, the latter case is impossible since $r \divides |C_H(M_0)|$ and $r \ndivides |VM|,$ so we must have $C_H(M_0)VM = H$. Then for all $t \in \pi(T)$, it follows from order considerations that $t \divides |C_H(M_0)|.$ This means an element of order $t$ centralizes an element of order $s$ in $M_0$, so the edge $t-s$ does not appear in $\ol \Gamma(H)$ or in $\ol \Gamma(G)$. However, since $s \in \pi(W) \cap \pi(T)$, this contradicts the hypothesis that the subgraph of $\ol \Gamma(G)$ induced by $\pi(T)$ has no isolated vertices. Therefore, $r-p \notin \ol \Gamma(L)$.
\end{proof}

The astute reader will notice that the above proposition can be generalized. Instead of requiring that the subgraph induced by $\pi(T)$ have no isolated vertices, it is enough to assume that only the vertices in $\pi(T) \cap \pi(W)$ are not isolated. This is of course a more complicated condition to establish, likely requiring an application of the prime divisibility conditions from \Cref{chief series} or \Cref{self-central}. But it is certainly doable and possibly applicable to the study of certain $K_4$ groups. However, we do not come across any applications in this paper, so we have chosen to state the more convenient version of the criterion instead.

The next $r-p$ criterion is a corollary to the following proposition, which generalizes \cite[Lemma 7.4]{2022}.

\begin{proposition}\label{funny}
    For each $p \in \pi(K) \setminus \pi(T)$, one of the following holds:

    \begin{enumerate}
        \item[(1)] For each odd prime $r \in \pi(T)$ satisfying $(r, |M(T)|) = 1$ (where $M(T)$ is the Schur multiplier of $T$), we have $r-p \notin \ol \Gamma(G).$
        \item[(2)] $G$ has a section $V.E$, where $V$ is a nontrivial elementary abelian $p$-group and $E$ is a perfect central extension of $T$.
    \end{enumerate}
\end{proposition}

\begin{proof}
    By \Cref{chief series}, it suffices to prove that if $W$ is a solvable $p'$-group, $H \cong W.T$, and $V$ is an $H$-module over the finite field $\mathbb F_p$, then either $r-p \notin \ol \Gamma(VH)$ or $H$ has a perfect central extension of $T$ as a subgroup. We do so by contradiction, letting $H \cong W.T$ and $V$ be a counterexample such that $|H|+|V|$ is minimal. Clearly, $V$ is faithful and irreducible.

    First, consider the case $r \divides |W|$. Let $M_1 \geq M_2$ be normal subgroups of $H$ contained in $W$ such that $G/M_1$ is an $r'$-group and $M_1/M_2$ is an elementary abelian $r$-group. Since we may assume that $r-p \in \ol \Gamma(VH)$, the Sylow $r$-subgroups of $H$ must act Frobeniusly on $V$ and accordingly must be cyclic. Hence $M_1/M_2$ is cyclic, and it follows that $M_1/M_2 \cong C_r$. In particular, $|(H/M_1)/C_{H/M_1}(M_1/M_2)|=|H/C_H(M_1/M_2)|$ divides $r-1$. Moreover, because $T$ is simple, we either have $C_H(M_1/M_2) \leq W$ or $C_H(M_1/M_2)W = H$. Now, letting $r^k$ denote the largest power of $r$ dividing $|T|$, we see that $r^k$ also divides $|C_H(M_1/M_2)|$, forcing $C_H(M_1/M_2)W = H$. Therefore, if $C_H(M_1/M_2) < H,$ then since 
    \[\frac{C_H(M_1/M_2)}{W \cap C_H(M_1/M_2)} \cong \frac{C_H(M_1/M_2)W}{W} \cong T,\] 
    the group $C_H(M_1/M_2)$ acting on $V$ constitutes a counterexample to the lemma with order $|C_H(M_1/M_2)| + |V| < |G| + |V|$. This contradicts our choice of $G$, so we can assume $C_H(M_1/M_2) = G$, i.e., $M_1/M_2$ is a central chief factor of $H$. In particular, a central extension of $T$ by $C_r$ appears as a quotient of $H$, but since $(r, |M(T)|) = 1$, this extension must be split, contradicting the fact that $G$ has cyclic Sylow $r$-subgroups.

    We have reduced to the case $r \ndivides |W|$. Now let us prove that every abelian normal subgroup $M \unlhd H$ with $M \leq W$ is central and cyclic. Let $R$ be a subgroup of order $r$ in $H$. Then the Zassenhaus decomposition for coprime actions tells us that
    $M = [R, M] × C_M(R)$, and $R$ acts Frobeniusly on $[R, M]$. Hence $R$ acts Frobeniusly on $V[R, M]$ which thus is nilpotent. Therefore $[R, M]$ acts trivially on $V$, but since the action of $H$ on $V$ is faithful, this forces $[R, M] = 1$. Since $R$ was chosen arbitrarily of order $r$, this shows that $C_H(M)$ contains all elements of order $r$ in $H$, and since $r$ divides $|T|$ but not $|W|$, we see that $C_H(M)$ has a factor group isomorphic to $T$. So by minimality, it follows that $C_H(M) = H$ and hence $M \leq Z(H)$. As a consequence, $M$ is cyclic.

    If an element of order $r$ centralizes $F(H) = F(W)$, then $C_H(F(W))W/W$ is a nontrivial normal subgroup of $T$, hence $C_H(F(W))$ has a factor group isomorphic to $T$. By the minimality of $H$ and $V$, we must have $H = C_H(F(W))$. Now Gaschutz's theorem implies that $W=F(W)=Z(G)$. In other words, $H$ is a central extension of $T$. Because $H'$ is also a central extension of $T$, minimality implies $H = H'$. Hence $H$ is a perfect central extension of $T$, and it follows from \Cref{reps} that $r-p \notin \ol \Gamma(VH)$, a contradiction. Therefore, no order $r$ element centralizes $F(H)$.

    Let $R_1$ be a Sylow $r$-subgroup of $H$ containing a subgroup $R_0 \cong C_r$ which, by the above, does not centralize $F(H)$. Then $R_0$ acts faithfully on some Sylow subgroup $Q$ of $F(H)$, and from the Zassenhaus decomposition, we know that $Q = [R_0, Q]C_Q(R_0)$ and $[R_0, Q] >1.$ Clearly $R_1$ normalizes $[R_0, Q]$, and since $R_1$ is cyclic of $r$-power order, $R_1$ acts faithfully on $[R_0, Q]$. So altogether we have that $R_1$ acts faithfully on $[R_0, Q]$, and $R_1[R_0, Q]$ acts faithfully on $V$, and $R_1$ acts Frobeniusly on $V$. This is not possible by \cite[Theorem 3.1 and Corollary 3.2]{atlas}, and this final contradiction completes the proof of the proposition.
\end{proof}

\begin{corollary}\label{funny2}
    Let $r \in \pi(T)$ be an odd prime such that $(r, |M(T)|) = 1$. Suppose that in every complex irreducible representation of a perfect central extension of $T$, some element of order $r$ has fixed points. Then $r-p \notin \ol \Gamma(G)$ for all $p \in \pi(K) \setminus \pi(T)$.
\end{corollary}

With that, we have established the only $r-p$ criteria to be used in this paper. Indeed, it just so happens that the primes dividing $|M(T)|$ tend to be the same primes whose Sylow subgroups do not satisfy the Frobenius criterion, and \Cref{Thanks TMK} turns out to be quite good at handling edge cases. Nevertheless, we present two more criteria. These can be useful for future research on prime graphs of finite groups, as there are theoretical situations where the above criteria do not apply but the upcoming ones do.

The first of these additional $r-p$ criteria supposes that $T$ contains a subgroup of convenient form.

\begin{proposition}\label{qr in T}
	Let $r\ne 2$ and $s$ be distinct primes and suppose $T$ contains a subgroup $S \rtimes R$ in which a cyclic $r$-subgroup $R$ acts nontrivially on an $s$-subgroup $S$. Further assume that $r$ is special and one of the following holds:
	\begin{enumerate}
		\item[(1)] $r$ is not a Fermat prime.
		\item[(2)] $s \ne 2.$
	\end{enumerate}
	Then $r-p \notin \ol \Gamma(G)$ for all $p \in \pi(N)$ (in particular, for all $p \in \pi(K) \setminus \pi(T)$).
\end{proposition}

\begin{proof}
	Let $p \in \pi(N)$. By \Cref{chief series}, $K$ has a quotient $L\cong V.W.T$ where $V$ is a nontrivial elementary abelian $p$-subgroup and $W$ is a $p'$-group. If $r \divides |W|$, then in particular $r$ must divide $|N|$, so the definition of special implies that $r$ is odd and adjacent to some vertex in $\pi(T)$. Thus the last statement of \Cref{chief series} implies that $r-p \notin \ol \Gamma(G)$. Assume then that $r \ndivides |W|$. By \Cref{self-central}, we may assume in addition that $V$ is self-central. Now consider the subgroup $L \cong V.W.S.R$. Letting $H$ denote the section $W.S$, we see that because $R$ acts nontrivially on $S$, it must also act nontrivially on $H$. In fact, since the action is coprime, we can take an $R$-invariant Sylow $s$-subgroup $\widetilde S \leq H$, and note that since $R$ acts nontrivially on the quotient $\widetilde S/\widetilde S \cap W\cong S$, it also acts nontrivially on $\widetilde S$. Then, since $r$ is non-Fermat or $s \ne 2$, \Cref{truelizard} shows that $r-p \notin \ol \Gamma(G)$. 
\end{proof}

The setup of the above lemma is very convenient, as it puts us in the position to apply \Cref{truelizard} almost immediately. In most situations, however, we do not have the luxury of assuming that a cyclic group of order $r$ acts nontrivially on the intermediate subgroup $H$. Thus we also consider trivial actions using the theory of central extensions, yielding a result reminiscent of \Cref{funny}.

\begin{proposition}\label{perfect central ext}
	For each $p \in \pi(K) \setminus \pi(T)$, one of the following holds: 
	\begin{enumerate}
		\item[(1)] For all special primes $r \in \pi(T)$, if $r$ is non-Fermat or adjacent to 2, then $r-p \notin \ol \Gamma(G)$.
		\item[(2)] $G$ has a section $V.E$, where $V$ is a nontrivial elementary abelian $p$-group and $E$ is a perfect central extension of $T$.
	\end{enumerate}
\end{proposition}

\begin{proof}
	By considering a chief series for $K$ and applying \Cref{self-central}, we can find a section $L \cong V.W.T$ where $V$ is a self-centralizing elementary abelian $p$-group and $W$ is a solvable $p'$-group (and a quotient of $N$). Also, the prime divisors of $|W|$ are limited by the last statement of \Cref{self-central}, but this will not be important until later. Our first goal is to reduce to the case that $L$ has a subgroup $L_0 \cong V.M.Q.E$, where $E$ is a perfect central extension of $T$, $Q$ is an abelian $\pi(W)$-group on which $E$ acts nontrivially, and $M$ is a solvable $\pi(W)$-group. 
	
	If $W$ is trivial, then (2) holds, so take $W \ne 1$. By considering a chief series of $L$, we can write $L = V.Q_1.Q_{2}.\cdots. Q_{n-1}.Q_n.T$ where each $Q_i$ is an elementary abelian $\pi(W)$-group (the chief factors of $V$ are not important). For $1\leq i \leq n$, let $L_i$ be the quotient $L_i = Q_i.Q_{i+1}. \cdots. Q_n.T$. If $L_n/Q_n \cong T$ acts nontrivially on $Q_n$, then the reduction is complete because $T$ is perfect. Otherwise, $L_n$ is a central extension of $T$, and it follows from \cite[(33.3)]{aschbacher} that $L_n'$ is perfect and 
	\[L_n'/Q_n \cap L_n' \cong L_n'Q_n/Q_n = L_n/Q_n \cong T.\]
	Consider the subgroup $V.Q_1.Q_2. \cdots. Q_{n-1}. L_n'\leq L$. If $L_n'$ acts nontrivially on $Q_{n-1}$, then the reduction is complete. Otherwise, $Q_{n-1}.L_n'$ is a central extension of $L_n'$, and the above reasoning shows that $(Q_{n-1}.L_n')'$ is a perfect central extension of $L_n'$. In fact, because $L_n'$ is a central extension of $T$, we see from \cite[(33.5)]{aschbacher} that $(Q_{n-1}.L_n')'$ is a perfect central extension of $T$. Now, we can examine the action of $(Q_{n-1}.L_n')'$ on $Q_{n-2}$ and apply the same argument as before. By induction, we can reduce to a subgroup of one of the following forms: $V.E$ for $E$ a perfect central extension of $T$, or $V.Q_1.Q_2.\cdots.Q_m.E$, where $E$ is a perfect central extension of $T$ acting nontrivially on $Q_m.$ In the former case, (2) holds. In the latter case, we have reduced to the situation described in the first paragraph.
	
	With the reduction complete, we show that (1) holds. Let $r \in \pi(T)$ be special and either non-Fermat or adjacent to 2. If $r \divides |W|$, then \Cref{chief series} and the definition of special implies that $r -p \notin \ol \Gamma(G)$, as desired. Now suppose $r \ndivides |W|$. By the above, $L$ has a subgroup $L_0 = V.M.Q.E$, where $E$ is a perfect central extension of $T$ by a group $A$, $Q$ is an abelian $\{r,p\}'$-group on which $E$ acts nontrivially, and $M$ is a solvable $\{r,p\}'$-group. Now, let $R$ be a subgroup of order $r$ in $E$ and consider the subgroup $(V.M.Q) \rtimes R$. We have $C_E(Q) A/A \unlhd E/A \cong T$, so either $C_E(Q)A = E$ or $C_E(Q)A = A.$ Suppose, for sake of contradiction, that $C_E(Q)A = E.$ Since $E$ is perfect and $A$ is central in $E$, we have
	\[E = E' = (C_E(Q)A)' = C_E(Q)'  \leq C_E(Q).\]
	Hence $C_E(Q) = E,$ contradicting the fact that $E$ acts nontrivially on $Q$. This proves $C_E(Q)A \ne E$, and it follows that $C_E(Q)A = A$. Then, since $r \ndivides |A|$, we see that $r\ndivides |C_E(Q)|$ as well. Consequently, $R$ is not a subgroup of $C_E(Q)$, which means $R$ must act nontrivially on $Q$. Letting $H = M.Q$, we see that $R$ also acts nontrivially on $H$. If $r$ is not Fermat, then since $V$ is self-central, we are done by \Cref{truelizard}. Assuming then that $r$ is adjacent to 2, we have two cases. If $2 \divides |W|$, then the terms of \Cref{self-central} and the fact that $2-r \in \ol \Gamma(G)$ (in fact $2-r \in \ol \Gamma(L_0)$ since 2 divides the order of every nonabelian simple group) force $r-p \notin \ol \Gamma(G).$  On the other hand, if $2 \ndivides |W|$, then $2 \ndivides |M||Q|$ and again we are done by \Cref{truelizard}.
\end{proof}

The most common special case of the above result is the following corollary.

\begin{corollary}\label{truestlizard}
	Let $r \in \pi(T)$ be special and assume that in every complex irreducible representation of every perfect central extension of $T$, some element of order $r$ has fixed points. If $r$ is non-Fermat or adjacent to 2, then $r-p \notin \ol \Gamma(G)$ for all $p \in \pi(K) \setminus \pi(T)$. 
\end{corollary}

\subsection{Orientations on $\ol \Gamma(G)$}

To help classify the prime graphs of solvable groups, the authors of \cite{2015} defined an orientation on the prime graph of a solvable group $G$ called the \textit{Frobenius digraph} of $G$. It was shown that the Frobenius digraph of a solvable group $G$ has no $3$-paths and then deduced from the Gallai-Hasse-Roy-Vitaver Theorem that $\ol \Gamma(G)$ is $3$-colorable. In \cite{2022}, a similar technique was used to show that prime graph complements of $K_3$-solvable groups are 3-colorable. In the next section, we will do the same for select $K_4$ groups (note that $\ol \Gamma(\Sz(8))$ and $\ol \Gamma(\PSL(3,4))$ are the complete graphs on four vertices and thus not 3-colorable). In this subsection, we prove important preliminary results that apply to all simple groups $T$ satisfing $\pi(T) = \pi(\Aut(T))$. Then, we use a similar philosophy to tackle the problem of constructing a group whose prime graph is isomorphic to a given graph. 

First, we quote the relevant definitions from \cite{2022}.
\begin{definition}
    A group $G = QP$ is {\it Frobenius of type $(p,q)$} if it is a Frobenius group where the Frobenius complement $P$ is a $p$-group and the Frobenius kernel $q$ is a $q$-group.
\end{definition}
\begin{definition}
    A group $H = P_1QP_2$ is called {\it $2$-Frobenius of type $(p, q, p)$} if it is a $2$-Frobenius group where the subgroup $P_1Q$ is Frobenius of type $(q, p)$ and the quotient group $QP_2$ is Frobenius of type $(p, q)$.
\end{definition}
By \cite[Theorem A]{1981}, if $G$ is solvable and $p-q \in \ol \Gamma(G)$, then the Hall $\{p,q\}$-subgroups of $G$ are either Frobenius or $2$-Frobenius. This crucial fact led naturally to the definition of the Frobenius digraph stated below.
\begin{definition}\label{frobenius digraph}
    Let $G$ be solvable. The {\it Frobenius digraph of G}, denoted $\vv \Gamma(G)$, is the orientation of $\ol \Gamma(G)$ where $p \to q$ if the Hall $\{p,q\}$-subgroup of $G$ is Frobenius of type $(p,q)$ or $2$-Frobenius of type $(p,q,p)$.
\end{definition}

The following observation lets us extend this orientation to the prime graph complements of $T$-solvable groups. If $T$ is a nonabelian simple group with $\pi(T) = \pi(\Aut(T))$ and $G$ is strictly $T$-solvable, then let $K \cong N.T$ be the subgroup granted by \Cref{2.4 2022} and note that $\vv \Gamma(N)$ induces an orientation on $\ol \Gamma(G) \setminus \pi(T)$ via the containment $\ol \Gamma(G) \setminus \pi(T) \subseteq \ol \Gamma(N)$. This orientation depends on $K$, and there could be a multitude of ways to extend this orientation to $\ol \Gamma(G)$. Regardless of our choices, we see that any 3-path in such an orientation of $\ol \Gamma(G)$ must include at least one vertex from $\pi(T)$. In fact, we can say more: by the next lemma, a 3-path whose initial vertex belongs to $\pi(T)$ must include at least one other vertex of $\pi(T)$. With that, one should hope that the remaining cases are disallowed by a carefully chosen orientation of the edges incident to $\pi(T)$, for then it would follow from the Gallai-Hasse-Roy-Vitaver Theorem that $\ol \Gamma(G)$ is 3-colorable. This is indeed possible for $T=A_7$ (see \Cref{section:A7}).
\begin{lemma}\label{no 2-paths}
    Let $K \cong N.T$, where $N$ is a solvable group and $T$ is a nonabelian simple group. Let $r$ be a prime dividing the order of $T$. If $p,q \in \pi(N)$, $p$ is odd, and $r-p, p-q \in \ol\Gamma(K)$, then $q \to p \in \vv\Gamma(N)$.
\end{lemma}

\begin{proof}
    By \Cref{chief series}, $K$ has a section $L = V.W.T$, where $V$ is an elementary abelian $p$-group and $W$ is a $p'$-group. Let $H = V.W = V \rtimes W \leq L$, $C = C_L(V)$, and note that $CH/H \unlhd L/H \cong T$. Because $T$ is simple, it follows that either $CH = L$ or $C \leq H$. Assume, for a contradiction, that $CH = L$. Since $L$ has $T$ as a factor group, we have $r \divides |L|$. Then, order considerations force $r \divides |C|$, that is, some order $r$ element of $L$ commutes with an order $p$ element of $L$, eliminating the $r-p$ edge from $\ol \Gamma(L)$ and hence from $\ol \Gamma(K)$ as well. This contradicts our assumption that $r-p \in \ol \Gamma(K)$, so we must have $C \leq H$. 
    
    Next, observe that $L/C$ acts faithfully on $V$ and thus embeds into the automorphism group $\Aut(V)$. Moreover, $L/C$ has a nonabelian quotient 
    \[\frac{L/C}{H/C} \cong L/H \cong T,\] 
    so $L/C$ and hence $\Aut(V)$ must be nonabelian. In particular, $V$ must be noncyclic. Now, because $N$ is solvable, it has a Hall $\{p,q\}$ subgroup $H_{pq}.$ Since $p-q \in \ol \Gamma(G)$, we must also have $p-q \in \ol \Gamma(N)$, so $H_{pq}$ is Frobenius of type $(p,q)$ or $(q,p)$, or 2-Frobenius of type $(p,q,p)$ or $(q,p,q)$. But if $S$ is a Sylow $p$-subgroup of $N$, then $S$ has the noncyclic group $V$ as a quotient. Thus $S$ is noncyclic and hence (as $p$ is odd) cannot act Frobeniusly on any other groups. This means $H_{pq}$ is Frobenius of type $(q,p)$ or 2-Frobenius of type $(q,p,q)$ which means that 
    $q \to p \in \vv\Gamma(N)$, as desired.
\end{proof}

The next lemma can often be applied after our $r-p$ criteria to establish 3-colorability of a prime graph complement.

\begin{lemma}\label{easy3color}
Let $G$ be strictly $T$-solvable, where $T$ is a nonabelian simple group satisfying $\pi(T) = \pi(\Aut(T))$. Suppose the $\pi(T)$-subgraph of $\ol\Gamma(G)$ does not connect to the remainder of the graph or only connects to the remainder of the graph through a single vertex $r\in \pi(T)$. Then every triangle in $\ol \Gamma(G)$ is contained in $\ol \Gamma(T)$. Furthermore, if the $\pi(T)$-subgraph is $3$-colorable, then $\ol\Gamma(G)$ has a $3$-coloring for which all vertices adjacent to $r$ not in $\pi(T)$ have the same color. 
\end{lemma}

\begin{proof}
If $N(\pi(T)) = \pi(T)$, then throughout the proof we take $r$ to be any element of $\pi(T)$. By \Cref{2.4 2022}, $G$ has a subgroup $K\cong N.T$ with $N$ solvable and $\pi(G)=\pi(K).$ Since $T$ has an order $r$ subgroup isomorphic to $C_{r}$, $K$ has a solvable subgroup isomorphic to $N.C_{r}$. The fact that $\pi(K)=\pi(G)$ implies $(\pi(G)\setminus\pi(T))\cup\{r\}) \subseteq \pi(N.C_{r})$, so all edges of $\ol\Gamma(G)$ besides possibly those in the $\pi(T)$-subgraph are contained in $\ol\Gamma(N.C_{r}).$ In addition, \Cref{ThomasMichaelKeller} implies that $\ol \Gamma(N.C_r)$ is triangle-free. Since vertices in $\pi(T)\setminus \{r\}$ are not adjacent to vertices in $\pi(G) \setminus \pi(T)$, it follows that every triangle in $\ol \Gamma(G)$ is contained in the subgraph induced by $\pi(T)$, which of course is contained in $\ol \Gamma(T)$.

Let $\vv\Gamma(N.C_{r})$ denote the Frobenius digraph (see \Cref{frobenius digraph}) of $N.C_{r},$ which by \cite[Corollary 2.7]{2015} has no directed $3$-paths. We orient the edges of $\ol\Gamma(G)\setminus \pi(T)$ according to 
their orientation in $\vv\Gamma(N.C_{r})$, and we denote this partial orientation of $\ol\Gamma(G)$ as $\vv\Gamma(G)$. Since $\vv\Gamma(G)$ has no directed $3$-paths, we can define a 3-coloring on $\ol \Gamma(G)$ as follows. Label all vertices in $\pi(G)\setminus \pi(T)$ with zero out-degree with $\mathcal{I},$ all vertices in $\pi(G)\setminus \pi(T)$ with zero in-degree and nonzero out-degree with $\mathcal{O},$ and all vertices in $\pi(G)\setminus \pi(T)$ with non-zero  in- and out-degree with $\mathcal{D}.$  Additionally, since the $\pi(T)$-subgraph is 3-colorable and only possibly connects to the remainder of the graph through the vertex $r$, we can easily extend this $3$-coloring to $\ol\Gamma(G).$ It follows from \Cref{no 2-paths} that for primes $p$ such that $r-p\in \ol\Gamma(G),$ $p$ has no outgoing edges in $\vv\Gamma(G).$ Thus all vertices in $N(r) \setminus \pi(T)$ have the same color, as desired.
\end{proof}

For a fixed nonabelian simple group $T$, one step towards a classification of the prime graphs of $T$-solvable groups is showing that the prime graphs of $T$-solvable groups all satisfy some set of conditions, one of which is typically 3-colorability. In this section so far, we have focused all of our attention on this step of the process, showing how relatively easy-to-access information such as fixed points in representations gives rise to restrictions on the structure of $\ol \Gamma(G)$. Now, we briefly turn our attention to the converse. Papers such as \cite{2022} utilized a variety of methods for this problem, but they all boiled down to the same argument. For the final lemma of this section, we state this argument in full generality. The proof is mostly quoted from \cite{2022}.

\begin{lemma}\label{graphlizard}
    Let $T$ be a nonabelian simple group, $\Gamma$ a graph, and $X$ a set of $|\pi(T)|$ vertices of $\Gamma$. Suppose $\ol \Gamma \setminus X$ is triangle-free and $\ol \Gamma$ has a 3-coloring $\{\mathcal{O}, \mathcal{D}, \mathcal{I}\}$ such that vertices in $N(X) \setminus X$ are all colored $\mathcal I$. Further suppose that $T$ has an extension $E$ by a solvable group such that $\pi(E) = \pi(T)$ and
    \begin{enumerate}
        \item[(1)] There exists a graph isomorphism $\varphi$ from the subgraph of $\ol \Gamma$ induced by $X$ to $\ol \Gamma(E)$. 
        \item[(2)] Given $v \in \ol \Gamma \setminus X$, there exists a complex irreducible representation $V$ of $E$ such that for all $x \in X$, $x-v \in \ol \Gamma$ if and only if elements of order $\varphi(x)$ of $E$ act without fixed points in $V$.
    \end{enumerate}
    Then $\Gamma$ is the prime graph of a $T$-solvable group.
\end{lemma}

\begin{proof}
We define a partial orientation of $\ol \Gamma$ as follows. In $\ol \Gamma \setminus X$, direct edges according to color: $\mathcal{O}\rightarrow\mathcal{D},$ $\mathcal{O}\rightarrow\mathcal{I},$ and $\mathcal{D}\rightarrow\mathcal{I}.$ For all edges $u-v\in \ol\Gamma$ where $u \in X$ and $v \in \ol \Gamma \setminus X$, define the orientation as $u\rightarrow v.$ Now let $n_o=|\mathcal{O}|$,  $n_i=|\mathcal{I}|$, and $n_d=|\mathcal{D}|.$ 

Choose $n_o$ distinct primes $p_1, \ldots, p_{n_o}$ not in $\pi(T).$ Define $p$ as the product of these primes. Using Dirichlet's theorem on arithmetic progressions, pick a set of distinct primes $q_1,\ldots, q_{n_d}$ other than those in $\pi(T)$ such that $q_i\equiv 1 \pmod{p_i}$ for all $i.$ Identify each vertex in $\mathcal{O}$ with one of the $p_i,$ and identify each vertex in $\mathcal{D}$ with one of the $q_i$. Define groups 
\[P=C_{p_1}\times \cdots \times C_{p_{n_o}} \quad \text{and} \quad Q=C_{q_1}\times \cdots \times C_{q_{n_d}}.\]
For all indices $i,j,$ if $p_j\rightarrow q_i$ is an edge in $\vv\Gamma,$ then let $C_{p_j}$ act Frobeniusly on $C_{q_i}$. This is possible because $q_i\equiv 1 \pmod{p_i}.$ Otherwise, if $p_j$ and $q_i$ are not adjacent in $\ol\Gamma,$ let $C_{p_j}$ act trivially on $C_{q_j}.$ This defines a group action of $P$ on $Q,$ so we obtain the induced semidirect product $K=Q\rtimes P$. Note that $K$ is solvable.

Now let $v_1,\ldots, v_n$ be the vertices in $\mathcal{I}.$ For each $k\in \{1,\ldots, n\}$, let $N^1(v_k),$ $N^2(v_k)$ denote the set of primes in $\ol\Gamma\setminus X$ with in-distance $1$ and $2$ to $v_k,$ respectively. If $N^1(v_k)$ is nonempty, let $B_k$ be the Hall $(N^1(v_k)\cup N^2(v_k))$-subgroup of $K$. By our definition of $K$, $\Fit(B_k)$ is a Hall $N^1(v_k)$-subgroup of $K.$ If $N^1(v_k)$ is empty, set $B_k=1.$ Now we divide into two cases for each $k$.

If $v_k$ is not adjacent to any vertex in $X$, consider the trivial complex representation of $T$ and pick a prime $r_k$ such that $|T\times B_k|\divides (r_k-1).$ By \cite[Lemma 3.5]{2015} there exists a modular representation $R_k$ of $T\times B_k$ over a finite field of characteristic $r_k$ such that $\Fit(B_k)$ acts Frobeniusly on $R_k$ and $T$ acts trivially on $R_k.$ 

If $v_k$ is adjacent to some vertex in $X$, let $V$ be the associated irreducible representation granted by the hypothesis. Applying Dirichlet's theorem on arithmetic progressions, take a prime $r_k$ such that $|T\times B_k|\divides (r_k-1).$ According to \cite[Lemma 3.5]{A5}, there exists a representation $R_k$ of $T\times B_k$ over a finite field of characteristic $r_k$ such that $\Fit(B_k)$ acts on $R_k$, and elements of $E$ act without fixed point in $R_k$ if and only if they act without fixed points in $V$. In other words, for all $x\in X$, we have $x-p \in \ol \Gamma(G)$ if and only if order $\varphi(x)$ elements of $E$ act without fixed points on $R_k.$ 

Let $J=R_1\times \cdots \times R_{n_i}.$ Note that by Dirichlet's theorem on arithmetic progressions, we can require each of the $r_k$ to be distinct from the $p_j$ and $q_i$ and additionally be mutually distinct. Thus, we have defined an action of each $T\times B_k$ on $R_k.$ This induces an action of $T\times K$ on $J,$ which induces the $T$-solvable semidirect product $G=J\rtimes(T\times K).$ By this construction, $\Gamma(G)\cong \Gamma$.
\end{proof}

\section{Classification results}\label{section:allresults}
We are finally ready to classify the prime graphs of $T$-solvable groups, where $T$ is any $K_4$ group except for $\Sz(8)$, $\Sz(32),$ and $\PSL(2,q)$. Some groups have similar or identical classification results, so we organize the subsections based on these similarities. Other groups, such as $A_7$, are unique and will be given their own subsections.

All simple groups $T$ studied in this section have the property $\pi(T) = \pi(\Aut(T))$. Therefore, when applying results from \Cref{section:preliminaries}, we implicitly use \Cref{2.4 2022} to produce a subgroup $K \leq G$ of the form $N.T$ where $N$ is solvable and $\pi(G) = \pi(K)$. In particular, quantified statements such as ''for each $p \in \pi(K) \setminus \pi(T)$" in \Cref{section:preliminaries} may instead be read as ''for each $p \in \pi(G) \setminus \pi(T)$".

\subsection{The alternating group $A_7$}\label{section:A7}

\indent In this subsection, we fully classify the prime graphs of $A_7$-solvable groups using the techniques developed in \Cref{section:preliminaries}. This is by far the most difficult classification result in this section. Throughout, it will be useful to remember that $\pi(A_7) = \{2,3,5,7\}$ and that $\ol \Gamma(A_7)$ is the complete graph on four vertices minus the edge $2-3$. 

The first lemma places further restrictions on the subgroup $K \cong N.A_7$ granted by \Cref{2.4 2022}. 

\begin{lemma}\label{W}
    Let $G$ be a strictly $A_7$-solvable group with a subgroup $K\cong N.A_7$ such that $5$ or $7$ divides $|N|$ and $\pi(G) = \pi(K)$. Then $\ol\Gamma(G)$ is $3$-colorable and triangle-free.
\end{lemma}
\begin{proof}
    Note that $A_7$ has subgroups isomorphic to $\PSL(2,7)$ and $A_6$ with prime divisors $\pi(\PSL(2,7)) = \{2,3,7\}$ and $\pi(A_6) = \{2,3,5\}$ respectively. Thus $G$ has subgroups $K_1 \cong N.\PSL(2,7)$ and $K_2 \cong N.A_6$. We now have two cases to consider.
    
    If $5$ divides $|N|,$ then $\pi(G) = \pi(K_1)$ and so $\ol{\Gamma}(G)$ is a subgraph of $\ol{\Gamma}(K_1)$. Since $K_1$ is $\PSL(2,7)$-solvable, \cite[Theorem 3.1 and Corollary 3.4]{2022} shows that $\ol{\Gamma}(K_1)$ is 3-colorable and has at most one triangle, which must be about vertices $\{2,3,7\}$ if it exists. However, $A_7$ is a section of $G$ with an element of order 6, so $2-3 \notin \ol\Gamma(G)$. Hence $\ol \Gamma(G)$ is triangle-free and 3-colorable.
    
    If $7$ divides $|N|,$ then $\ol{\Gamma}(G)$ is a subgraph of $\ol{\Gamma}(K_2).$ Since $K_2$ is $A_6$-solvable, by \cite[Corollary 4.3]{2022}, $\ol{\Gamma}(K_2)$ is 3-colorable and can only exhibit a triangle at the vertices $\{2,3,5\}.$ Again, since $2-3 \notin \ol \Gamma(G)$, it follows that $\ol{\Gamma}(G)$ is triangle-free and 3-colorable.
\end{proof}

The prime graph complements of solvable groups are always triangle-free, and in fact this feature is shared by $T$-solvable groups for certain values of $T$ (see the next subsection). However, this is certainly not the case for $A_7$-solvable groups.  Triangles on the vertices $\{2,5,7\}$ or $\{3,5,7\}$ are clearly possible, as they both appear in $\ol \Gamma(A_7)$. The next lemma shows that $\{2,5,7\}$ are $\{3,5,7\}$ are in fact the only possible triangles in $\ol \Gamma(G)$ for an $A_7$-solvable group $G$.

\begin{lemma}\label{no 5-7-p}
    Let $G$ be strictly $A_7$-solvable. Then $\ol \Gamma(G)$ has at most two triangles, which have vertices $\{2,5, 7\}$ or $\{3, 5, 7\}$ if they exist.
\end{lemma}

\begin{proof}
    First, we show that a triangle in $\ol \Gamma(G)$ must include the vertices 5 and 7. Assume $\ol\Gamma(G)$ has a triangle $\{p,q,r\}$ for some $p,q,r\in \pi(G).$ By \Cref{2.4 2022}, $G$ has a subgroup $K \cong N.A_7$ where $N$ is solvable and $\pi(G) = \pi(K)$. Since $A_6, \PSL(2,7) \lesssim A_7,$ $G$ also has subgroups $K_1 \cong N.\PSL(2,7)$ and $K_2 \cong N.A_6$. Suppose, for a contradiction, that $5\notin \{p,q,r\}$ or $7\notin \{p,q,r\}$. If $5 \notin \{p,q,r\}$, then $\ol\Gamma(K_1)$ has a $\{p,q,r\}$-triangle and $\{p,q,r\}=\{2,3,7\}$ by \cite[Theorem 3.1 and Corollary 3.4]{2022}. Similarly, if $7\notin \{p,q,r\}$ then $\ol\Gamma(K_2)$ has a $\{p,q,r\}$ triangle and $\{p,q,r\}=\{2,3,5\}$ by \cite[Corollary 4.3]{2022}. As the edge $2-3$ cannot exist in $\ol \Gamma(G)$, we reach a contradiction in both cases. Hence $5,7\in \{p,q,r\}$.

    To complete the proof, it suffices to show that for all $p \in \pi(G) \setminus \pi(A_7),$ either $5-p \notin \ol \Gamma(G)$ or $7-p \notin \ol \Gamma(G).$ By \Cref{funny}, since $|M(T)| = 6$, we may assume that $G$ has a section $V.E$ where $V$ is a nontrivial elementary abelian $p$-group and $E$ is a perfect central extension of $A_7$. Then, a well-known classification result shows that $E$ is isomorphic to either $A_7$ itself, the double cover $2.A_7$, the triple cover $3.A_7$, or the Schur cover $6.A_7$. We verified in GAP \cite{GAP4} that for each complex irreducible representation of these covers, either order 5 or order 7 elements have fixed points (see \Cref{table:3} of the Appendix). Therefore, \Cref{reps} implies $5-p \notin \ol \Gamma(G)$ or $7-p \notin \ol \Gamma(G)$.
\end{proof}

The next two lemmas consider the situation in which $2$ is adjacent to some vertex outside of $\pi(A_7)$.

\begin{lemma}\label{edges from 2}
    Let $G$ be strictly $A_7$-solvable, $K \cong N.A_7$ the subgroup granted by \Cref{2.4 2022}, and $E$ the double cover $2.A_7$. If $2 - p \in \ol{\Gamma}(G)$ for some $p \in \pi(G) \setminus \pi(A_7)$, then there exists a solvable $2'$-group $L$ such that $K \cong L.E$. In particular, $2 - 5, 2 - 7 \notin \ol{\Gamma}(G)$.
\end{lemma}

\begin{proof}
    Let $Q \in \Syl_2(K)$ and $P \in \Syl_p(K)$. Note that $NQ$ is solvable, and we may assume that $PQ$ is a Hall $\{2,p\}$-subgroup of $NQ$. Since $2-p \in \ol \Gamma(G)$, we know that $PQ$ is Frobenius or 2-Frobenius of type $(2, p)$ or $(2,p,2)$, respectively. But the latter is not possible, since in a 2-Frobenius group the top Frobenius complement must be cyclic, but $PQ$ has a Klein-4 factor group. Thus $Q$ must act Frobeniusly on $P$. Hence, \cite[Corollary 6.17]{FGT} shows that $Q$ is either cyclic or generalized quaternion. Consider the Sylow 2-subgroup $Q_0 = Q \cap N$ of $N$. Because $Q/Q_0 \in \Syl_2(A_7)$ is isomorphic to $D_8$, $Q$ must in fact be generalized quaternion. Now, recall that generalized quaternion groups only have one normal subgroup of index 8, and this subgroup is always cyclic. Thus $Q_0$ is cyclic of order $2^m$ for some $m\geq1$.

    By Burnside's $p$-complement theorem, we can write $N = L.Q_0$, where $L$ is a normal Hall $2'$-subgroup of $N$. In fact, $L$ is characteristic in $N$, which implies it is normal in $G$. Observe that $N/L\cong Q_0$ is abelian, and thus $(K/L)/(N/L) \cong A_7$ acts on $N/L\cong Q_0$ by conjugation. 
    However, the simplicity of $A_7$ implies that this action is trivial.
    Hence all of $K/L \cong Q_0.A_7$ acts trivially on $Q_0$. In particular, $Q_0$ belongs to the center of some Sylow 2-subgroup of $K/L$, which is isomorphic to $Q$ because $2 \ndivides |L|$. We showed above that $Q$ is generalized quaternion, so this center has order 2; since $Q_0$ is nontrivial, it follows that $Q_0 \cong C_2.$ Therefore $Q \cong Q_{16}$, the generalized quaternion group of order $16$. It is well-known that the group $Q_{16}$ does not split over $C_2$, our extension $Q_0.A_7$ cannot split either. Therefore, $Q_0.A_7$ is the double cover of $A_7$, establishing the first statement of the lemma. In particular, $Z(Q_0.A_7) = Q_0$, so the order 5 and 7 elements of $Q_0.A_7$ commute with at least one order 2 element. This means $5-2, 7-2 \notin \ol \Gamma(G)$.
\end{proof}

The reader has probably noticed that many results of this paper are essentially of the form "if this edge exists, then that edge doesn't exist." Because the hypothesis remains valid for subgroups of $G$ and the conclusion is preserved when extending back to $G$, \Cref{2.4 2022} comes in very handy for the proofs of such results. The following lemma is a rare case where this strategy does not apply, for one of the statements is of the form "if this edge exists, then so does that edge." The existence of edges in the prime graph complement of a subgroup does not readily translate to the existence of edges in the prime graph complement of $G$. Hence, we must instead apply \Cref{N.M} to analyze the structure of $G$ as a whole. Fortunately, our hypotheses and the nature of $\Aut(A_7)$ allow us to do so without much trouble.

\begin{lemma}\label{a7doublecoverfun}
Let $G$ be strictly $A_7$-solvable such that 2 is adjacent to some vertex $q \in \pi(G) \setminus \pi(A_7)$, $\ol \Gamma(G)$ contains at least one triangle, and $5-3 \in \ol \Gamma(G)$. If $5-p\in \ol \Gamma(G)$ for some $p \in \pi(G) \setminus \pi(A_7)$, then $2-p\in \ol\Gamma(G)$. On the other hand, if $7-p \in \ol \Gamma(G)$, then $2-p \notin \ol \Gamma(G)$.
\end{lemma}

\begin{proof}
    Before assuming that $5-p \in \ol \Gamma(G)$ or $7-p \in \ol \Gamma(G)$, let us first lay some groundwork. According to \Cref{N.M} and the fact that $\Aut(A_7) \cong S_7 \cong A_7.C_2$, we can find a solvable group $N$ such that either $G \cong N.A_7$ or $G \cong N.S_7$. If $G \cong N.S_7$, then examine a quotient of $G$ isomorphic to $V.W.S_7$ for some solvable group $W$ and elementary abelian $q$-group $V$. Since the Sylow 2-subgroups of $S_7$ are isomorphic to $D_8 \times C_2$ and thus do not satisfy the Frobenius criterion, we see that any Sylow 2-subgroup of $W.S_7$ acts with fixed points on $V$. This implies $2-q \notin \ol \Gamma(G)$, a contradiction. Hence $G \cong N.A_7$, and by \Cref{edges from 2} (In this case $K=G$), we can in fact write $G \cong L.E$ for some solvable $2'$-group $L$ and $E$ the double cover of $A_7$. 
    
    By \Cref{W} and the assumption that $\ol \Gamma(G)$ contains a triangle, we see that $|L|$ is not divisible by 5 or 7. Also $G$ is clearly $\{2,3,5,7,p\}$-separable, so by \cite[Theorem 3.20]{FGT} there exists a Hall $\{2,3,5,7,p\}$-subgroup $H \leq G$. Then, according to \cite[Lemma 2.1]{2022}, we have $H \cong L_{3p}.C_2.A_7$ for some Hall $\{3,p\}$-subgroup $L_{3p} \leq L$. Since $5-3,5-p \in \ol \Gamma(G)$, $H$ has a cyclic subgroup of order 5 or 7 acting Frobeniusly on $L_{3p}$. Hence $L_{3p}$ is nilpotent, and we can write $L_{3p} \cong P \times Q$ for some Sylow $p$-subgroup $P$ and Sylow $3$-subgroup $Q$ of $L_{3p}$. Because $Q$ is a normal Sylow subgroup of $L_{3p}$, it is in fact characteristic in $L_{3p}$ and hence normal in $H$. By the Schur-Zassenhaus theorem, we then have $H/Q \cong P \rtimes E$.

    We prove the first of our desired implications by contrapositive. Suppose $2-p \notin \ol \Gamma(G)$, so there exists an element $x \in H$ of order $2p$. Replacing $x$ with some conjugate of itself, we may assume that the image of $x^p$ modulo $Q$ belongs to $E$. Now let $P_0 \leq P$ be a normal subgroup of $H/Q$ minimal with respect to the condition that it contains $x^2$. Passing $H$ to a smaller quotient if necessary, we may assume that $P_0$ is a minimal normal subgroup of $H$ and thus elementary abelian. Now, $E$ acts on $P_0$ in such a way that $x^p$ (modulo $Q$) fixes $x^2$. In other words, we have an order 2 element of $E$ fixing an order $p$ element of $P_0$. By \Cref{reps} and \Cref{table:3}, some order 5 element of $E$ also has fixed points in $P_0$, which means $5-p \notin \ol \Gamma(G)$, as needed.

    Now let us prove the second implication. Suppose $7-p \in \ol \Gamma(G)$, so that in particular $7-p \in \ol \Gamma(H/Q)$. Let $P_1$ denote a maximal normal subgroup of $H/Q$ (properly) contained in $P$. Then $(H/Q)/P_1 \cong (P/P_1) \rtimes E$ and every element of $E$ of order 7 acts fixed point freely on $P/P_1$. Applying \Cref{reps} and the information about the double cover of $A_7$ in \Cref{table:3}, we find that some order 2 elements of $E$ also acts with fixed points on $P/P_1$. This means $2-p \notin \ol \Gamma((H/Q)/P_1)$, and it follows that $2-p \notin \ol \Gamma(G)$, as desired.
\end{proof}

Together with the $r-p$ criteria from \Cref{section:preliminaries}, we almost have enough for our main result. The next lemma establishes the 3-colorability of $A_7$-solvable groups.

\begin{lemma}\label{harder3color}
Let $\Gamma$ be the prime graph of some strictly $A_7$-solvable group $G$. Then one of the following holds:
\begin{enumerate}
    \item $\ol\Gamma$ has a $3$-coloring where all neighbors of $2,5,$ and $7$ not in $\pi(A_7)$ have the same color.
    \item $\ol\Gamma$ is triangle-free and $3$-colorable.
\end{enumerate} 
\end{lemma}

\begin{proof}

By \Cref{2.4 2022}, $G$ has a subgroup of the form $K\cong N.A_7$ where $N$ is solvable and $\pi(G)=\pi(K).$ If $5$ or $7$ divide $|N|$ we are done by \Cref{W}, and for the remainder of the proof we will assume otherwise. For any solvable group $S$, let $\vv\Gamma(S)$ denote the Frobenius digraph of $S$ (see \Cref{frobenius digraph}). By \cite[Corollary 2.7]{2015}, $\vv\Gamma(S)$ has no directed $3$-paths. 
We induce a partial orientation $\vv\Gamma$ on edges of $\ol\Gamma$ according to the orientation they are given in $\vv\Gamma(N.C_5)$ or $\vv\Gamma(N.C_7).$  To see that this is well-defined, note that all edges of $\vv\Gamma(N.C_5)$ are either contained in $\vv\Gamma(N)$ or are of the form $5\rightarrow p.$ Likewise, all edges of $\vv\Gamma(N.C_7)$ are either contained in $\vv\Gamma(N)$ or are of the form $7\rightarrow p.$ Since $5,7\ndivides |N|,$ all edges shared by $\vv\Gamma(N.C_5)$ and $\vv\Gamma(N.C_7)$ are oriented according to $\vv\Gamma(N).$ Thus the orientations given by $\vv\Gamma(N.C_5)$ and $\vv\Gamma(N.C_7)$ align where they overlap in $\ol\Gamma(G)$. If $5-7\in \ol\Gamma(G),$ we additionally define the orientation $5\rightarrow 7.$ By \Cref{no 2-paths}, vertices adjacent to $5$ or $7$ have zero out-degree in $\vv\Gamma.$ Thus $5$ and $7$ have no outgoing $2$-paths, and it follows that $\vv\Gamma$ has no $3$-paths. 

We assign a $3$-coloring to vertices of $\pi(N)\cup\{5,7\}$ as follows: Label all vertices in $\pi(N)\cup\{5,7\}$ with zero out degree as $\mathcal{I},$ all vertices in $\pi(N)\cup\{5,7\}$ with zero in-degree and non-zero out-degree as $\mathcal{O},$ and all vertices in $\pi(N)\cup\{5,7\}$ with nonzero in- and out-degree with $\mathcal{D}.$ As $V(\ol\Gamma\setminus\{2,3\})\subseteq \pi(N)\cup \{5,7\}$, this procedure induces a 3-coloring on the subgraph of $\ol\Gamma$ induced by $V(\ol\Gamma\setminus\{2,3\})$. But observe that the Sylow $3$-subgroups of $A_7$ do not satisfy the Frobenius criterion, so by \Cref{FC} we have $3-p\notin \ol\Gamma$ for all $p\in \pi(G) \setminus \pi(A_7).$ Thus if  $2-p\notin \ol\Gamma(G)$ for all $p\in \pi(G)\setminus \pi(A_7),$ then the 3-coloring of the vertices $V(\ol\Gamma\setminus\{2,3\})$ can easily be extended to $\ol\Gamma.$ Additionally, all neighbors of $5$ and $7$ outside of $\pi(A_7)$ belong to the same color $\mathcal{I}$, as desired. 

Now suppose there exists an edge $2-p\in \ol\Gamma$ for some $p\in \pi(G) \setminus \pi(A_7).$ By \Cref{edges from 2}, we see that $5-2,7-2\notin \ol\Gamma$ and that Hall $\{2,p\}$-subgroups of $N$ must be Frobenius of type $(2,p)$ or $2$-Frobenius of type $(2,p,2),$ giving the orientation $2\rightarrow p$ in $\vv\Gamma.$ As $2 \divides |N|,$ $\mathcal{O},\mathcal{D},\mathcal{I}$ gives a $3$-coloring of $\ol\Gamma\setminus\{3\}.$ 
Moreover, by \Cref{no 2-paths}, neighbors of $2$, $5,$ and $7$ have zero out-degree, placing them in $\mathcal{I}.$ Because the vertex $3$ can only be adjacent to vertices $5$ and $7,$ this $3$-coloring can easily be extended to $\ol\Gamma,$ completing the proof.
\end{proof}

We arrive at a full classification of the prime graphs of $A_7$-solvable groups.
\begin{theorem}\label{a7bigproof}
     Let $\Gamma$ be an unlabeled simple graph. Then $\Gamma$ is isomorphic to the prime graph of an $A_7$-solvable group if and only if one of the following is satisfied:
    \begin{enumerate}
    \item[(1)] $\ol \Gamma$ is triangle-free and 3-colorable.
    \item[(2)] There exist vertices $a,b,c,d$ in $\ol \Gamma$ and a $3$-coloring for which all vertices adjacent to but not included in $\{a,b,c,d\}$ have the same color. Additionally, one of the following holds:
    \begin{enumerate}
        \item[(2.1)] $\ol\Gamma$ contains exactly two triangles $\{a,b,c\}$ and $\{b,c,d\}.$ All edges incident to $a,b,d$ are included in these two triangles.
        \item[(2.2)] $\ol\Gamma$ contains exactly one triangle $\{a,b,c\}$ and a vertex $d$ adjacent only to vertex $c.$ Vertices $a$ and $b$ are not adjacent to any other vertices.
        \item[(2.3)] $\ol\Gamma$ contains exactly one triangle $\{a,b,c\}$ and a vertex $d$ that is adjacent only to vertex $b.$ Vertices $a$ and $b$ are not adjacent to any other vertices.
        \item[(2.4)] $\ol\Gamma$ contains exactly one triangle $\{a,b,c\}$ and a vertex $d.$  Vertex $a$ is not adjacent to any other vertices. Vertex $d$ is not adjacent to $a,b,$ or $c.$ Additionally, $N(b)\setminus\{a,c\}\subseteq N(d),$ and $N(c)\cap N(d)=\emptyset.$
        \item[(2.5)] $\ol\Gamma$ contains exactly one triangle $\{a,b,c\}$ and an isolated vertex $d$.The vertex $a$ is not adjacent to any other vertices.
    \end{enumerate}
    \end{enumerate} 
\end{theorem}

\begin{proof}
We will first prove the forward direction. Suppose that $\Gamma$ is the graph of an $A_7$-solvable group $G$. By \Cref{ThomasMichaelKeller}, we may assume that $G$ is strictly $A_7$-solvable. Then by \Cref{harder3color}, $\ol \Gamma$ is 3-colorable. If $\ol\Gamma(G)$ is also triangle-free, then $\Gamma$ satisfies (1) and we are done. Thus we will proceed assuming $\ol \Gamma$ has at least one triangle. 

Let us get a few miscellaneous observations out of the way before proceeding. Since the Sylow $3$-subgroups of $A_7$ do not satisfy the Frobenius criterion, we know that $3-p\notin \ol\Gamma$ for all $p\in \pi(G) \setminus \pi(A_7).$ Moreover, another application of \Cref{harder3color} shows that $\ol \Gamma(G)$ has a 3-coloring for which all neighbors of $2,5,7$ not in $\pi(A_7)$ have the same color. Finally, according to \Cref{no 5-7-p}, the only possible triangles in $\ol\Gamma$ are $\{2,5,7\}$ and $\{3,5,7\}.$ This means in particular that $5-7\in \ol\Gamma$ and one of $5-2 \in \ol \Gamma$ or $5-3 \in \ol \Gamma$ must hold. With that, we split into three cases.

\textbf{Case 1:} $5-2,5-3\in \ol\Gamma:$ As $5-7,5-2,5-3\in \ol\Gamma,$ by \Cref{Thanks TMK} and \Cref{edges from 2}, there are no edges $5-p$ or $2-p$ in $\ol\Gamma$ for $p\notin \pi(A_7).$  Moreover, as $\ol\Gamma$ has at least one triangle, we have $7-2\in \ol\Gamma$ or $7-3\in \ol\Gamma.$ If $7-2$ and $7-3\in \ol\Gamma,$ then (2.1) holds. Otherwise, (2.3) is satisfied.

\textbf{Case 2:} $5-2\in \ol\Gamma, 5-3\notin \ol\Gamma:$ By \Cref{edges from 2} there are no edges $2-p$ in $\ol\Gamma$ for any $p\notin \pi(T).$ In the case that $7-3\in \ol\Gamma,$ it follows from \Cref{Thanks TMK} that $5-p\notin \ol\Gamma$ for $p\notin \pi(A_7).$ Thus (2.2) holds. In the case that $7-3\in \ol\Gamma,$ \Cref{no 5-7-p} shows that $\ol\Gamma$ has exactly one triangle $\{2,5,7\}.$ Thus in this case, $\ol\Gamma$ satisfies (2.5).

\textbf{Case 3:} $5-3\in \ol\Gamma, 5-2 \notin \ol\Gamma:$ In the case that $7-2\in \ol\Gamma,$ similar to Case $1$, we have $5-p$ and $2-p\notin \ol\Gamma$ for all $p\notin \pi(T).$ As such, $\ol\Gamma$ satisfies (2.2). In the case that $7-2\notin \ol\Gamma$ and $2-p\in \ol\Gamma$ for some $p\notin \pi(A_7),$ \Cref{edges from 2} and \Cref{a7doublecoverfun} show that $N(5)\setminus\{3,7\}\subseteq N(2)$ and $N(7)\cap N(2)=\emptyset.$ As such $\ol\Gamma$ satisfies (2.4). In the case that $7-2\notin \ol\Gamma$ and $2-p\notin \ol\Gamma$ for all $p\notin \pi(A_7),$ $\ol\Gamma$ has exactly one triangle $\{3,5,7\},$ thus satisfying (2.5).

This completes the forwards direction. For the backwards direction, if $\ol\Gamma$ is triangle-free and $3$-colorable, then $\Gamma$ is the prime graph of a solvable group by \Cref{ThomasMichaelKeller}. Hence we need only consider $\Gamma$ satisfying claims (2.1)-(2.5).

The cases (2.1)-(2.4) are a straightforward application of \Cref{graphlizard} and the information from \Cref{table:3}. We will go through the process for the most complicated case (2.4) and leave the rest to the reader. Suppose (2.4) holds and let $X = \{a,b,c,d\}$. Making the assignments $a\mapsto 3$, $b\mapsto 5$, $c\mapsto 7$, and $d\mapsto 2$, we obtain a graph isomorphism from the subgraph of $\ol \Gamma$ induced by $X$ to $\ol \Gamma(E)$, where $E$ is the double cover of $A_7$ (see table). For each $v \in N(X) \setminus X$, our hypothesis states that one of the following holds: (1) $v$ is adjacent to $b$ and $d$ but not $a$ or $c$; (2) $v$ is adjacent to $d$ but not $a$, $b$, or $c$; or (3) $v$ is adjacent to $c$ but not $a$, $b$ or $d.$ Now consider the complex irreducible representations of $E$ described in \Cref{table:3}. If (1) holds, note that there is a representation in which order 3 and 7 elements have fixed points while order 2 and 5 do not. If (2) holds, take the representation in which elements of order 3, 5, and 7 have fixed points and order 2 elements do not. Finally, if (3) holds, select the representation in which elements of order 2, 3, and 5 have fixed points while order 7 elements do not. These representations obey the conditions laid out in \Cref{graphlizard} with respect to our chosen graph isomorphism, so $\Gamma$ is the prime graph of an $A_7$-solvable group. For cases (2.1), (2.2), and (2.3), apply a similar argument with $E = A_7,$ $64.A_7,$ and $16.A_7$ respectively.

Finally, we consider the case that (2.5) holds. Let $\ol \Gamma_0$ be the graph obtained from $\ol \Gamma$ by introducing the edge $d-v$ whenever $b-v$ exists for some $v \in \ol \Gamma \setminus X$. By this construction, $\ol \Gamma_0$ satisfies (2.4) so, by the above, $\Gamma_0$ is the prime graph of some $A_7$-solvable group $G$ (where $a$ is identified with 3, $b$ with 5, $c$ with 7, and $d$ with 2). Thus $\Gamma(G\times C_2)\cong \Gamma,$ showing that $\Gamma$ is also the prime graph of an $A_7$-solvable group. This completes the proof.
\end{proof}

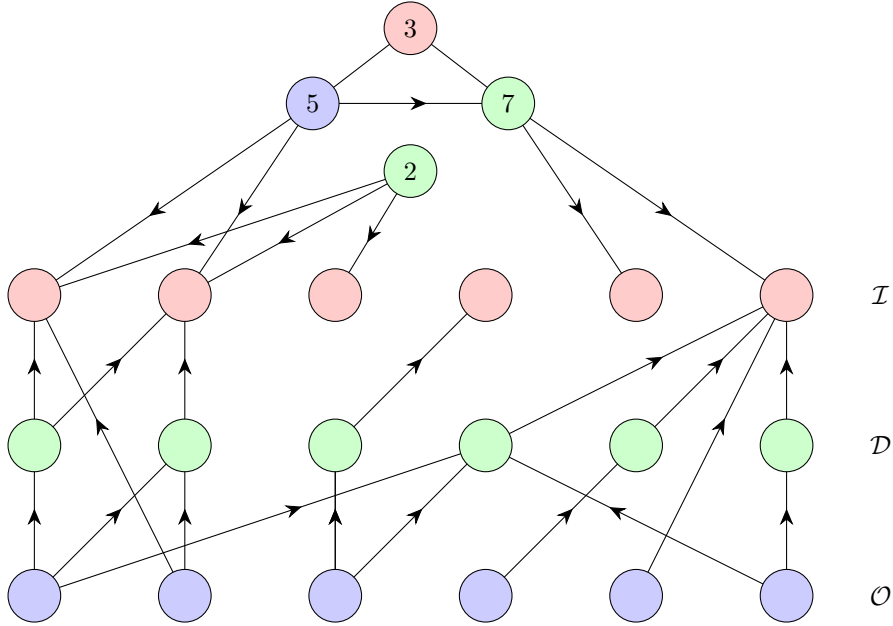
\begin{figure}[h!]
\centering
\begin{tikzpicture}
\tikzset{myarrow/.style={postaction={decorate},
decoration={markings,
mark=at position #1 with {\arrow{Stealth[scale=1.3,angle'=45]},semithick},
}}}

\tikzset{my2arrow/.style={postaction={decorate},
decoration={markings,
mark=at position #1 with {\arrow{Stealth[scale=1.3,angle'=45]},semithick},
}}}

\node (31) at (0, -.35) {$\bullet$};
\node (5) at (-1.3, .55) {$\bullet$};
\node (3) at (0, 1.55) {$\bullet$};
\node (2) at (1.3, .55) {$\bullet$};

\node (r1) at (-5, -2) {$\bullet$};
\node (r2) at (-3, -2) {$\bullet$};
\node (r3) at (-1, -2) {$\bullet$};
\node (r4) at (1, -2) {$\bullet$};
\node (r5) at (3, -2) {$\bullet$};
\node (r6) at (5, -2) {$\bullet$};

\node (q1) at (-5, -4) {$\bullet$};
\node (q2) at (-3, -4) {$\bullet$};
\node (q3) at (-1, -4) {$\bullet$};
\node (q4) at (1, -4) {$\bullet$};
\node (q5) at (3, -4) {$\bullet$};
\node (q6) at (5, -4) {$\bullet$};

\node (p1) at (-5, -6) {$\bullet$};
\node (p2) at (-3, -6) {$\bullet$};
\node (p3) at (-1, -6) {$\bullet$};
\node (p4) at (1, -6) {$\bullet$};
\node (p5) at (3, -6) {$\bullet$};
\node (p6) at (5, -6) {$\bullet$};

\draw (3) -- (5);

\draw [my2arrow=.6]  (5) -- (2);
\draw (3) -- (2);

\draw [my2arrow=.6] (5) -- (r1);
\draw [my2arrow=.6] (5) -- (r2);

\draw [my2arrow=.6] (31) -- (r1);
\draw [my2arrow=.6] (31) -- (r2);
\draw [my2arrow=.6] (31) -- (r3);

\draw [my2arrow=.6] (2) -- (r5);
\draw [my2arrow=.6] (2) -- (r6);

\draw [my2arrow=.6] (q2) -- (r2);

\draw [my2arrow=.6](p1) -- (q4);
\draw [my2arrow=.6](p1) -- (q2);
\draw [my2arrow=.6](p6) -- (q4);
\draw [my2arrow=.6](q4) -- (r6);
\draw [my2arrow=.6](q6) -- (r6);
\draw [my2arrow=.6](q1) -- (r2);

\draw [my2arrow=.6](p3) -- (q4);
\draw [my2arrow=.6](q5) -- (r6);
\draw [my2arrow=.6](p2) -- (r1);
\draw [my2arrow=.6](q1) -- (r1);
\draw [my2arrow=.6](p1) -- (q1);

\draw [my2arrow=.6](q3) -- (r4);

\draw [my2arrow=.6](p5) -- (r6);
\draw [my2arrow=.6](p2) -- (q2);
\draw [my2arrow= .6](p3) -- (q3);
\draw [my2arrow= .6](p4) -- (q5);
\draw [my2arrow= .6](p3) -- (q3);
\draw [my2arrow= .6](p6) -- (q6);

\draw[fill=red!20!white] (r1) circle (0.35);
\draw[fill=red!20!white] (r2) circle (0.35);
\draw[fill=red!20!white] (r3) circle (0.35);
\draw[fill=red!20!white] (r4) circle (0.35);
\draw[fill=red!20!white] (r5) circle (0.35);
\draw[fill=red!20!white] (r6) circle (0.35);

\draw[fill=green!20!white] (q1) circle (0.35);
\draw[fill=green!20!white] (q2) circle (0.35);
\draw[fill=green!20!white] (q3) circle (0.35);
\draw[fill=green!20!white] (q4) circle (0.35);
\draw[fill=green!20!white] (q5) circle (0.35);
\draw[fill=green!20!white] (q6) circle (0.35);

\draw[fill=blue!20!white] (p1) circle (0.35);
\draw[fill=blue!20!white] (p2) circle (0.35);
\draw[fill=blue!20!white] (p3) circle (0.35);
\draw[fill=blue!20!white] (p4) circle (0.35);
\draw[fill=blue!20!white] (p5) circle (0.35);
\draw[fill=blue!20!white] (p6) circle (0.35);

\draw[fill=green!20!white] (31) circle (0.35);
\draw[fill=red!20!white] (3) circle (0.35);
\draw[fill=blue!20!white] (5) circle (0.35);
\draw[fill=green!20!white](2) circle (0.35);

\node at (5) {$5$};
\node at (3) {$3$};
\node at (31) {$2$};
\node at (2) {$7$};

\node at (r1) {};
\node at (r2) {};
\node at (r3) {};
\node at (r4) {};
\node at (r6) {};

\node at (q1) {};
\node at (q2) {};
\node at (q3) {};
\node at (q4) {};
\node at (q6) {};

\node at (p1) {};
\node at (p2) {};
\node at (p3) {};
\node at (p4) {};
\node at (p6) {};

\node at (6.25, -2) {$\cal{I}$};
\node at (6.25, -4) {$\cal{D}$};
\node at (6.25, -6) {$\cal{O}$};
\node at (-6.25, -2) {$\;$};

\end{tikzpicture}
    \caption{This graph is the prime graph complement of a $A_7$-solvable group satisfying \Cref{a7bigproof} under condition (2.4) with partial orientation and resulting $3$-coloring as in the proof of \Cref{easy3color}.}
    \label{fig:a7_construction}
\end{figure}

\subsection{Prime graph complements that are triangle-free and 3-colorable}

\indent There are some $K_4$ groups $T$ for which the prime graphs of $T$-solvable groups are no more complicated than that of solvable groups. The $T$ of this type are $A_{10}$, $J_2$, $O_5(4)$, $O_5(5)$, $O_5(7)$, $O_5(9)$, $O_8^+(2)$, $U_3(9)$, ${}^3D_4(2)$, $A_9$, and $O_7(2)$. As it happens, the prime graphs of $T$-solvable groups for such $T$ are always triangle-free and 3-colorable.

\begin{lemma}\label{3 prime groups} Let $G$ be a $T$-solvable group where $T$ is one of $A_{10}$, $J_2$, $O_5(4)$, $O_5(5)$, $O_5(7)$, $O_5(9)$, $O_8^+(2)$, $U_3(9)$, ${}^3D_4(2)$, $A_9$, and $O_7(2)$. Then $\ol \Gamma(G)$ is $3$-colorable and triangle-free.
	
\end{lemma}
\begin{proof}
	One can check from \Cref{table:5} that $\ol \Gamma(T)$ is triangle-free and 3-colorable. Thus, by \Cref{easy3color}, it suffices to show that at most one vertex in $\pi(T)$ connects to the remainder of $\ol \Gamma(G)$. Let $s<r$ denote the largest primes dividing  $|T|,$ that is, the primes other than 2 or 3. We see from \Cref{table:2} that the Sylow $2$- and $3$-subgroups of $T$ do not satisfy the Frobenius criterion. Therefore, \Cref{FC} shows that $2-p, 3-p\notin \ol\Gamma(G)$ for all $p\in\pi(G) \setminus \pi(T).$ If we assume $T \ne A_9, O_7(2)$, then the same argument shows that $s -p \notin \ol \Gamma(G)$ for all $p \in \pi(G) \setminus \pi(T)$, as needed. On the other hand, if $T = A_9$, then $r=7$, $|M(T)| = 2$, and in every complex irreducible representation of a perfect central extension of $T$, there exists an element of order 7 with fixed points (see \Cref{table:3}). Similarly, if $T= O_7(2)$, then $r=7$, $|M(T)| = 2$, and in every complex irreducible representation of $T$, there exists an element of order 7 with fixed points. Hence, an application of \Cref{funny2} with $r=7$ completes the proof.
\end{proof}

We now present our classification result.

\begin{theorem}
    Let $T$ be one of $A_{10}$, $J_2$, $O_5(4)$, $O_5(5)$, $O_5(7)$, $O_5(9)$, $O_8^+(2)$, $U_3(9)$, ${}^3D_4(2)$, $A_9$, or $O_7(2)$. Then $\Gamma$ is the prime graph of a $T$-solvable group if and only if $\ol\Gamma$ is triangle-free and $3$-colorable.
\end{theorem}
\begin{proof}
The forward direction is shown in \Cref{3 prime groups}. For the backwards direction, we note that by \Cref{ThomasMichaelKeller}, for $\Gamma$ such that $\ol\Gamma$ is triangle-free and $3$-colorable, we can find a solvable group $N$ such that $\Gamma(N)\cong \Gamma.$ As $N$ is trivially $T$-solvable, the proof is complete. 
\end{proof} 

\subsection{The groups $\PSL(3,7)$, $U_3(5)$, ${}^2F_4(2)'$, $\PSL(4,3)$, $A_8$, and $U_3(8)$}

Now, we consider groups $T$ such that when $G$ is $T$-solvable, the subgraph of $\ol \Gamma(G)$ induced by $\pi(T)$ is a union of connected components of $\ol \Gamma(G)$. Unless $\ol \Gamma(G)$ is triangle-free and 3-colorable, this induced subgraph will always consist of a triangle plus at most one extra edge. The groups $\PSL(3,7)$, $U_3(5)$, ${}^2F_4(2)'$, $\PSL(4,3)$, $A_8$, and $U_3(8)$ are all like this, and their classification proofs are very similar. First, we prove that for $T$ one of these groups and $G$ a $T$-solvable group, $\ol \Gamma(G)$ is 3-colorable and either $\ol \Gamma(G)$ is triangle-free or $r-p\notin \ol\Gamma(G)$ for all $r\in \pi(T)$ and $p\in \pi(G) \setminus \pi(T).$ This result is split into two lemmas.


\begin{lemma}\label{trash}
	Let $T$ be one of $\PSL(3,7)$, $U_3(5)$, ${}^2F_4(2)'$, $\PSL(4,3)$, and $U_3(8)$, and let $G$ be a strictly $T$-solvable group. Then $\ol\Gamma=\ol\Gamma(G)$ is $3$-colorable and $r-p \notin \ol \Gamma(G)$ for all $r \in \pi(T)$, $p \in \pi(G) \setminus \pi(T)$.
\end{lemma}

\begin{proof}
	Note that $6 \divides |T|$ and, by \Cref{table:2}, the Sylow 2- and 3-subgroups of $T$ do not satisfy the Frobenius Criterion. It follows from \Cref{FC} that $2-p, 3-p\notin \ol\Gamma(G)$ for all $p\in \pi(G) \setminus \pi(T).$
	
	Now let $s < r$ be the two remaining primes in $\pi(T)$. Note that $M(T)$ is always a $\{2,3\}$-group, so that $(r, |M(T)|) = (s, |M(T)|) = 1$. Moreover, \Cref{table:3} shows that in every complex irreducible representation of a perfect central extension of $T$, there exist elements of order $r$ and $s$ with fixed points. Hence, it follows from \Cref{funny2} that $r-p, s-p \notin \ol \Gamma(G)$ for all $p \in \pi(G) \setminus \pi(T)$, as desired. 
\end{proof}

\begin{lemma}\label{more A8 stuff}
Let $G$ be a strictly $A_8$-solvable group. Then $\ol\Gamma(G)$ is $3$-colorable and one of the following holds: 
\begin{enumerate}
    \item[(1)] $\ol\Gamma(G)$ is triangle-free.
    \item[(2)] $r-p\notin \ol\Gamma(G)$ for all $r\in \pi(A_8)$ and $p\in \pi(G) \setminus \pi(A_8).$  
\end{enumerate}
\end{lemma}
\begin{proof}
Assume that $\ol \Gamma(G)$ contains a triangle. Because the Sylow 2- and 3-subgroups of $T$ do not satisfy the Frobenius Criterion (see \Cref{table:2}), it follows from \Cref{FC} that $2-p, 3-p\notin \ol\Gamma(G)$ for all $p\in \pi(G) \setminus \pi(T).$ In addition, since $|M(A_8)| = 2$, \Cref{funny2} and the information of \Cref{table:3} imply $7-p \notin \ol \Gamma(G)$ for all $p \in \pi(G) \setminus \pi(T)$. Then, we see from \Cref{easy3color} and \Cref{table:1} that $\ol \Gamma(G)$ is 3-colorable and every triangle in $\ol \Gamma(G)$ is contained in $\ol \Gamma(A_8)$. Since $\ol \Gamma(A_8)$ only has one triangle, namely $\{2, 5, 7\}$ (see \Cref{table:5}, this is the only possible triangle in $\ol \Gamma(G)$.

Now let $p \in \pi(G) \setminus \pi(T)$ be arbitrary. Applying \Cref{funny} with $r=5$, we may assume that $G$ has a section $V.E$ where $V$ is a nontrivial elementary abelian $p$-subgroup and $E$ is a perfect central extension of $A_8$. A well-known classification result shows that $E$ is either $A_8$ or the double cover $2.A_8$. If $E = A_8$, then \Cref{reps} and \Cref{table:3} give $5-p \notin \ol \Gamma(G)$, establishing (2). On the other hand, if $E = 2.A_8$, then $2-5, 2-7 \notin \ol \Gamma(G)$. Since every triangle in $\ol \Gamma(G)$ must also be contained in $\ol \Gamma(A_8)$, this forces $\ol \Gamma(G)$ to be triangle-free, which is (1).
\end{proof}

We now present the main theorem of this subsection.

\begin{theorem}
Let $T$ be one of $\PSL(3,7)$, $U_3(5)$, ${}^2F_4(2)'$, $\PSL(4,3)$, $A_8$, and $U_3(8)$ and let $\Gamma$ be an unlabeled simple graph. Then $\Gamma$ is isomorphic to the prime graph of a $T$-solvable group $G$ if and only if $\ol\Gamma$ is $3$-colorable and one of the following holds:

\begin{enumerate}
    \item[(1)] $\ol\Gamma$ is triangle-free.
    \item[(2)] $\ol\Gamma$ has exactly one isolated triangle $\{a,b,c\}$ and an isolated vertex $d$. 
    \item[(3)] $\ol\Gamma$ has exactly one triangle $\{a,b,c\},$ the vertices $a,b$ are not adjacent to any other vertices in $\ol\Gamma,$ vertex $c$ is adjacent to exactly one other vertex $d,$ and $d$ is not adjacent to vertices other than $c.$ 
\end{enumerate}
\end{theorem}

\begin{proof}
We will first prove the forward direction. Assume that $\Gamma$ is the prime graph of a $T$-solvable group $G.$ We may assume that $G$ is strictly $T$-solvable, otherwise (1) holds. Then \Cref{trash} and \Cref{more A8 stuff} show that $\ol\Gamma$ is $3$-colorable and either $\ol \Gamma(G)$ is triangle-free or $r-p \notin \ol \Gamma(G)$ for all $r \in \pi(T), p \in \pi(G) \setminus \pi(T)$. In the former case, (1) holds, so we may assume that $\ol \Gamma(G)$ contains a triangle and that $r-p \notin \ol \Gamma(G)$ for all $r \in \pi(T), p \in \pi(G) \setminus \pi(T)$. 
 As $T$ is a section of $G,$ the subgraph of $\ol\Gamma$ induced by $\pi(T)$ must be a subgraph of $\ol\Gamma(T)$, which can be found in \Cref{table:5}. In particular, the subgraph of $\ol\Gamma$ induced by $\pi(T)$ must consist of a triangle $\{a,b,c\}$ and a vertex $d$, which is either isolated or adjacent only to $c$. In the former case, $(2)$ holds, and in the latter case, $(3)$ holds. 

To prove the backwards direction, consider any 3-colorable graph $\Gamma$ satisfying one of the claims (1)-(3). If $\ol\Gamma$ satisfies $(1),$ then by \Cref{ThomasMichaelKeller} we can find a solvable group $N$ with $\ol\Gamma(N)=\ol\Gamma.$ Otherwise,  $\ol\Gamma\setminus\{a,b,c,d\}$ is $3$-colorable and triangle-free, and using techniques from \cite[Theorem 2.8]{2015}, we can find a solvable group $N$ with $(|N|,|T|)=1$ and $\ol\Gamma(N)=\ol\Gamma\setminus\{a,b,c,d\}.$ If $\ol\Gamma$ satisfies $(3),$ then $\ol\Gamma$ must be isomorphic to $\ol\Gamma(N\times T).$ Suppose then that $\ol\Gamma$ satisfies (2).
If $T$ is $A_8$ then $\ol\Gamma(T\times C_3)$ consists of a triangle and an isolated vertex and $\ol\Gamma$ must be isomorphic to $\ol\Gamma(N\times T\times C_3).$ If $T$ is one of $\PSL(3,7),$ $U_3(5),$ $^2F_4(2)',$ $\PSL(4,3),$ or $U_3(8),$ then $\ol\Gamma(T\times C_2)$ consists of a triangle and an isolated vertex and $\ol\Gamma$ must be isomorphic to $\ol\Gamma(N\times T\times C_2).$ This completes the proof.
\end{proof}

\subsection{The groups $\PSL(3,5)$, $\PSL(3,8)$, $\PSL(3,17)$, $U_3(4)$, $U_3(7)$, and $U_5(2)$}

The classification result for these groups is similar to the one in the previous subsection, except that one vertex in $\pi(T)$ is now allowed to be adjacent to vertices in $\pi(G) \setminus \pi(T)$.

\begin{theorem}\label{PSL Frenzy}
Let $T$ be one of $\PSL(3,5),$ $\PSL(3,8),$ $\PSL(3,17),$ $U_3(4),$ $U_3(7),$ or $U_5(2)$ and let $\Gamma$ be an unlabeled simple graph. Then $\Gamma$ is isomorphic to the prime graph of a $T$-solvable group if and only if $\ol\Gamma$ is $3$-colorable and one of the following holds.

\begin{enumerate}
    \item[(1)] $\ol\Gamma$ is triangle-free.
    \item[(2)] $\ol\Gamma$ contains exactly one triangle $\{a,b,c\}$ and an isolated vertex $d$ such that $a,b$ are not adjacent to other vertices in $\ol\Gamma.$ Additionally, $\ol\Gamma$ has a $3$-coloring for which all neighbors of $c$ other than $a,$ $b,$ and $d$ have the same color.
    \item[(3)] $\ol\Gamma$ contains exactly one triangle $\{a,b,c\}$ and the vertices $a,b$ are not adjacent to any other vertices in $\ol\Gamma.$ $\ol\Gamma$ has a vertex $d$ adjacent only to vertex $c.$ Additionally, $\ol\Gamma$ has a $3$-coloring for which all neighbors of $c$ other than $a,$ and $b,$ and $d$ have the same color.     
\end{enumerate}

\end{theorem}

\begin{proof}
We will first prove the forward direction. Assume that $\Gamma$ is the prime graph of a $T$-solvable group $G.$ By \Cref{ThomasMichaelKeller}, we may assume that $G$ is strictly $T$-solvable. Let $s_1, s_2, s_3, s_4$ be the prime divisors of $|T|$ in the order prescribed by the table below. Note that in each case, $\ol \Gamma(T)$ consists of the triangle $\{s_1,s_2,s_3\}$ and the edge $s_1-s_4.$

\begin{table}[h!]
    \centering
    \begin{tabular}{|c|c|c|c|c|} \hline
        $T$ & $s_1$ & $s_2$ & $s_3$ & $s_4$ \\ \hline
        $\PSL(3,5)$ & 31 & 3 & 5 & 2 \\
        $\PSL(3,8)$ & 73 & 3 & 2 & 7 \\
        $\PSL(3,17)$ & 307 & 3 & 17 & 2 \\
        $U_3(4)$ & 13 & 3 & 2 & 5\\
        $U_3(7)$ & 43 & 3 & 7 & 2\\
        $U_5(2)$ & 11 & 5 & 2 & 3 \\ \hline
    \end{tabular}
\end{table}

Since the Sylow $s_3$- and $s_4$-subgroups of $T$ do not satisfy the Frobenius criterion by \Cref{table:2}, \Cref{FC} eliminates the edges $s_3-p$ and $s_4-p$ from $\ol\Gamma$ for all $p\in \ol \Gamma(G) \setminus \pi(T).$ Furthermore, since $|M(T)| = 1$, the information of \Cref{table:3} alongside \Cref{funny} shows that $s_2 -p \notin \ol \Gamma$ for all $p \in \pi(G) \setminus \pi(T)$. Thus by \Cref{easy3color}, we see that $\ol\Gamma$ has a $3$-coloring for which all vertices adjacent to $s_1$ and not in $\pi(T)$ have the same color. Now if $\ol \Gamma$ is triangle-free, then (1) holds. Otherwise, an additional consequence of \Cref{easy3color} is that $\ol \Gamma$ contains exactly one triangle $\{s_1, s_2, s_3\}$. Thus (2) holds if $s_1-s_4\notin \ol\Gamma$, and (3) holds if $s_1-s_4\in \ol\Gamma.$ 

We now proceed with the reverse direction, in each case finding a $T$-solvable group $G$ such that $\Gamma(G)\cong \Gamma.$ If $\Gamma$ satisfies (1), then \cite[Theorem 3.8]{2015} produces a solvable group $G$ with $\Gamma\cong \Gamma(G).$ Note that $G$ is trivially $T$-solvable in this case.

Now suppose $\Gamma$ satisfies (3) and let $X=\{a,b,c,d\}.$ We proceed with a straightforward application of \Cref{graphlizard} and information from \Cref{table:3}.  Making the assignments $a \mapsto s_4,$ $b \mapsto s_2,$ $c \mapsto s_1,$ $d\mapsto s_3,$ we obtain a graph isomorphism from the subgraph of $\ol\Gamma$ induced by $X$ to $\ol\Gamma(T).$ For each $v\in N(X)\setminus X,$ our hypothesis states that $v$ is adjacent to $c$ but not $a,$ $b,$ or $d.$ Now consider the complex irreducible representations of $T$ described in \Cref{table:3}. Note that there is a representation in which order $s_2,s_3,$ and $s_4$ elements have fixed points but order $s_1$ elements do not. This representation obeys the conditions laid out in \Cref{graphlizard} with respect to our chosen graph isomorphism, so $\Gamma$ is the prime graph of a $T$-solvable group.

Now we consider the case that (2) holds. Let $\ol\Gamma_0$ be the graph obtained from $\ol\Gamma$ by introducing the edge $s_3-s_1.$ By this construction, $\ol\Gamma_0$ satisfies (3) and as such is the prime graph of some $T$-solvable group $G$ with the same assignments as above. Thus $\Gamma(G\times C_{s_3})\cong \Gamma,$ showing that $\Gamma$ is the prime graph of a $T$-solvable group. This completes the proof.
\end{proof}

\begin{figure}[h!]
	\centering
	\begin{tikzpicture}
		\tikzset{myarrow/.style={postaction={decorate},
				decoration={markings,
					mark=at position #1 with {\arrow{Stealth[scale=1.3,angle'=45]},semithick},
		}}}
		
		\tikzset{my2arrow/.style={postaction={decorate},
				decoration={markings,
					mark=at position #1 with {\arrow{Stealth[scale=1.3,angle'=45]},semithick},
		}}}
		
		\node (31) at (0, 0) {$\bullet$};
		\node (5) at (-1.7, .9) {$\bullet$};
		\node (3) at (0, 1.8) {$\bullet$};
		\node (2) at (1.7,.9) {$\bullet$};
		
		\node (r1) at (-5, -2) {$\bullet$};
		\node (r2) at (-3, -2) {$\bullet$};
		\node (r3) at (-1, -2) {$\bullet$};
		\node (r4) at (1, -2) {$\bullet$};
		\node (r5) at (3, -2) {$\bullet$};
		\node (r6) at (5, -2) {$\bullet$};
		
		\node (q2) at (-3, -4) {$\bullet$};
		\node (q3) at (-1, -4) {$\bullet$};
		\node (q4) at (1, -4) {$\bullet$};
		\node (q5) at (3, -4) {$\bullet$};
		
		\node (p2) at (-3, -6) {$\bullet$};
		\node (p3) at (-1, -6) {$\bullet$};
		\node (p4) at (1, -6) {$\bullet$};
		\node (p5) at (3, -6) {$\bullet$};
		\node (p6) at (5, -6) {$\bullet$};

		\draw(31) -- (3);
		\draw (31) -- (5);
		\draw (3) -- (5);
		\draw(31) -- (2);
		
		\draw [my2arrow=.6] (31) -- (r1);
		\draw [my2arrow=.6] (31) -- (r3);
		\draw [my2arrow=.6] (31) -- (r6);
		
		\draw [my2arrow=.6] (q2) -- (r2);
		
		\draw [my2arrow=.6](q2) -- (r1);
		\draw [my2arrow=.6](p6) -- (q4);
		\draw [my2arrow=.6](q4) -- (r6);
		
		\draw [my2arrow=.6](p3) -- (q4);
		\draw [my2arrow=.6](q5) -- (r6);
		
		\draw [my2arrow=.6](q3) -- (r4);
		
		\draw [my2arrow=.6](p5) -- (r6);
		
		\draw [my2arrow= .6](p3) -- (q3);
		\draw [my2arrow= .6](p4) -- (q5);
		\draw [my2arrow= .6](p3) -- (q3);

        \draw [my2arrow= .6]
        (p2) -- (q2);
        \draw [my2arrow= .6]
        (q2) -- (r3);
        \draw [my2arrow= .6]
        (p3) -- (q2);

		\draw[fill=red!20!white] (r1) circle (0.35);
		\draw[fill=red!20!white] (r2) circle (0.35);
		\draw[fill=red!20!white] (r3) circle (0.35);
		\draw[fill=red!20!white] (r4) circle (0.35);
		\draw[fill=red!20!white] (r5) circle (0.35);
		\draw[fill=red!20!white] (r6) circle (0.35);

        \draw[fill=green!20!white] (q2) circle (0.35);
		\draw[fill=green!20!white] (q3) circle (0.35);
		\draw[fill=green!20!white] (q4) circle (0.35);
		\draw[fill=green!20!white] (q5) circle (0.35);
		
		\draw[fill=blue!20!white] (p2) circle (0.35);
		\draw[fill=blue!20!white] (p3) circle (0.35);
		\draw[fill=blue!20!white] (p4) circle (0.35);
		\draw[fill=blue!20!white] (p5) circle (0.35);
		\draw[fill=blue!20!white] (p6) circle (0.35);
		
		\draw[fill=green!20!white] (31) circle (0.35);
		\draw[fill=blue!20!white] (3) circle (0.35);
		\draw[fill=red!20!white] (5) circle (0.35);
		\draw[fill=red!20!white](2) circle (0.35);
		
		\node at (5) {$5$};
		\node at (3) {$3$};
		\node at (31) {$31$};
		\node at (2) {$2$};
		
		\node at (r1) {};
		\node at (r2) {};
		\node at (r3) {};
		\node at (r4) {};
		\node at (r6) {};
		
		\node at (q2) {};
		\node at (q3) {};
		\node at (q4) {};
		
		\node at (p2) {};
		\node at (p3) {};
		\node at (p4) {};
		\node at (p6) {};
		
		\node at (6.25, -2) {$\cal{I}$};
		\node at (6.25, -4) {$\cal{D}$};
		\node at (6.25, -6) {$\cal{O}$};
		\node at (-6.25, -2) {$\;$};
		
	\end{tikzpicture}
	\caption{This graph is the prime graph complement of a $\PSL(3,5)$-solvable group $G$ with partial orientation and resulting $3$-coloring as in the proof of \Cref{easy3color}.}
	\label{fig:psl_construction}
\end{figure}
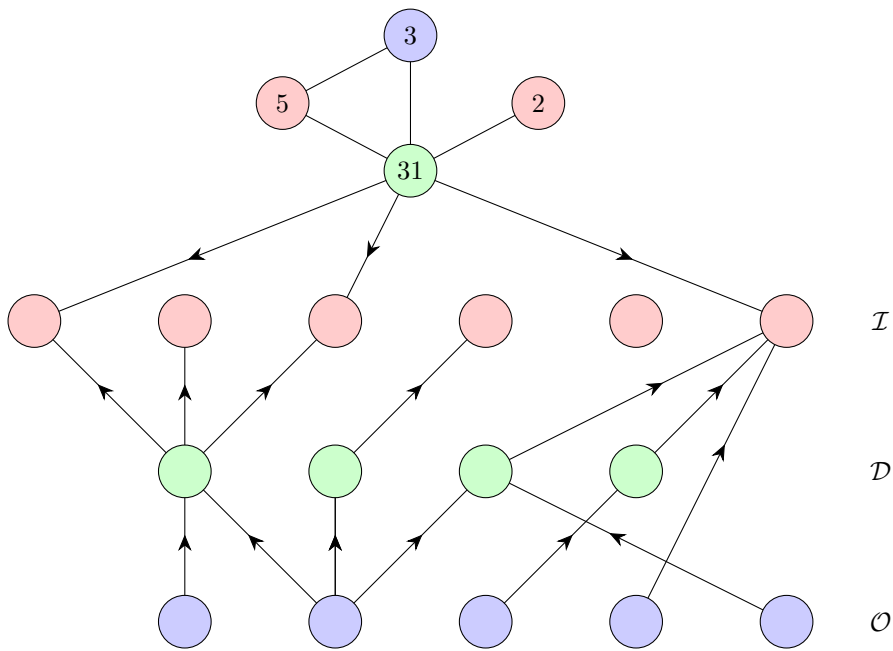
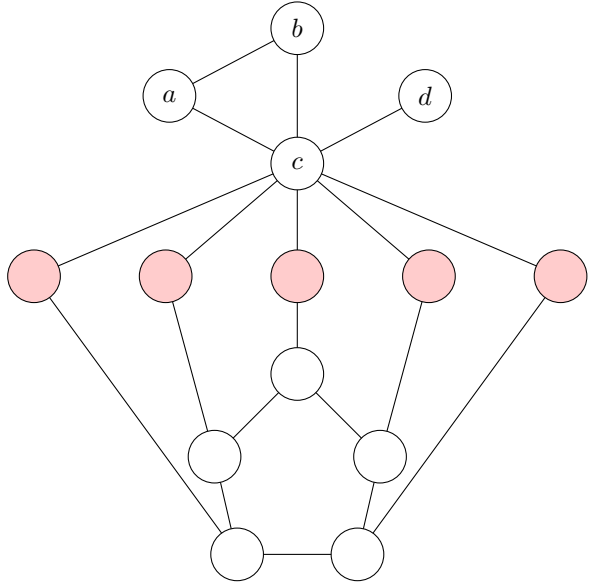
\begin{figure}[h!]
	\centering
	\begin{tikzpicture}
		\node (31) at (0, 0) {$\bullet$};
		\node (5) at (-1.7, .9) {$\bullet$};
		\node (3) at (0, 1.8) {$\bullet$};
		\node (2) at (1.7,.9) {$\bullet$};

		\node (r1) at (-3.5, -1.5) {$\bullet$};
		\node (r2) at (-1.75, -1.5) {$\bullet$};
		\node (r3) at (0, -1.5) {$\bullet$};
		\node (r4) at (1.75, -1.5) {$\bullet$};
		\node (r5) at (3.5, -1.5) {$\bullet$};
		
		\node (p3) at (0, -2.8) {$\bullet$};
		\node (p2) at (-1.1, -3.9) {$\bullet$};
		\node (p4) at (1.1, -3.9) {$\bullet$};
		\node (p1) at (-0.8, -5.2) {$\bullet$};
		\node (p5) at (0.8, -5.2) {$\bullet$};
		
		\draw (31) -- (5);
		\draw (31) -- (3);
		\draw (5) -- (3);
		\draw (31) -- (2);
		
		\draw (31) -- (r1);
		\draw (31) -- (r2);
		\draw (31) -- (r3);
		\draw (31) -- (r4);
		\draw (31) -- (r5);
		
		\draw (p1) -- (r1);
		\draw (p2) -- (r2);
		\draw (p3) -- (r3);
		\draw (p4) -- (r4);
		\draw (p5) -- (r5);
		
		\draw (p1) -- (p2);
		\draw (p2) -- (p3);
		\draw (p3) -- (p4);
		\draw (p4) -- (p5);
		\draw (p5) -- (p1);
		
		\draw[fill=white] (3) circle (0.35);
		\draw[fill=white] (5) circle (0.35);
		\draw[fill=white] (31) circle (0.35);
		\draw[fill=white] (2) circle (0.35);
		
		\draw[fill=red!20!white] (r1) circle (0.35);
		\draw[fill=red!20!white] (r2) circle (0.35);
		\draw[fill=red!20!white] (r3) circle (0.35);
		\draw[fill=red!20!white] (r4) circle (0.35);
		\draw[fill=red!20!white] (r5) circle (0.35);
		
		\draw[fill=white] (p1) circle (0.35);
		\draw[fill=white] (p2) circle (0.35);
		\draw[fill=white] (p3) circle (0.35);
		\draw[fill=white] (p4) circle (0.35);
		\draw[fill=white] (p5) circle (0.35);
		
		\node at (5) {$a$};
		\node at (3) {$b$};
		\node at (31) {$c$};
		\node at (2) {$d$};
	\end{tikzpicture} 
	\caption{This graph is not the prime graph complement of any $\PSL(3,5)$-solvable group as there is no $3$-coloring for which all neighbors of $c$ other than $a$, $b,$ and $d$ have the same color.}
\end{figure}

\newpage

\subsection{The groups $M_{11}$ and $M_{12}$}
We now turn our attention to the prime graphs of $M_{11}$- and $M_{12}$-solvable groups. For reference, note that $\ol \Gamma(M_{11})$ is the complete graph on $\{2, 3, 5, 11\}$ minus the $2-3$ edge, and $\ol \Gamma(M_{12}) = \ol \Gamma(M_{11}) \setminus \{2-5\}$ (see \Cref{table:5}).

\begin{lemma}\label{Ms}
	Let $T$ be $M_{11}$ or $M_{12}$ and let $G$ be strictly $T$-solvable. Then $2-p, 3-p, 5-p \notin \ol \Gamma(G)$ for all $p \in \pi(G) \setminus \pi(T)$. Moreover, every triangle in $\ol \Gamma(G)$ also lies in $\ol \Gamma(T)$, and $\ol \Gamma(G)$ has a 3-coloring such that every neighbor of 11 outside of $\pi(T)$ has the same color.
\end{lemma}

\begin{proof}
	\Cref{FC} immediately shows that $\ol \Gamma(G)$ has no edges $2-p$ and $3-p$ for any prime $p \in \pi(G) \setminus \pi(T)$. Now, observe that $|M(T)|$ is  $1$ or $2$ and in every complex irreducible representation of a perfect central extension of $T$, there exist order 5 elements with fixed points. By \Cref{funny2}, it follows that $5-p \notin \ol \Gamma(G)$ for all $p \in \pi(G) \setminus \pi(T)$. Then, \Cref{easy3color} finishes the proof.
\end{proof}

\begin{lemma}\label{fancy}
    Let $G$ be a strictly $M_{11}$-solvable group such that $2-5,3-5 \in \ol \Gamma(G)$. Additionally, assume that $\ol \Gamma(G)$ contains a triangle. Then $2-11,3-11 \in \ol \Gamma(G)$.
\end{lemma}

\begin{proof}
    Since $\Out(M_{11}) = 1$, \Cref{N.M} shows that $G \cong N.M_{11}$ for some solvable group $N$. To show that $11 \ndivides |N|$, let us assume the contrary. Note that $M_{11}$ contains an isomorphic copy of $A_6$, and so $G$ has a subgroup $K \cong N.A_6$ satisfying $\ol \Gamma(G) \subseteq \ol \Gamma(K)$. Because $\ol \Gamma(G)$ contains a triangle, \cite[Corollary 4.3]{2022} forces it to have the vertices $\{2,3,5\},$ contradicting the fact that $2-3 \notin \ol \Gamma(M_{11})$. Hence $11 \ndivides |N|$, as claimed. Then, if $2,3 \ndivides |N|$, then the fact that $2-11,3-11 \in \ol \Gamma(M_{11})$ implies the same for $\ol \Gamma(G)$, as needed. Thus we can assume either $2 \divides |N|$ or $3 \divides |N|$. The remainder of the proof in each case is identical, so we only consider the case $2 \divides |N|.$

    By considering a chief series of $G$, we can find a section $L= V.W.M_{11}$ where $V$ is a nontrivial elementary abelian $2$-group and $W$ is a solvable $2'$-group. Let $R$ be a cyclic subgroup of $M_{11}$ of order 11. If $5 \divides |W|$, then since $R$ acts coprimely on $W$, there exists an $R$-invariant Sylow 5-subgroup $S$ of $W$. Now \Cref{Ms} shows that $5-11 \in \ol \Gamma(G)$, so $R$ in fact must act nontrivially on $S$. Since $M_{11}$ contains a Frobenius subgroup $C_{11} \rtimes C_5$, it follows that
    the group $SR.C_5$ is a 2-Frobenius group. Therefore,
    $SR$ has a complement $Q \cong C_5$,
    and $RQ$ is a Frobenius group acting faithfully on $V$.
    By \Cref{truelizard}, it follows that $2-5 \notin \ol \Gamma(G)$, a contradiction. Hence $5 \ndivides |W|$. 
    
    Since $M_{11}$ contains a Frobenius subgroup $C_{11} \rtimes C_5$, $L$ has a subgroup $L_0 \cong V.W.C_{11}.C_5,$ which, by Schur-Zassenhaus, splits as $(V.W.C_{11}) \rtimes C_5.$ Because $C_5$ acts Frobeniusly on the quotient $C_{11}$ in this case, it acts nontrivially on a $C_5$-invariant Sylow 11-subgroup of $W.C_{11}$ (this is all possible because the action is coprime). By \Cref{truelizard}, $2-5 \notin \ol \Gamma(G)$, a contradiction.
\end{proof}

In the above proof, it is worth emphasizing the necessity of \Cref{N.M} as opposed to the usual \Cref{2.4 2022}. If we instead took an arbitrary subgroup $K \cong N.T$ with $\pi(G) = \pi(K)$, we would not have been able to reduce to the case that $2$ or $3$ divides $|N|$. 

We now provide a classification theorem for the groups $M_{11}$ and $M_{12}$. Since $\ol \Gamma(M_{11})$ is the complete graph minus one edge, the result for $M_{11}$ is reminiscent of that of $A_7$.

\begin{theorem}
     Let $\Gamma$ be an unlabeled simple graph. Then $\Gamma$ is isomorphic to the prime graph of an $M_{11}$-solvable group if and only if one of the following is satisfied:
    \begin{enumerate}
    \item[(1)] $\ol \Gamma$ is triangle-free and 3-colorable.
    \item[(2)] There exist vertices $a,b,c,d$ in $\ol \Gamma$ and a $3$-coloring for which all vertices adjacent to but not included in $\{a,b,c,d\}$ have the same color. Additionally, one of the following holds:
    \begin{enumerate}
        \item[(2.1)] $\ol\Gamma$ contains exactly two triangles $\{a,b,c\}$ and $\{b,c,d\}.$ All edges incident to $a,b,d$ are included in these two triangles.
        \item[(2.2)] $\ol\Gamma$ contains exactly one triangle $\{a,b,c\}$ and a vertex $d$ adjacent only to vertex $c.$ Vertices $a$ and $b$ are not adjacent to any other vertices.
        \item[(2.3)] $\ol\Gamma$ contains exactly one triangle $\{a,b,c\}$ and an isolated vertex $d.$  All edges incident to $a,b$ are included in this triangle.
    \end{enumerate}
    \end{enumerate} 
\end{theorem}

\begin{proof}
    For the forward direction, let $\Gamma$ be the prime graph of an $M_{11}$-solvable group $G$. We may assume $G$ is strictly $M_{11}$-solvable, otherwise (1) holds by \Cref{ThomasMichaelKeller}. Now, \Cref{Ms} implies the following. First, we have $2-p, 3-p, 5-p \notin \ol \Gamma$ for all $p \in \pi(G) \setminus \pi(T)$. Moreover, $\ol \Gamma$ has a 3-coloring such that all neighbors of 11 outside of $\pi(M_{11})$ have the same color, and $\ol\Gamma$ contains at most two triangles, which must be $\{2,5,11\}$ or $\{3,5,11\}$ triangles if they exist (see \Cref{table:5}). We may assume that at least one of these triangles exist, otherwise (1) holds.

    In what follows, we make constant use of the fact that $2-3 \notin \ol \Gamma(M_{11})$, which implies $2-3 \notin \ol \Gamma$. If $2-5,3-5 \in \ol \Gamma$, then \Cref{fancy} implies that both a $\{2,5,11\}$-triangle and a $\{3,5,11\}$-triangle exist, so (2.1) holds. If $2-5 \in \ol \Gamma$ but $3-5 \notin \ol \Gamma$, then $\ol \Gamma$ must contain a $\{2,5,11\}$-triangle. From there, if $3-11 \in \ol \Gamma$, then (2.2) holds, whereas if $3-11 \notin \ol \Gamma$, then (2.3) holds. Similarly, if $2-5 \notin \ol \Gamma$ but $3-5 \in \ol \Gamma$, then either (2.2) or (2.3) holds depending on if $2-11 \in \ol \Gamma$. Because we are assuming that at least one triangle exists, we have exhausted all cases of the forward direction.

    For the reverse direction, suppose $\Gamma$ is a graph satisfying (2) (if it satisfies (1) then we are done by the classification theorem for solvable groups). Then $\ol\Gamma\setminus\{a,b,c,d\}$ is $3$-colorable and triangle-free, so by the techniques of \cite[Theorem 2.8]{2015} we can find a solvable group $N$ with $(|N|,|M_{11}|)=1$ and $\ol\Gamma(N)=\ol\Gamma\setminus\{a,b,c,d\}.$ If (2.1) or (2.2) holds, then we are done according to an application of \Cref{graphlizard} with $E=M_{11}, 3^5.M_{11}$ respectively (specifically, $3^5.M_{11}$ is PerfectGroup(1924560,1) in the GAP library of perfect groups, see \cite{GAP4}). In each case, the graph isomorphism makes the assignments $a \mapsto 2$, $b \mapsto 5$, $c \mapsto 11$, and $d \mapsto 3$. See \Cref{table:3} as usual for representation-theoretic justification. On the other hand, if (2.3) holds, then by the previous argument we can find an $M_{11}$-solvable group $G$ such that $\ol \Gamma(G)$ is equal to $\ol \Gamma$ plus the additional edge $c-d.$ Now, we simply have $\ol \Gamma \cong \ol \Gamma(G \times C_d).$
\end{proof}

Because $\ol \Gamma(M_{12})$ contains only one triangle, the following classification is closer to the ones in the last couple subsections as opposed to $M_{11}$ or $A_7$.

\begin{theorem}
Let $\Gamma$ be an unlabeled simple graph. Then $\Gamma$ is isomorphic to the prime graph of a $M_{12}$-solvable group if and only if $\ol\Gamma$ is $3$-colorable and one of the following holds.

\begin{enumerate}
    \item[(1)] $\ol\Gamma$ is triangle-free.
    \item[(2)] $\ol\Gamma$ has exactly one triangle $\{a,b,c\},$  vertex $c$ is adjacent to exactly one other vertex $d,$ and $d$ is not adjacent to vertices other than $c.$  Vertices $a,b,c,d$ are not adjacent to any other vertices.
    \item[(3)] $\ol\Gamma$ has exactly one triangle $\{a,b,c\}$ and an isolated vertex $d$ such that $a,b$ are not adjacent to other vertices in $\ol\Gamma.$ Additionally, $\ol\Gamma$ has a $3$-coloring for which all neighbors of $c$ other than $a$ and $b$ have the same color.
\end{enumerate}
\end{theorem}
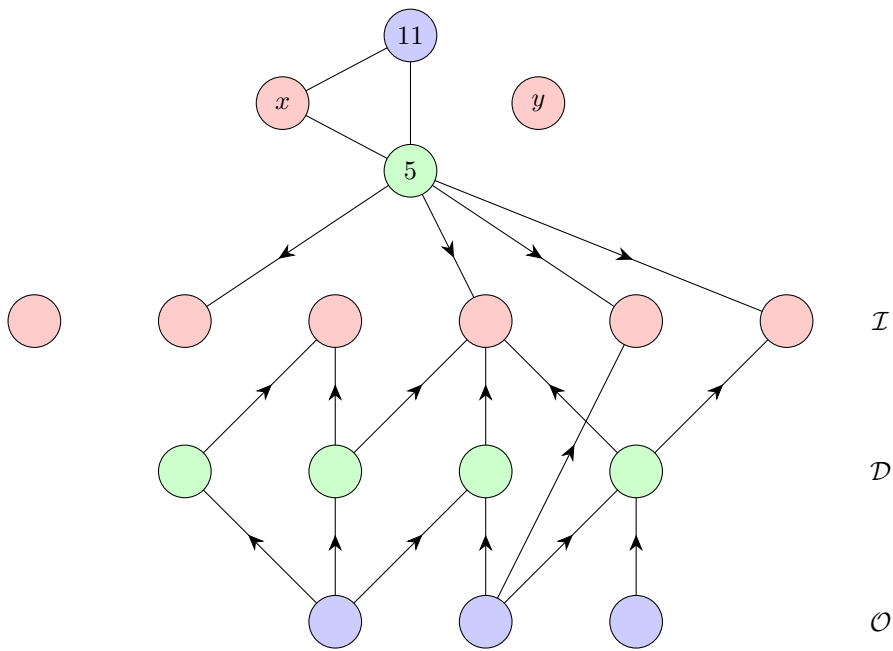
\begin{figure}[h!]
	\centering
	\begin{tikzpicture}
		\tikzset{myarrow/.style={postaction={decorate},
				decoration={markings,
					mark=at position #1 with {\arrow{Stealth[scale=1.3,angle'=45]},semithick},
		}}}
		
		\tikzset{my2arrow/.style={postaction={decorate},
				decoration={markings,
					mark=at position #1 with {\arrow{Stealth[scale=1.3,angle'=45]},semithick},
		}}}
		
		\node (5) at (0, 0) {$\bullet$};
		\node (x) at (-1.7, .9) {$\bullet$};
		\node (11) at (0, 1.8) {$\bullet$};
		\node (y) at (1.7,.9) {$\bullet$};
		
		\node (r1) at (-5, -2) {$\bullet$};
		\node (r2) at (-3, -2) {$\bullet$};
		\node (r3) at (-1, -2) {$\bullet$};
		\node (r4) at (1, -2) {$\bullet$};
		\node (r5) at (3, -2) {$\bullet$};
		\node (r6) at (5, -2) {$\bullet$};
		
		\node (q2) at (-3, -4) {$\bullet$};
		\node (q3) at (-1, -4) {$\bullet$};
		\node (q4) at (1, -4) {$\bullet$};
		\node (q5) at (3, -4) {$\bullet$};
		
		\node (p3) at (-1, -6) {$\bullet$};
		\node (p4) at (1, -6) {$\bullet$};
		\node (p5) at (3, -6) {$\bullet$};

		\draw(5) -- (11);
		\draw (5) -- (x);
		\draw (11) -- (x);
		
		\draw [my2arrow=.6] (5) -- (r2);
		\draw [my2arrow=.6] (5) -- (r4);
        \draw [my2arrow=.6] 
        (5) -- (r5);
		\draw [my2arrow=.6] (5) -- (r6);
		
		\draw [my2arrow=.6] (q2) -- (r3);
		
		\draw [my2arrow=.6](q3) -- (r3);
		
		\draw [my2arrow=.6](p3) -- (q4);
		
		\draw [my2arrow=.6](q3) -- (r4);
		
		\draw [my2arrow= .6](p3) -- (q2);
        \draw [my2arrow= .6]
        (q4) -- (r4);
        \draw [my2arrow= .6]
        (q5) -- (r6);
        \draw [my2arrow= .6]
        (p4) -- (q5);
        \draw [my2arrow= .6]
        (p3) -- (q3);
        \draw [my2arrow= .6]
        (p4) -- (q4);
        \draw [my2arrow= .6]
        (p4) -- (r5);
        \draw [my2arrow= .6]
        (p5) -- (q5);
        \draw [my2arrow= .6]
        (q5) -- (r4);

		\draw[fill=red!20!white] (r1) circle (0.35);
		\draw[fill=red!20!white] (r2) circle (0.35);
		\draw[fill=red!20!white] (r3) circle (0.35);
		\draw[fill=red!20!white] (r4) circle (0.35);
		\draw[fill=red!20!white] (r5) circle (0.35);
		\draw[fill=red!20!white] (r6) circle (0.35);
		
		\draw[fill=green!20!white] (q2) circle (0.35);
		\draw[fill=green!20!white] (q3) circle (0.35);
		\draw[fill=green!20!white] (q4) circle (0.35);
		\draw[fill=green!20!white] (q5) circle (0.35);

		\draw[fill=blue!20!white] (p3) circle (0.35);
		\draw[fill=blue!20!white] (p4) circle (0.35);
		\draw[fill=blue!20!white] (p5) circle (0.35);
		
		\draw[fill=green!20!white] (5) circle (0.35);
		\draw[fill=blue!20!white] (11) circle (0.35);
		\draw[fill=red!20!white] (x) circle (0.35);
		\draw[fill=red!20!white](y) circle (0.35);
		
		\node at (x) {$x$};
		\node at (11) {$11$};
		\node at (5) {$5$};
		\node at (y) {$y$};
		
		\node at (r1) {};
		\node at (r2) {};
		\node at (r3) {};
		\node at (r4) {};
		\node at (r6) {};
		
		\node at (q2) {};
		\node at (q3) {};
		\node at (q4) {};
		
		\node at (p3) {};
		\node at (p4) {};
            \node at (p5) {};
		
		\node at (6.25, -2) {$\cal{I}$};
		\node at (6.25, -4) {$\cal{D}$};
		\node at (6.25, -6) {$\cal{O}$};
		\node at (-6.25, -2) {$\;$};
		
	\end{tikzpicture}
	\caption{This graph is simultaneously the prime graph complement of an $M_{11}$-solvable group ($x \mapsto 2, y \mapsto 3$) and an $M_{12}$-solvable group ($x \mapsto 3, y \mapsto 2$), with partial orientation and $3$-coloring as in the proof of \Cref{easy3color}.}
	\label{fig:M_construction}
\end{figure}

\begin{proof}
For the forward direction, assume that $\Gamma$ is the prime graph of an $M_{12}$-solvable group $G.$ By \Cref{Ms}, we have $2-p, 3-p, 5-p \notin \ol \Gamma$ for all $p \in \pi(G) \setminus \pi(T)$. In addition, $\ol \Gamma$ has a 3-coloring such that all neighbors of 11 outside of $\pi(T)$ have the same color, and $\ol \Gamma$ contains at most one triangle $\{3,5,11\}$. If this triangle does not exist, then (1) holds. Assume then that $\ol \Gamma$ possesses a $\{3,5,11\}$-triangle. In the case that $2-11\in \ol\Gamma,$ \Cref{Thanks TMK} shows that $11-p\notin \ol\Gamma(G)$ for $p\in \pi(G) \setminus \pi(T).$ Then, since $2-3,2-5 \notin \ol \Gamma(M_{12})$, we see that $\ol\Gamma$ satisfies (2). If, on the other hand, $2-11\notin \ol\Gamma,$ then (3) holds, completing the forward direction.

We now proceed with the backward direction. If $\Gamma$ satisfies (1), then we are done by \Cref{ThomasMichaelKeller}. If $\Gamma$ satisfies (2) or (3), then $\ol\Gamma\setminus\{a,b,c,d\}$ is $3$-colorable and triangle-free, so using techniques from \cite[Theorem 2.8]{2015} we can find a solvable group $N$ with $(|N|,|M_{12}|)=1$ and $\ol\Gamma(N)=\ol\Gamma\setminus\{a,b,c,d\}.$ If (2) holds, then $\ol\Gamma\cong \ol\Gamma(N\times T).$ Next, suppose $\Gamma$ satisfies (3) and let $X=\{a,b,c,d\}.$ We proceed with a straightforward application of \Cref{graphlizard} and information from \Cref{table:3}. Making assignments $a\mapsto 3,$ $b\mapsto 5,$ $c\mapsto 11,$ $d\mapsto 2$ gives a graph isomorphism from the subgraph of $\ol\Gamma$ induced by $X$ to $\vv\Gamma(E)$ where $E$ is the double cover $2.M_{12}.$ For each $v\in N(X)\setminus X,$ (3) states that $v$ is adjacent to $c$ but not $a,$ $b,$ or $d.$ Now consider the complex irreducible representations of $E$ described in \Cref{table:3}. Note that there is a representation in which order $2,3,$ and $5$ elements have fixed points but order $11$ elements do not. This representation obeys the conditions laid out in \Cref{graphlizard} with respect to our chosen graph isomorphism, so $\Gamma$ is the prime graph of an $M_{12}$-solvable group. This completes the proof.
\end{proof}

\subsection{The groups $G_2(3)$, $\PSL(3,4)$, and $U_4(3)$}

In this subsection we once again consider groups $T$ such that the subgraph of $\ol\Gamma(G)$ induced by $\pi(T)$ is a union of connected components. Unlike before, this induced subgraph will have more complicated structure.

First, we consider $G_2(3).$ It will be useful to note that $\ol \Gamma(G_2(3))$ is the complete graph on the vertices $\{2,3,7,13\}$ minus the edge $2-3$ (see \Cref{table:5}). In particular, the prime graph complement of a strictly $G_2(3)$-solvable group $G$ never contains the edge $2-3.$

\begin{lemma}\label{G2(3)}
    Let $G$ be a strictly $G_2(3)$-solvable group. Then $N(\pi(T)) \subseteq \pi(T)$. Moreover, $\ol \Gamma(G)$ is 3-colorable and has at most two triangles, which must be $\{2,7,13\}$- or $\{3,7,13\}$-triangles if they exist.
\end{lemma}

\begin{proof}
	Because the Sylow $2$- and $3$-subgroups of $G_2(3)$ do not satisfy the Frobenius criterion, \Cref{FC} gives $2-p, 3-p \notin \ol \Gamma(G)$ for all $p \in \pi(G) \setminus \pi(T)$. Furthermore, since $|M(G_2(3))| = 3$ and order 7 and 13 elements have fixed points in every complex irreducible representation of a perfect central extension of $G_2(3)$ (see \Cref{table:3}), \Cref{funny2} shows that $7-p, 13-p \notin \ol \Gamma(G)$ for all $p \in \pi(G) \setminus \pi(T)$. Then, we are done by \Cref{easy3color} and the fact that $\ol \Gamma(G_2(3))$ consists only of the triangles $\{2, 7, 13\}$ and $\{3, 7, 13\}$.
\end{proof}

\begin{theorem}
    Let $\Gamma$ be an unlabeled simple graph. Then $\Gamma$ is isomorphic to the prime graph of a $G_2(3)$-solvable group if and only if $\ol\Gamma$ is $3$-colorable and one of the following holds:
    \begin{enumerate}
        \item[(1)] $\ol \Gamma$ is triangle-free.
        \item[(2)] There exists a subset $X = \{a,b,c,d\} \subseteq V(\ol\Gamma)$ 
        such that $N(X) \subseteq X$, $\ol \Gamma \setminus X$ is triangle-free, and the subgraph induced by $X$ is not complete. 
    \end{enumerate}
\end{theorem}
\begin{proof}
    We first prove the forward direction. Suppose that $\Gamma$ is the prime graph of a $G_2(3)$-solvable group $G$. We may assume that  $G$ is strictly $G_2(3)$-solvable (otherwise (1) holds by \Cref{ThomasMichaelKeller}), so that the conclusion of \Cref{G2(3)} holds. In particular, $N(\pi(T)) \subseteq \pi(T)$ and the subgraph induced by $\pi(G_2(3))$ is not complete because $\ol \Gamma(G_2(3))$ is not complete. This establishes (2).

    For the reverse direction, if $\ol \Gamma$ is triangle-free and $3$-colorable, then $\Gamma$ is the prime graph of a solvable group by \Cref{ThomasMichaelKeller}. Otherwise, $\Gamma$ satisfies $(2)$ and we note that $\ol\Gamma \setminus \{a,b,c,d\}$ is $3$-colorable and triangle-free. Therefore, using techniques from \cite[Theorem 2.8]{2015}, we can find a solvable group $N$ with $(|N|,|T|)=1$ and $\ol\Gamma(N)=\ol\Gamma \setminus \{a,b,c,d\}.$ Now we split into cases for the subgraph of $\ol\Gamma$ induced by $X$, which we call $\Lambda.$ If $\Lambda \cong K_4\setminus \{e\},$ we have that $\Gamma\cong \Gamma(N\times G_2(3)).$ If $\Lambda$ is a triangle connected to a vertex via a single edge, we have $\Gamma \cong \Gamma(N\times \Aut(G_2(3))).$ Finally if $\Lambda$ is a triangle with an isolated vertex, $\Gamma \cong \Gamma(N\times C_3 \times G_2(3)).$ As we have found in each case a group $G$ such that $\Gamma\cong \Gamma(G),$ the result is established.
\end{proof}

Next, we shift our focus to the groups $\PSL(3,4)$ and $U_4(3)$. It will be useful to note that $\ol \Gamma(\PSL(3,4))$ is the complete graph on the primes $\{2, 3, 5, 7\}$ and $\ol \Gamma(U_4(3)) = \ol \Gamma(\PSL(3,4)) \setminus \{2-3\}$ (see \Cref{table:5}). The proofs that $\pi(\PSL(3,4))$ and $\pi(U_4(3))$ induce isolated subgraphs in their respective prime graph complements are identical; they are both established by the following lemma.

\begin{lemma}\label{isolated}
    Let $T$ be $\PSL(3,4)$ or $U_4(3)$ and let $\Gamma$ be the prime graph of a $T$-solvable group $G$. Then one of the following conditions hold:
    \begin{enumerate}
        \item[(1)] $\ol \Gamma$ is triangle-free and 3-colorable.
        \item[(2)] $N(\pi(T)) \subseteq \pi(T)$ and $\ol \Gamma \setminus \pi(T)$ is triangle-free and 3-colorable.
    \end{enumerate}
\end{lemma}

\begin{proof}
    We may assume that $G$ is strictly $T$-solvable. Because the Sylow 2- and 3-subgroups of $T$ do not satisfy the Frobenius criterion, we automatically have $2-p$, $3-p \notin \ol \Gamma$ for all $p \in \pi(G) \setminus \pi(\PSL(3,4)).$ 
    As noted in \Cref{table:3}, in every complex irreducible representation of a perfect central extension of $\PSL(3,4)$, some element of order 5 has fixed points. Hence, by \Cref{funny2}, we have $5-p \notin \ol \Gamma(G)$ for any $p \in \pi(G) \setminus \pi(T)$. Next, we apply \Cref{funny} with $r=7$ to get that for each $p \in \pi(G) \setminus \pi(T)$, either $7-p \notin \ol \Gamma(G)$ or $G$ has a section $V.E$ for some nontrivial elementary abelian $p$-group $V$ and perfect central extension $E$ of $T$. 
    
    Suppose first that $7-p \notin \ol \Gamma(G)$ for all $p \in \pi(G) \setminus \pi(T)$. Then we have $N(2,3,5,7) = \{2,3,5,7\}$, and if we let $K \cong N.T$ denote the subgroup granted by \Cref{2.4 2022}, we see that $\ol \Gamma(G) \setminus \pi(T)$ is contained in the 3-colorable and triangle-free graph $\ol \Gamma(N)$. This establishes (2).
    
    Now suppose there exists $p \in \pi(G) \setminus \pi(T)$ such that $7-p \in \ol \Gamma(G)$, so that $G$ has a section $V.E$ as described above. Write $E = A.T$ for some abelian group $A$. Consulting \Cref{table:3}, we see that if $6 \ndivides |A|$, \Cref{reps} implies that $7-p \notin \ol \Gamma(G)$, a contradiction. Hence $6 \divides |A|$, and it follows that 2 and 3 are completely isolated in $\ol \Gamma(G)$. In particular, the subgraph of $\ol \Gamma(G)$ induced by $\pi(T)$ is 3-colorable, triangle-free, and only connects to the remainder of the graph through the vertex 7. From \Cref{easy3color}, then, we see that $\ol \Gamma(G)$ is triangle-free and 3-colorable.
\end{proof}

The classification result for $\PSL(3,4)$ is the first in which $\ol \Gamma$ is not required to be 3-colorable.

\begin{theorem}
    Let $\Gamma$ be an unlabeled simple graph. Then $\Gamma$ is isomorphic to the prime graph of a $\PSL(3,4)$-solvable group if and only if one of the following holds:
    \begin{enumerate}
        \item[(1)] $\ol \Gamma$ is triangle-free and 3-colorable.
        \item[(2)] There exists a subset $X = \{a,b,c,d\} \subseteq V(\ol\Gamma)$ 
        such that $N(X) \subseteq X$ and $\ol \Gamma \setminus X$ is triangle-free and 3-colorable.
    \end{enumerate}
\end{theorem}

\begin{proof}
    The forward direction is handled by \Cref{isolated}. For the reverse direction, suppose $\Gamma$ satisfies (2) (if it satisfies (1), then we are done by the classification of the prime graphs of solvable groups). By the techniques established in \cite[Theorem 2.8]{2015}, we can find a solvable group $N$ such that $(|N|, |T|) = 1$ and $\ol \Gamma(N) \cong \ol \Gamma \setminus \{a,b,c,d\}.$ Now, we split into cases based on the isomorphism type of the subgraph induced by $X$, which we shall denote by $\Lambda$. If $\Lambda$ is the complete graph, then $\Gamma \cong \Gamma(T \times N)$. If $\Lambda$ is the complete graph minus one edge, then take a semidirect product $E_1 \cong \PSL(3,4) \rtimes C_2$ such that $\Lambda \cong \ol \Gamma(E_1)$ and note that $\Gamma \cong \Gamma(E_1 \times N)$. Next, if $\Lambda$ is a triangle connected to the remaining vertex via exactly one edge, then take a semidirect product $E_2 \cong \PSL(3,4) \rtimes C_2$ such that $\Lambda \cong \ol \Gamma(E_2)$ and observe that $\Gamma \cong \Gamma(E_2 \times N).$ The existence of the groups $E_1$ and $E_2$ can be verified using GAP \cite{GAP4} (for more details, see Section \ref{section:GAP}). Finally, if $\Lambda$ is an isolated triangle together with an isolated vertex, then $\ol \Gamma \cong \ol \Gamma(\PSL(3,4) \times N \times C_2)$. We have considered all graphs on four vertices that contain a triangle, so the proof is complete.
\end{proof}

We finish this subsection with a classification for $U_4(3).$ Note that $\ol \Gamma(U_4(3))$ is the complete graph on $\{2, 3, 5, 7\}$ minus the edge $2-3$ (see \Cref{table:5}).

\begin{theorem}
    Let $\Gamma$ be an unlabeled simple graph. Then $\Gamma$ is isomorphic to the prime graph of a $U_4(3)$-solvable group if and only if $\ol \Gamma$ is 3-colorable and one of the following holds:
    \begin{enumerate}
        \item[(1)] $\ol \Gamma$ is triangle-free.
        \item[(2)] There exists a subset $X = \{a,b,c,d\} \subseteq V(\Gamma)$ 
        such that $N(X) \subseteq X$, $\ol \Gamma \setminus X$ is triangle-free, and the subgraph induced by $X$ is not complete.
    \end{enumerate}
\end{theorem}

\begin{proof}
    Because $\pi(U_4(3))$ is triangle-free and 3-colorable, \Cref{isolated} establishes the forward direction. For the reverse direction, we may assume that $\Gamma$ satisfies (2). By the proof of \cite[Theorem 2.8]{2015}, we can find a solvable group $N$ such that $(|N|, |T|) = 1$ and $\ol \Gamma(N) \cong \ol \Gamma \setminus \{a,b,c,d\}.$ Now, like in the above theorem, we split into cases based on the isomorphism type of the subgraph $\Lambda$ induced by $X$. If $\Lambda$ is the complete graph minus one edge, then $\Gamma \cong \Gamma(U_4(3) \times N)$. If $\Lambda$ is a triangle connected to the remaining vertex via exactly one edge, then take a semidirect product $E \cong U_4(3) \rtimes C_2$ such that $\Lambda \cong \ol \Gamma(E)$ and observe that $\Gamma \cong \Gamma(E \times N).$ GAP \cite{GAP4} can once again be used to verify the existence of such an $E$ (for more details, see Section \ref{section:GAP}). Finally, if $\Lambda$ is an isolated triangle together with an isolated vertex, then $\ol \Gamma \cong \ol \Gamma(U_4(3) \times N \times C_2)$. There is no need to consider isomorphism types of $\Lambda$ that do not contain a triangle, so the proof is complete.
\end{proof}

\section{Outlook}
We have classified the prime graphs of $T$-solvable groups for all $K_4$ groups $T$ except for $\Sz(8)$, $\Sz(32)$, and the potentially infinite family of $\PSL(2,q)$. A natural direction to proceed is to finish the classification for $\Sz(8)$ and $\Sz(32)$. However, one must overcome certain challenges that were not considered in this paper. For $\Sz(8)$, one must modify the results of \Cref{section:preliminaries} to allow for $\pi(T) \ne \pi(\Aut(T))$. For $\Sz(32)$, there are memory limitations to overcome in GAP in order to find extensions of $\Sz(32)$ realizing each possible configuration of the graph complement. For the infinite family of $\PSL(2,q)$, it is likely possible to solve them one at a time, but a more general technique would be needed to solve the family completely. In all of these cases, many of the Sylow subgroups do satisfy the Frobenius Criterion, making these classifications more difficult.

It would also be important to consider groups that have multiple nonisomorphic $K_4$ composition factors. We conjecture that no new graph structures can occur, as is the case with $K_3$ groups in \cite{2022}. Yet another approach is to find versions of our classification results for strictly $T$-solvable groups. This would require a more careful study of the triangle-free and $3$-colorable case.

Additionally, the question of Maslova's conjecture in \cite[Section 5]{Maslova} about whether a prime graph complement can be triangle-free but not $3$-colorable remains unresolved. However, none of the graphs appearing in our classification results have allowed for this structure. It may be useful to study the prime graph of $A_n$-solvable groups, $\Sz(2^{2n+1})$-solvable groups, or $K_n$-solvable groups for larger $n$ in order to realize a triangle-free but not $3$-colorable graph as a prime graph complement. 

\section{A note about GAP}\label{section:GAP}

The GAP system \cite{GAP4} was used extensively throughout this project. It contains (and has the ability to construct) many finite groups, including the $K_4$ groups and many helpful $K_4$-solvable groups used in this paper. The contents of the tables of Appendix B, including information about Sylow subgroups, prime graphs, and fixed points of complex irreducible representations were all computed with GAP. 


The techniques introduced in Section \ref{section:preliminaries} demonstrate the necessity of constructing the perfect central extensions of some group $G$, as well as being able to compute the fixed points of elements in complex irreducible representations of these extensions. This is straightforward for groups of smaller order, but certain computational issues occur for groups with comparatively larger order. For example, for $U_4(3)$, constructing the perfect central extensions in GAP proved difficult. Instead, we needed to use the {\it ATLAS of Finite Groups} \cite{atlas} to find the characters and manually apply the formula in Table \ref{table:3}. Afterward, Dr. Alex Hulpke was able to provide a permutation representation of $36.U_4(3)$, which was quite helpful in verifying our results. Similarly, we were unable to construct the irreducible characters of $\PSL(3,17)$ in GAP. Instead, we needed to use information from prior work by Simpson and Frame in \cite{PSL}, which lists the character tables for $\PSL(3,q)$, and we manually applied the formula as well. We expect similar issues to occur for further classifications, especially as the sizes of groups increase. 

Further details and the exact algorithms used are included on this \href{https://github.com/abiteofdata/K4-Groups}{GitHub page}. The page gives detailed documentation of the techniques used throughout,  constructions of important groups in Section \ref{section:allresults}, and important techniques used to reduce computation time. The reader is encouraged to verify the validity of the statements in the tables and the existence of the groups.

\appendix

\section{}

In \cite{REU} and \cite{A5}, the prime graphs of $A_5$-solvable groups were classified. Then \cite{2022} classified the prime graphs of all the remaining $T$-solvable groups for $K_3$ groups $T$, and the current paper extended this to (in some sense) most $K_4$ groups. Both \cite{2022} and the current paper use a uniform standard of presenting these classification results, whereas the presentation in \cite{A5} differs markedly from that. But $A_5$ is a very prominent group, and it thus makes sense to revisit $A_5$ here and present the classification of the prime graphs of $A_5$-solvable groups, including a full proof, in the same way as for the other groups. This will also serve two other worthwhile purposes: not only will it demonstrate the power of the techniques developed in Section \ref{section:preliminaries} in that the new proof will be much shorter than the original proof, but it will also correct a serious flaw in the proof of an important result, \cite[Lemma 4.5]{A5}, which appears below as Lemma \ref{kmain}. By necessity, the corrected proof makes use of some deep results in the first edition of Aschbacher's well-known book {\it Finite Group Theory} \cite{aschbacher}. In hindsight, it seems noteworthy that the case of $A_5$-solvable groups is -- not on the computational end, but on the theoretical side -- the hardest of all $T$-solvable groups for $K_3$ or $K_4$ groups $T$ studied up to this point.


\begin{lemma}\label{kmain}
Let $G$ be a finite group and $N$ a normal subgroup of $G$ such that $G/N\cong A_5$ and $N$ is solvable. Also assume that 5 does not divide the order of $N$. Now let $V$ be a finite $G$-module over $\F_q$, where $q$ is a prime not dividing $|G|$. Suppose that here exist no elements of order $3q$ in the semidirect product $GV$. Then there also exist no elements of order $5q$ in $GV$.
\end{lemma}

\begin{proof}
Suppose the result is not true and let the pair of groups $(G, V)$ be a minimal counterexample, i.e., $|G|+|V|$ is minimal. First, observe that $G$ is $\{2,3,5\}$-separable and thus has a Hall $\{2,3,5\}$-subgroup $H$. If $H<G$, then clearly $HV$ is a smaller counterexample, so we may assume that $G$ is a $\{2,3,5\}$-group. Moreover, by the minimality condition, we may assume that the action of $G$ on $V$ is faithful and irreducible.

Our first goal is to show that $|N|$ is not divisible by 3. Assume that 3 divides $|N|$. Then there exist normal subgroups $S$ and $T$ of $G$ with $T<S\leq N$ such that $S/T$ is a chief factor of $G$ of order divisible by 3, and we may also choose $S,T$ such that 3 does not divide $|N/S|$. Since $GV$ has no elements of order $3q$, any Sylow 3-subgroup acts Frobeniusly on $V$ and thus is cyclic. This shows that $S/T$ is cyclic of order 3.

Now $(G/T)/C_{G/T}(S/T)$ has order 1 or 2. If it has order 2, then $G/C_G(S/T)$ has order 2 and so $C_G(S/T)$ acting on $V$ will be a smaller counterexample, contradicting our minimal choice of the pair $(G, V)$. Thus it has order 1, meaning that $S/T$ is a central chief factor and that $N/T$ is 3-nilpotent. So $G$ has a normal subgroup $R$ such that $G/R$ is a central extension of $A_5$ by $C_3$ (the cyclic group of order 3). Since the Schur multiplier of $A_5$ is 2, it is clear that this extension must be split, and hence the Sylow 3-subgroups of $G$ cannot be cyclic, a contradiction. So we know that 3 does not divide the order of $N$, and hence $N$ is a 2-group. In particular, $F(G)=N$, where $F(G)$ is the Fitting subgroup of $G$.

Next, let $M$ be an abelian normal subgroup of $G$. Let $R \leq G$ with $|R| = 3$. Then the Zassenhaus decomposition
for coprime actions tells us that $M=[R,M]\times C_M(R)$, and $R$ acts Frobeniusly on $[R,M]$. Hence $R$ acts Frobeniusly on $[R,M]V$ which thus is nilpotent. Therefore $[R,M]$  acts trivially on $V$, but since the action of $G$ on $V$ is faithful, this forces $[R,M]=1$. Since $R$ was chosen arbitrarily of order 3, this shows that $C_G(M)$ contains all elements of order 3 in $G$, and since $3$ divides $|A_5|$ but not $|N|$, $C_G(M)$ has a factor group isomorphic to $A_5$. So by minimality, it follows that $C_G(M)=G$ and hence $M\leq Z(G)$. As a consequence, $M$ is cyclic. This shows that all abelian normal subgroups of $G$ are cyclic and central in $G$. 

So every characteristic abelian subgroup of $N$ is cyclic and central in $G$, and thus by \cite[Corollary 1.4 (vii)]{Manz_Wolf_1993} we obtain that $N$ is a central product of a (possibly cyclic) extraspecial $2$-group $E$ and a cyclic $2$-group which is the center of $G$. If $N$ is cyclic, then $N=Z(G)$, and hence $G$ is a central extension of $A_5$. By minimality, this extension is non-split, and thus either $G=A_5$ or $G=\mbox{SL}(2,5)$. In both cases, one can check Table \ref{table:4} that in all complex representations, if elements of order 3 act Frobeniusly, then so do elements of order 5, and this carries over to the action on $V$, since $(|G|,|V|)=1$. This is a contradiction. 

So $N$ is not cyclic. Thus $E>Z(E)$ and $N=EZ(G)$  with $E\cap Z(G)=Z(E)=:Z$ being cyclic of order 2. Now let $R\leq G$ such that $|R|=3$. Then $N=[R,N]C_N(R)$. Since $R$ acts Frobeniusly on $V$, from Flavell's result \cite{Flavell} we know that $[R,N]$ is a nonabelian special 2-group. Also note that by coprime action we have $[R,N]=[N,R]=[N,R,R]$. Moreover, since $[R, N]$ is nonabelian, it contains $Z=N'$, and hence $[R,N]$ is normal in $N$. Hence every characteristic abelian subgroup of $[R,N]$ is normal in $G$ and thus, as seen above, central in $G$. Therefore we can apply \cite[(24.7)]{aschbacher} to the action of $R$ on $[R,N]$ which then shows that $[R,N]$ is extraspecial. Write $|[R,N]|=2^{2m+1}$. Then by \cite[(36.1(1))]{aschbacher} we know that $m=1$. This implies that
$[R,N]\cong Q_8$.

Now write $|E|=2^{2n+1}$ and put $W=N/Z(G)$. 
Since $A_5$ can be generated by
two (suitable) 3-cycles, there exist two subgroups $R_1$, $R_2$ of $G$ which are of order 3 such that $G=\langle R_1, R_2, N\rangle$.
By the previous paragraph, we know that  $[R_i,N]\cong Q_8$ for $i=1,2$. Hence we have $\dim (C_W(R_i))=\dim(W)-2$. Thus 
\begin{align*}
    \dim(W) &\geq \dim(C_W(R_1)+C_W(R_2)) \\
    &= \dim(C_W(R_1))+\dim(C_W(R_2))  - \dim(C_W(R_1)\cap C_W(R_2))
\end{align*}

Since $C_W(R_1)\cap C_W(R_2)=C_W(G/N)=C_W(G)$, we further obtain
    \[\dim(W)\geq \dim(W)-2 + \dim(W)-2 -\dim(C_W(G))\] 
and therefore $\dim(C_W(G))\geq \dim(W)-4$. As $A_5$ is simple, $G/N$ acts faithfully on $X:=W/C_W(G)$. Now since elements of order 5 of $G/N$ act faithfully on $X$, we see that if $S$ is a Sylow 5-subgroup of $G$, then $X$ must be irreducible as $S$-module and of dimension 4 (since $4$ is the smallest integer $n$ such that $5|2^n-1$, and since $\dim(X)\leq 4$), and $S$ acts Frobeniusly on $X$.  Hence altogether $X$ is a faithful, irreducible $G/N$-module, and so if $M\leq N$ is the inverse image of $C_W(G)$ in $N$ (i.e., $M/Z(G)=C_W(G))$, then $M$ is normal in $G$ and $N/M\cong X$ is a faithful, irreducible $G/N$-module over $\F_2$ of dimension 4. By Burichenko's results in 
\cite{burichenko} and using the fact that $A_5\cong \PSL(2,5)$, we conclude that $G/M$ splits over $N/M$. So let $K\leq G$ such that $M\leq K$ and $K/M$ is a complement of $N/M$ in $G/M$. Then $K$ has a factor group isomorphic to $A_5$ and acts faithfully on $V$ and thus the pair $(K,V)$ is a smaller counterexample, contradicting the minimal choice of $(G,V)$.

This final contradiction shows that the case of $N$ being noncyclic is impossible, which concludes the proof of the lemma.
\end{proof}

\begin{lemma}\label{Wa5}
Let $G$ be a strictly $A_5$-solvable group with a subgroup $K\cong N.A_5$ such that $3$ or $5$ divides $|N|$ and $\pi(G)=\pi(K).$ Then $\ol\Gamma(G)$ is $3$-colorable and triangle-free.
\end{lemma}
\begin{proof}
Note that $A_5$ has solvable subgroups isomorphic to $D_{10}$ and $A_4$. Thus $G$ has solvable subgroups $K_1 \cong N.D_{10}$ and $K_2 \cong N.A_4.$ If $5$ divides $|N|,$ then $\ol\Gamma(G)$ is a subgraph of $\ol\Gamma(K_2)$ and if $3$ divides $|N|,$ then $\ol\Gamma(G)$ is a subgraph of $\ol\Gamma(K_1).$ As $K_1$ and $K_2$ are solvable, we conclude that $\ol\Gamma(G)$ is $3$-colorable and triangle-free.
\end{proof}

\begin{lemma}\label{4.6 2022}
Let $G$ be a strictly $A_5$-solvable group and $K\cong N.A_5$ the subgroup granted by \Cref{2.4 2022}. Suppose $\ol\Gamma(G)$ has a $\{p,3,5\}$ triangle for some prime $p\neq 2.$ For any $q\in \pi(G)\setminus \pi(A_5),$ if $3-q\in \ol\Gamma(G)$ then $5-q\in \ol\Gamma(G).$
\end{lemma}

\begin{proof}
This follows directly from \Cref{kmain}. As $\ol\Gamma(G)$ has a triangle, \Cref{Wa5} shows $5\ndivides |N|.$ Consider a chief series 
\[1=K_0 \unlhd K_1 \unlhd \cdots \unlhd K_n=K\]
such that $K_{n-1}=N$ and $K_n/K_{n-1}\cong A_5.$ Take any $q$ such that $3-q\in \ol\Gamma(G),$ that is $G$ and thus $K$ have no order $3q$ elements. For any $i$ such that $K_{i+1}/K_i=Q_i$ is an elementary abelian $q$-group, let $M_i$ be a Hall $\{2,3,5,p\}$-subgroup of $K/K_i.$ Considering the action of $M_i$ on $Q_i,$ \Cref{kmain} shows $Q_i\rtimes M_i$ has no order $5q$ elements. Thus $K$ has no order $5q$ elements, showing $\ol\Gamma(G)$ has an edge $5-q.$
\end{proof}

\begin{lemma}\label{A5doublecoverfun}
Let $G$ be strictly $A_5$-solvable, $K\cong N.A_5$ the subgroup granted from \Cref{2.4 2022}, and $E$ the double cover $\text{SL}(2,5)\cong 2.A_5.$ If $2-p\in \ol\Gamma(G)$ for some $p\in \pi(G)\setminus \pi(A_5),$ then there exists a solvable $2'$-group $L$ such that $K\cong L.E.$ In particular, $2-3,2-5\notin \ol\Gamma(G)$.
\end{lemma}

\begin{proof}
Let $Q\in \text{Syl}_2(K)$ and $P\in \text{Syl}_p(K).$ Note that $NQ$ is solvable, and we may assume that $PQ$ is a Hall $\{2,q\}$-subgroup of $NQ.$ Since $2-p\in \ol\Gamma(G),$ we may assume that $PQ$ is Frobenius or $2$-Frobenius of type $(2,p)$ or $(2,p,2)$ respectively. But the latter is not possible, since in a $2$-Frobenius group the top Frobenius complement must be cyclic, but $PQ$ has a Klein-$4$ quotient. Hence \cite[Corollary 6.17]{FGT} shows that $Q$ is either cyclic or generalized quaternion. Consider the Sylow-$2$ subgroup $Q_0=Q\cap N$ of $N.$ Because $Q/Q_0 \in \text{Syl}_2(A_5)$ is isomorphic to $V_4,$ $Q$ must in fact be generalized quaternion. Now, recall that generalized quaternion groups have only one normal subgroup of index $4,$ and this subgroup is always cyclic. Thus $Q_0$ is cyclic of order $2^m$ for some $m\geq 1.$

By Burnside's $p$-complement theorem, we can write $N=L.Q_0,$ where $L$ is a normal Hall $2'$ subgroup of $N.$ In fact, $L$ is characteristic in $N$ and thus normal in $G.$ Observe that $N/L\cong Q_0$ is abelian, and thus $(K/L)(N/L)\cong A_5$ acts on $N/L\cong Q_0$ by conjugation. However, the simplicity of $A_5$ implies that this action is trivial. Hence all of $K/L\cong Q_0.A_5$ acts trivially on $Q_0.$ In particular, $Q_0$ belongs to the center of some Sylow $2$-subgroup of $K/L,$ which is isomorphic to $Q$ because $2\ndivides |L|.$ We showed above that $Q$ is generalized quaternion, so this center has order $2$; since $Q_0$ is non-trivial, it follows that $Q_0\cong C_2.$ Therefore $Q\cong Q_8,$ the quaternion group of order $8$. It is well-known that $Q_8$ does not split over $C_2,$ thus our extension $Q_0.A_5$ cannot split either. Therefore $Q_0.A_5$ is the double cover of $A_5,$ establishing the first statement of the lemma. In particular, $Z(Q_0.A_5)=Q_0,$ so the order $3$ and order $5$ elements of $Q_0.A_5$ commute with at least one order $2$ element, showing $2-3,2-5\notin \ol\Gamma(G).$
\end{proof}

\begin{lemma}\label{no3p}
Let $G$ be a strictly $A_5$-solvable group such that $2-3\in \ol\Gamma(G)$ or $2-5\in \ol\Gamma(G)$. Then $3-p\notin \ol\Gamma(G)$ for any $p\notin \pi(A_5).$
\end{lemma}

\begin{proof}
We first note that $(3,M(A_5))=1. $ Thus for each $p\in \pi(G)\setminus \pi(A_5)$, \Cref{funny} shows that either $3-p\notin \ol\Gamma(G)$ or $G$ has a section $V.E$ where $E$ is a perfect central extension of $A_5$ and $V$ is a nontrivial elementary abelian $p$-group.

To prove the contrapositive, suppose there exists $p\in \pi(G)\setminus\pi(A_5)$ such that $3-p\in \ol\Gamma(G)$ so that $G$ has such a section $V.E.$ By Table \ref{table:4}, we have two cases. If $E$ is not the double cover $2.A_5\cong SL(2,5),$ then \Cref{reps} implies $3-p\notin \ol\Gamma(G),$ a contradiction. If $E$ is the double cover $2.A_5\cong \text{SL}(2,5),$ then $2-3,2-5\notin \ol\Gamma(G).$
\end{proof}

\begin{lemma}\label{3p2p5p}
Let $G$ be a strictly $A_5$-solvable group and $K\cong N.A_5$ the subgroup granted from \Cref{2.4 2022}. If $3-p\in \ol\Gamma(G)$ and $2-q\in \ol\Gamma(G)$ for some primes $p,q\in \pi(G)\setminus \pi(A_5).$ Then one of the following holds:
\begin{enumerate}
\item $\ol\Gamma(G)$ is triangle-free and $3$-colorable.
\item $2-p\in \ol\Gamma(G)$ and $5-p\in \ol\Gamma(G)$
\end{enumerate}
\end{lemma}

\begin{proof}
By \Cref{A5doublecoverfun}, $K$ can be written as $L.E$ where $L$ is a $2'$ group and $E$ is the double cover $2.A_5\cong \text{SL}(2,5).$ If $(|L|,15)\neq 1$ then by \Cref{Wa5}, $(1)$ is satisfied. Thus we may proceed assuming $(|L|,15)=1.$ Note that $K$ is $\{2,3,5,p\}$-separable, so by \cite[Theorem 3.20]{FGT} we can take a Hall-$\{2,3,5,p\}$ subgroup $H\leq K.$ Then, according to \cite[Lemma 2.1]{2022}, we have $H\cong P.E$ where $P$ is a $p$-group. By the Schur-Zassenhaus theorem, $H\cong P\rtimes E.$ 

Now suppose, for the sake of contradiction, that $2-p\notin \ol\Gamma(G)$ or $5-p\notin \ol\Gamma(G),$ giving an element $x$ of order $2p$ or $5p$ in $H.$ Replacing $x$ with some conjugate of itself if necessary, we may assume that $x^p \in E.$ Now let $P_0\leq P$ be a normal subgroup of $H$ minimal with respect to the condition that it contains $x^2$ if $|x|=2p$ and $x^5$ if $|x|=5p.$ Passing $H$ to a smaller quotient if necessary, we may assume that $P_0$ is a minimal normal subgroup of $H$ and thus elementary abelian. Now, $x^p$ is an element of $E$ fixing $x^2$ if $|x|=2p$ and $x^5$ if $|x|=5p.$ In other words, we have an order $2$ or $5$ element of $E$ fixing an order $p$  element of $P_0.$ By \Cref{reps} and Table \ref{table:4}, some order $5$ element of $E$ also has fixed points in $P_0,$  which means $3-p\notin \ol\Gamma(G),$ a contradiction.
\end{proof}

\begin{lemma}\label{a53color}
Let $\Gamma$ be the prime graph of some strictly $A_5$-solvable group $G.$ Then one of the following holds:
\begin{enumerate}
 \item $\ol\Gamma$ is triangle-free and $3$-colorable.
\item $\ol\Gamma$ has a $3$-coloring where all neighbors of $2,3,$ and $5$ not in $\pi(A_5)$ have the same color.
\end{enumerate}
\end{lemma}

\begin{proof}
By \Cref{2.4 2022}, $G$ has a subgroup of the form $K\cong N.A_5$ where $N$ is solvable and $\pi(G)=\pi(K).$ If $3$ or $5$ divides $|N|,$ then $\ol\Gamma(G)$ is $3$-colorable and triangle-free by \Cref{Wa5}, so we proceed with the proof assuming otherwise. We induce a partial orientation $\vv{\Gamma}$ on edges of $\ol\Gamma$ according to the orientation they are given in $\vv\Gamma (N.C_3)$ or $\vv\Gamma (N.C_5).$ To see that this is well-defined, note that all edges of $\vv\Gamma(N.C_5)$ are either contained in $\vv\Gamma(N.C_3)$ or are of the form $5\rightarrow p.$ Likewise, all edges of $\vv\Gamma(N.C_3)$ are either contained in $\vv\Gamma(N.C_5)$ or are of the form $3\rightarrow p.$ Since $3,5\ndivides |N|,$ all edges shared by $\vv\Gamma(N.C_5)$ and $\vv\Gamma(N.C_3)$ are oriented according to $\vv\Gamma(N).$ Thus the orientations given by $\vv\Gamma(N.C_3)$ and $\vv\Gamma(N.C_5)$ align where they overlap in $\ol\Gamma(G).$ If $3-5\in \ol\Gamma(G),$ we additionally define the orientation $3\rightarrow 5.$ By \Cref{no 2-paths}, vertices adjacent to $3$ or $5$ have zero out-degree in $\ol\Gamma.$ Thus $3$ and $5$ have no outgoing $2$-paths, and it follows that $\vv\Gamma$ has no $3$-paths.

We assign a $3$-coloring to vertices of $\pi(N)\cup\{3,5\}$ as follows: label all vertices in $\pi(N)\cup\{3,5\}$ with zero out-degree as $\mathcal{I},$ all vertices in $\pi(N)\cup\{3,5\}$ with zero in-degree and non-zero out-degree as $\mathcal{O},$ and all vertices in $\pi(N)\cup\{3,5\}$ with nonzero in- and out-degree with $\mathcal{D}.$ As $V(\ol\Gamma\setminus\{2\})\subseteq \pi(N)\cup \{3,5\}$, this procedure induces a 3-coloring on the subgraph of $\ol\Gamma$ induced by $V(\ol\Gamma)\setminus\{2\}$. 

In the case that vertex $2$ is only adjacent to vertices in $\{3,5\},$ the $3$-coloring of the vertices  $V(\ol\Gamma\setminus\{2\})$ can easily be extended to $\ol\Gamma.$ Additionally, all neighbors of $3$ and $5$ outside of $\pi(A_5)$ will belong to the same color $\mathcal{I},$ as desired.

In the case that $2-p\in \ol\Gamma$ for some $p\in \pi(G)\setminus \pi(A_5),$ from \Cref{A5doublecoverfun} we have $2-5,2-3\notin \ol\Gamma.$ Additionally $2\divides |N|$ and the Hall $\{2,p\}$-subgroups of $N$ must be Frobenius of type $(2,p),$ giving the orientation $2\rightarrow p$ in $\vv\Gamma.$ Additionally, \Cref{no 2-paths} shows vertices in $N(2)\setminus\{3,5\}$ have zero out-degree in $\vv\Gamma(N)$ and thus $\vv\Gamma.$ Thus all vertices of $N(2)\setminus\{3,5\}$ belong to the same color $\mathcal{I}.$
Thus our above $\mathcal{O},\mathcal{D},\mathcal{I}$ coloring gives a $3$-coloring of $\ol\Gamma$ where all neighbors of $2,3,$ and $5$ have the same color $\mathcal{I}$. 
\end{proof}

\begin{theorem}\label{A5}
    Let $\Gamma$ be an unlabeled simple graph. Then $\Gamma$ is isomorphic to the prime graph of an $A_5$-solvable group if and only if one of the following is satisfied:
    \begin{enumerate}
    \item $\ol\Gamma$ is triangle-free and $3$-colorable
    \item There exist vertices $a,b,c$ in $\ol\Gamma$ and a $3$-coloring for which all vertices adjacent to but not included in $\{a,b,c\}$ have the same color. Additionally, $\ol\Gamma$ has at least one triangle and one of the following holds:
    \begin{enumerate}
    \item [(2.1)]$\ol\Gamma$ has exactly one triangle $\{a,b,c\}.$ Vertices $a$ and $b$ are not adjacent to any other vertices.
    \item[(2.2)] All triangles in $\ol\Gamma$ are about vertices $\{b,c,d\}$ for vertices $d\neq a.$ Additionally, $a-c,a-b\notin \ol\Gamma,$ and $N(b)\setminus \{c\}\subset N(a)\cap N(c).$ 
    \item [(2.3)] All triangles in $\ol\Gamma$ are about vertices $\{b,c,d\}$ for vertices  $d\neq a.$ Vertex $a$ is isolated and $N(b)\setminus\{c\}\subset N(c).$
    \end{enumerate}
    \end{enumerate}
\end{theorem}

\begin{proof}
We will first prove the forward direction. Suppose that $\Gamma$ is the graph of an $A_5$-solvable group $G.$ By \Cref{ThomasMichaelKeller}, we may assume that $G$ is strictly $A_5$-solvable. Then by \Cref{a53color}, $\ol\Gamma$ is $3$-colorable. If $\ol\Gamma$ is also triangle-free, then $\Gamma$ satisfies $(1)$ and we are done. Thus we will proceed assuming $\ol\Gamma$ has at least one triangle. Let $K\cong N.A_5$ be the subgroup of $G$ granted by \Cref{2.4 2022}. Taking solvable subgroups $K_1$ and $K_2$ as in \Cref{Wa5}, it is clear that $\ol\Gamma \setminus \{3\}$ and  $\ol\Gamma \setminus \{5\}$ are triangle-free. Thus any triangle in $\ol\Gamma$ must contain the edge $3-5.$ Additionally, \Cref{a53color} gives that $\ol\Gamma$ has a $3$-coloring in which all vertices adjacent to but not including $\{2,3,5\}$ have the same color.

In the case that $\ol\Gamma$ has a $\{2,3,5\}$-triangle, \Cref{no3p} shows that $3-p\notin \ol\Gamma$ for all $p\in \pi(G) \setminus \pi(A_5)$ and \Cref{A5doublecoverfun} shows that $2-p\notin \ol\Gamma$ for all $p\in \pi(G) \setminus \pi(A_5).$ Thus in this case, $\ol\Gamma$ satisfies (2.1).

In the case that $\ol\Gamma$ has a $\{3,5,p\}$-triangle for some prime $p\neq 2,$ \Cref{no3p} gives $2-3,2-5\notin \ol\Gamma.$ 
If $2-r\notin \ol\Gamma$ for all primes $r\in \pi(G) \setminus \pi(A_5),$ then \Cref{4.6 2022} shows that $N(3)\setminus \{5\}\subset N(5)$, so (2.3) is satisfied. Else if $2-r\in \ol\Gamma$ for some prime $r\in \pi(G) \setminus \pi(A_5),$ note that by \Cref{3p2p5p}, we have $ 2-s,5-s\in \ol\Gamma$ whenever $3-s\in \ol\Gamma$. Hence $N(3)\setminus \{5\}\subset N(2)\cap N(5),$ and $\ol\Gamma$ satisfies (2.2). This completes the forwards direction. 

For the backwards direction, if $\ol\Gamma$ is triangle-free and $3$-colorable, then $\Gamma$ is the prime graph of a solvable group by \Cref{ThomasMichaelKeller}. Hence we need only consider $\Gamma$ satisfying the claims (2.1),(2.2), and (2.3). These cases are a straightforward application of \Cref{graphlizard} and the information from Table \ref{table:4}. In each case, we let $X=\{a,b,c\}$ and make assignments $a\mapsto 2,$ $b\mapsto 3,$ and $c\mapsto 5.$

First, suppose (2.1) holds. We have a graph isomorphism from the subgraph of $\ol\Gamma$ induced by $X$ to $\ol\Gamma(A_5).$ For each $v\in N(X)\setminus X,$ our hypothesis states that $v$ is adjacent to $c$ but not $a$ or $b.$ Note that there is a complex irreducible representation of $A_5$ in which order $2$ and $3$ elements have fixed points but order $5$ elements do not. This representation obeys the conditions laid out in \Cref{graphlizard} with respect to our chosen graph isomorphism, so $\Gamma$ is the prime graph of an $A_5$-solvable group.

Next, suppose (2.2) holds. We have a graph isomorphism from the subgraph of $\ol\Gamma$ induced by $X$ to $\ol\Gamma(\text{SL}(2,5)).$ For each $v\in N(X)\setminus X,$ our hypothesis states that one of the following holds: (1) $v$ is adjacent to $a$ but not $b$ or $c;$ (2) $v$ is adjacent to $c$ but not $a$ or $b$; (3) $v$ is adjacent to $a$ and $c$ but not $b$; (4) $v$ is adjacent to $a,$ $b,$ and $c.$ As demonstrated in Table \ref{table:4}, $\text{SL}(2,5)$ has complex irreducible representations respecting these edge conditions. Thus $\Gamma$ is the prime graph of an $A_5$-solvable group.

Finally, suppose (2.3) holds. We have a graph isomorphism from the subgraph of $\ol\Gamma$ induced by $X$ to $ \ol\Gamma(C_2.S_5)$. For each $v\in N(X)\setminus X,$ our hypothesis states that either $v$ is adjacent to $c$ but not $b$ or $v$ is adjacent to $c$ and $b.$ Once again Table \ref{table:4} gives complex irreducible representations of $C_2.S_5$ (a double cover of $S_5$) corresponding to these relations. Thus $\Gamma$ is the prime graph of an $A_5$-solvable group and our proof is complete.
\end{proof}

\newpage

\section{}
\begin{table}[H]
\centering
\begin{tabular}{|c||c|c|c|c|c|}
    \hline
    $G$ & $|G|$ & $|\Out(G)|$ & $\ol \Gamma(G)$ & $|M(G)|$ & Relevant Subgroups\\ \hline \hline
    $A_7$ & $2^3 \cdot 3^2 \cdot 5 \cdot 7$ & $2$ & $\kyle$ & 6 & $A_6$, $\PSL(2,7)$\\
    $A_8$ & $2^6 \cdot 3^2 \cdot 5 \cdot 7$ & $2$ & $\spoon$ & 2 & $A_7$\\
    $A_9$ & $2^6 \cdot 3^4 \cdot 5 \cdot 7$ & $2$ & $\stargraph$ & 2 & $A_7$\\
    $A_{10}$ & $2^7 \cdot 3^4 \cdot 5^2 \cdot 7$ & $2$ & $\Aten$ & 2 & N/A \\
    $M_{11}$ & $2^4 \cdot 3^2 \cdot 5 \cdot 11$ & $1$ & $\kyle$ & 1 & $C_{11} \rtimes C_5$\\
    $M_{12}$ & $2^6 \cdot 3^3 \cdot 5  \cdot 11$ & $2$ & $\spoon$ & 2 & $M_{11}$\\
    $J_2$ & $2^7 \cdot 3^3 \cdot 5^2 \cdot 7$ & $2$ & $\stargraph$ & 2 & N/A \\
    $\PSL(3,4)$ & $2^6 \cdot 3^2 \cdot 5 \cdot 7$ & $12$ & $\kfour$ & 48 &$A_6, \PSL(2,7)$ \\
    $\PSL(3,5)$ & $2^5 \cdot 3 \cdot 5^3 \cdot 31$ & $2$ & $\spoon$ & 1 & $C_{31} \rtimes C_3$\\
    $\PSL(3,7)$ & $2^5 \cdot 3^2 \cdot 7^3 \cdot 19$ & $6$ & $\spoon$ & 3 & N/A \\
    $\PSL(3,8)$ & $2^9 \cdot 3^2 \cdot 7^2 \cdot 73$ & $6$ & $\spoon$ & 1 & $C_{73} \rtimes C_3$ \\
    $\PSL(3,17)$ & $2^9 \cdot 3^2 \cdot 17^3 \cdot 307$ & $2$ & $\spoon$ & 1 & $C_{307} \rtimes C_3, (C_{17} \times C_{17}) \rtimes C_6$\\
    $\PSL(4,3)$ & $2^7 \cdot 3^6 \cdot 5 \cdot 13$ & $4$ & $\spoon$ & 2 & $\PSL(3,3)$\\
    $O_5(4)$ & $2^8 \cdot 3^2 \cdot 5^2 \cdot 17$ & $4$ & $\stargraph$ & 1 & N/A \\
    $O_5(5)$ & $2^6 \cdot 3^2 \cdot 5^4 \cdot 13$ & $4$ & $\stargraph$ & 2 & N/A \\
    $O_5(7)$ & $2^8 \cdot 3^2 \cdot 5^2 \cdot 7^4$ & $2$ & $\stargraph$ & 2 & N/A \\
    $O_5(9)$ & $2^8 \cdot 3^8 \cdot 5^2 \cdot 41$ & $4$ & $\stargraph$ & 2 & N/A \\
    $O_7(2)$ & $2^9 \cdot 3^4 \cdot 5 \cdot 7$ & $1$ & $\stargraph$ & 2 & $A_7, A_8$\\
    $O_8^+(2)$ & $2^{12} \cdot 3^5 \cdot 5^2 \cdot 7$ & $6$ & $\stargraph$ & 4 & N/A\\
    $\Sz(8)$ & $2^6 \cdot 5 \cdot 7 \cdot 13$ & $3$ & $\kfour$ & 4 & $D_{10}, D_{14}, D_{26}$\\
    $\Sz(32)$ & $2^{10} \cdot 5^2 \cdot 31 \cdot 41$ & $5$ & $\kfour$ & 1 & $D_{50}, D_{62}, D_{82}$\\
    $U_3(4)$ & $2^6 \cdot 3 \cdot 5^2 \cdot 13$ & $4$ & $\spoon$ & 1 & $C_{13} \rtimes C_3$\\
    $U_3(5)$ & $2^4 \cdot 3^2 \cdot 5^3 \cdot 7$ & $6$ & $\spoon$ & 3 & $A_6$ \\
    $U_3(7)$ & $2^7 \cdot 3 \cdot 7^3 \cdot 43$ & $2$ & $\spoon$ & 1 & $C_{43} \rtimes C_3$ \\
    $U_3(8)$ & $2^9 \cdot 3^4 \cdot 7 \cdot 19$ & $18$ & $\spoon$ & 3 & $\PSL(2,7), A_6, A_7$\\
    $U_3(9)$ & $2^5 \cdot 3^6 \cdot 5^2 \cdot 73$ & $4$ & $\stargraph$ & 1 & N/A \\
    $U_4(3)$ & $2^7 \cdot 3^6 \cdot 5 \cdot 7$ & $8$ & $\kyle$ & 36 & $A_7, \PSL(3,4)$\\
    $U_5(2)$ & $2^{10} \cdot 3^5 \cdot 5 \cdot 11$ & $2$ & $\spoon$ & 1 & $C_{11} \rtimes C_5$\\
    $^2 F_4(2)'$ & $2^{11} \cdot 3^3 \cdot 5^2 \cdot 13$ & $2$ & $\spoon$ & 1 & N/A \\
    $G_2(3)$ & $2^6 \cdot 3^6 \cdot 7 \cdot 13$ & $2$ & $\kyle$ & 3 & $C_7 \rtimes C_6$, $C_{13} \rtimes C_6$ \\
    $^3 D_4(2)$ & $2^{12} \cdot 3^4 \cdot 7^2 \cdot 13$ & $3$ & $\stargraph$ & 1 & N/A \\
    $\PSL(2,p^f)$ & $\frac{q(q^2-1)}{(2,q-1)}, q = p^f$ & $f \cdot (2,q-1)$ & $\kyle$ & $(2,q-1)$ &Varies\\
    \hline
\end{tabular}
    \caption{$K_4$ groups and some important structural information. Some important subgroups that were used for classification are included. $M(G)$ denotes the Schur multiplier of $G$.}
    \label{table:1}
\end{table}

\begin{table}[H]
\centering
\begin{tabular}{|c||c|c|c|c|c|}
    \hline
    $G$ & $2$-subgroup & $p_1$-subgroup & $p_2$-subgroup & $p_3$-subgroup & \# \\ \hline \hline
    $A_7$ & Yes & No & Yes & Yes & 1 \\
    $A_8$ & No & No & Yes & Yes & 2 \\
    $A_9$ & No & No & Yes & Yes & 2 \\
    $A_{10}$ & No & No & No & Yes & 3 \\
    $M_{11}$ & No & No & Yes & Yes & 2 \\
    $M_{12}$ & No & No & Yes & Yes & 2 \\
    $J_2$ & No & No & No & Yes & 3 \\
    $\PSL(3,4)$ & No & No & Yes & Yes & 2 \\
    $\PSL(3,5)$ & No & Yes & No & Yes & 2 \\
    $\PSL(3,7)$ & No & No & No & Yes & 3 \\
    $\PSL(3,8)$ & No & Yes & No & Yes & 2 \\
    $\PSL(3,17)$ & No & Yes & No & Yes & 2 \\
    $\PSL(4,3)$ & No & No & Yes & Yes & 2 \\
    $O_5(4)$ & No & No & No & Yes & 3 \\
    $O_5(5)$ & No & No & No & Yes & 3 \\
    $O_5(7)$ & No & No & Yes & No & 3 \\
    $O_5(9)$ & No & No & No & Yes & 3 \\
    $O_7(2)$ & No & No & Yes & Yes & 2 \\
    $O_8^+(2)$ & No & No & No & Yes & 3 \\
    $\Sz(8)$ & No & Yes & Yes & Yes & 1 \\
    $\Sz(32)$ & No & Yes & Yes & Yes & 1 \\
    $U_3(4)$ & No & Yes & No & Yes & 2 \\
    $U_3(5)$ & No & No & No & Yes & 3 \\
    $U_3(7)$ & No & Yes & No & Yes & 2 \\
    $U_3(8)$ & No & No & Yes & Yes & 2 \\
    $U_3(9)$ & No & No & No & Yes & 3 \\
    $U_4(3)$ & No & No & Yes & Yes & 2 \\
    $U_5(2)$ & No & No & Yes & Yes & 2 \\
    ${}^2F_4(2)'$ & No & No & No & Yes & 3 \\
    $G_2(3)$ & No & No & Yes & Yes & 2 \\
    ${}^3D_4(2)$ & No & No & No & Yes & 3\\
    \hline
\end{tabular}
    \caption{A list of $K_4$ groups (beside $\PSL(2,q)$) and whether each Sylow subgroup satisfies the Frobenius Criterion, where $2 < p_1 < p_2 < p_3$. The number counts how many of the Sylow subgroups \textbf{do not} satisfy the Criterion. In our case, a higher number means that the case was easier to handle.} 
    \label{table:2}
\end{table}

\begin{center}
\begin{table}[H]
\centering
\begin{tabular}{|l||p{10cm}|}
\hline
$G$ & Fixed Point Information \\ \hline \hline
$A_7$, $3.A_7$, $16.A_7$, $64.A_7$ & (2, 3, 5, 7), (2, 3, 5) \\ \hline
$2.A_7$ & (2, 3, 5, 7), (2, 3, 5), (3, 5, 7), (3, 7) \\ \hline 
$6.A_7$ &(2, 3, 5, 7), (2, 3, 5), (3, 5, 7), (3, 5), (3, 7)  \\ \hline
$A_8$ & (2, 3, 5, 7) \\ \hline
$2.A_8$ & (2, 3, 5, 7), (2, 3, 7) \\ \hline
$A_9$ & (2, 3, 5, 7) \\ \hline
$2.A_9$ & (2, 3, 5, 7), (2, 3, 7) \\ \hline
$M_{11}$ & (2, 3, 5, 11), (2, 3, 5) \\ \hline
$M_{12}$ & (2, 3, 5, 11) \\ \hline
$2.M_{12}$ & (2, 3, 5, 11), (2, 3, 5) \\ \hline 
$n.\PSL(3,4)$, $n \mid 48$, $6 \nmid n$ & (2, 3, 5, 7) \\ \hline
$n.\PSL(3,4)$, $n \mid 48$, $6 \mid n$ & (2, 3, 5, 7), (2, 3, 5) \\ \hline
$\PSL(3, 5)$ & (2, 3, 5, 31), (2, 3, 5)\\ \hline
$\PSL(3,7)$, $3.\PSL(3,7)$ & (2, 3, 7, 19) \\ \hline
$\PSL(3, 8)$ & (2, 3, 7, 73), (2, 3, 7)\\ \hline
$\PSL(3, 17)$ & (2, 3, 17, 307), (2, 3, 17) (from \cite{PSL})\\ \hline
$\PSL(4, 3)$, $2.\PSL(4, 3)$ & (2, 3, 5, 13)\\ \hline
$O_7(2)$ & (2, 3, 5, 7) \\ \hline
$2.O_7(2)$ & (2, 3, 5, 7), (2, 3, 7) \\ \hline
$U_3(4)$ & (2, 3, 5, 13), (2, 3, 5)\\ \hline
$U_3(5)$, $3.U_3(5)$ & (2, 3, 5, 7)\\ \hline
$U_3(7)$ & (2, 3, 7, 43), (2, 3, 7)\\ \hline
$U_3(8)$, $3.U_3(8)$ & (2, 3, 7, 19)\\ \hline
$1, 2, 3_1, 3_2, 4, 6_2, 12_2, \&\, 18_2.U_4(3)$ & (2, 3, 5, 7) \\ \hline
$6_1, 12_1, 18_1, \&\, 36. U_4(3)$ & (2, 3, 5, 7), (2, 3, 5) \\ \hline
$U_5(2)$ & (2, 3, 5, 11), (2, 3, 5)\\ \hline
${}^2F_4(2)'$ & (2, 3, 5, 13) \\ \hline 
$G_2(3)$, $3.G_2(3)$ & (2, 3, 7, 13) \\ \hline
\end{tabular}
\captionsetup{singlelinecheck=false}
\caption{General information on complex irreducible representations of relevant groups, computed via \cite{GAP4}. The symbols $n.T$ denote a perfect extension of a $K_4$ group $T$ by a group of order $n$ (often these groups are unique, but if not we refer to all of them). Many of these groups can be found in GAP using the PerfectGroup() function (see Section \ref{section:GAP}). Except for $16.A_7$ and $64.A_7$, which are PerfectGroup($8!, 2$) and PerfectGroup($4 \cdot 8!$, 1) respectively, this table only refers to perfect \textit{central} extensions. The existence of fixed points was determined using \cite{GAP4}, \cite{atlas}, and the following formula for the multiplicity of the trivial representation in the representation $V$ with character $\chi$ restricted to a cyclic subgroup $\langle x \rangle$ of order $r$: 
\[m = \frac{1}{r}\sum_{j=1}^r \chi(x^j).\] A row vector $(p_1, p_2,\ldots,  p_n)$ appears in the list if and only if there exists an irreducible representation of $G$ for which there exist elements of order $p_1, \ldots, p_n$ with fixed points while elements of other prime orders act Frobeniusly (recall that ``fixed points" are, by definition, non-zero).}
\label{table:3}
\end{table}
\end{center}

\begin{table}[H]
\centering
\begin{tabular}{|l||p{10cm}|}
    \hline
    $G$ & Information about Irr$(G)$ \\ \hline \hline
     $A_5$ & (2, 3, 5), (2, 3) \\ \hline
     $S_5$ & (2, 3, 5), (2, 3) \\ \hline
     $\SL(2,5)$ & (2, 3, 5), (2, 3), (3, 5), (3), () \\ \hline
     $C_2.S_5$ & (2, 3, 5), (2, 3), (2) \\ \hline
\end{tabular}
\caption{General information on irreducible complex representations of $A_5$, $S_5$, and related groups, computed via \cite{GAP4}, using the same method as Table \ref{table:3}. The extension $C_2.S_5$ refers to the group SmallGroup(240,90). A row vector $(p_1, p_2,\ldots,  p_n)$ appears in the list if and only if there exists an irreducible representation of $G$ with respect to which there exist elements of order $p_1, \ldots , p_n$ with fixed points while elements of other prime orders act Frobeniusly.}
\label{table:4}
\end{table}

\clearpage

\newcommand{\Aseven}{
\begin{tikzpicture}[main/.style = {draw, circle, scale = 0.8}]
    \node[main] (c) {7};
    \node[main] (a)[above left= 0.45cm and 0.45cm of c] {3};
    \node[main] (b)[below left= 0.45cm and 0.45cm of a] {5};
    \node[main] (d)[below left= 0.45cm and 0.45cm of c] {2};
    
    \draw (a) -- (b);
    \draw (a) -- (c);
    \draw (b) -- (c);
    \draw (b) -- (d);
    \draw (c) -- (d);
\end{tikzpicture}
}

\newcommand{\Aeight}{
\begin{tikzpicture}[main/.style = {draw, circle, scale = 0.8}]
    \node[main] (c) {7};
    \node[main] (a)[above left= 0.45cm and 0.45cm of c] {5};
    \node[main] (b)[below left= 0.45cm and 0.45cm of c] {2};
    \node[main] (d)[right= 0.45cm of c] {3};
    
    \draw (a) -- (b);
    \draw (a) -- (c);
    \draw (b) -- (c);
    \draw (c) -- (d);
\end{tikzpicture}
}

\newcommand{\Meleven}{
\begin{tikzpicture}[main/.style = {draw, circle, scale = 0.8}]
    \node[main] (c) {11};
    \node[main] (a)[above left= 0.45cm and 0.45cm of c] {3};
    \node[main] (b)[below left= 0.45cm and 0.45cm of a] {5};
    \node[main] (d)[below left= 0.45cm and 0.45cm of c] {2};
    
    \draw (a) -- (b);
    \draw (a) -- (c);
    \draw (b) -- (c);
    \draw (b) -- (d);
    \draw (c) -- (d);
\end{tikzpicture}
}

\newcommand{\Mtwelve}{
\begin{tikzpicture}[main/.style = {draw, circle, scale = 0.8}]
    \node[main] (c) {11};
    \node[main] (a)[above left= 0.45cm and 0.45cm of c] {5};
    \node[main] (b)[below left= 0.45cm and 0.45cm of c] {3};
    \node[main] (d)[right= 0.45cm of c] {2};
    
    \draw (a) -- (b);
    \draw (a) -- (c);
    \draw (b) -- (c);
    \draw (c) -- (d);
\end{tikzpicture}
}

\newcommand{\PSLthreefour}{
\begin{tikzpicture}[main/.style = {draw, circle, scale = 0.8}]
    \node[main] (c) {7};
    \node[main] (a)[above left= 0.45cm and 0.45cm of c] {3};
    \node[main] (b)[below left= 0.45cm and 0.45cm of a] {5};
    \node[main] (d)[below left= 0.45cm and 0.45cm of c] {2};
    
    \draw (a) -- (b);
    \draw (a) -- (c);
    \draw (a) -- (d);
    \draw (b) -- (c);
    \draw (b) -- (d);
    \draw (c) -- (d);
\end{tikzpicture}
}

\newcommand{\PSLthreefive}{
\begin{tikzpicture}[main/.style = {draw, circle, scale = 0.8}]
    \node[main] (c) {31};
    \node[main] (a)[above left= 0.45cm and 0.45cm of c] {5};
    \node[main] (b)[below left= 0.45cm and 0.45cm of c] {3};
    \node[main] (d)[right= 0.45cm of c] {2};
    
    \draw (a) -- (b);
    \draw (a) -- (c);
    \draw (b) -- (c);
    \draw (c) -- (d);
\end{tikzpicture}
}

\newcommand{\PSLthreeseven}{
\begin{tikzpicture}[main/.style = {draw, circle, scale = 0.8}]
    \node[main] (c) {19};
    \node[main] (a)[above left= 0.45cm and 0.45cm of c] {7};
    \node[main] (b)[below left= 0.45cm and 0.45cm of c] {3};
    \node[main] (d)[right= 0.45cm of c] {2};
    
    \draw (a) -- (b);
    \draw (a) -- (c);
    \draw (b) -- (c);
    \draw (c) -- (d);
\end{tikzpicture}
}

\newcommand{\PSLthreeeight}{
\begin{tikzpicture}[main/.style = {draw, circle, scale = 0.8}]
    \node[main] (c) {73};
    \node[main] (a)[above left= 0.45cm and 0.45cm of c] {3};
    \node[main] (b)[below left= 0.45cm and 0.45cm of c] {2};
    \node[main] (d)[right= 0.45cm of c] {7};
    
    \draw (a) -- (b);
    \draw (a) -- (c);
    \draw (b) -- (c);
    \draw (c) -- (d);
\end{tikzpicture}
}

\newcommand{\PSLthreeseventeen}{
\begin{tikzpicture}[main/.style = {draw, circle, scale = 0.8}]
    \node[main] (c) {307};
    \node[main] (a)[above left= 0.45cm and 0.45cm of c] {17};
    \node[main] (b)[below left= 0.475cm and 0.475cm of c] {3};
    \node[main] (d)[right= 0.45cm of c] {2};
    
    \draw (a) -- (b);
    \draw (a) -- (c);
    \draw (b) -- (c);
    \draw (c) -- (d);
\end{tikzpicture}
}

\newcommand{\PSLfourthree}{
\begin{tikzpicture}[main/.style = {draw, circle, scale = 0.8}]
    \node[main] (c) {13};
    \node[main] (a)[above left= 0.45cm and 0.45cm of c] {5};
    \node[main] (b)[below left= 0.45cm and 0.45cm of c] {3};
    \node[main] (d)[right= 0.45cm of c] {2};
    
    \draw (a) -- (b);
    \draw (a) -- (c);
    \draw (b) -- (c);
    \draw (c) -- (d);
\end{tikzpicture}
}

\newcommand{\Szsmall}{
\begin{tikzpicture}[main/.style = {draw, circle, scale = 0.8}]
    \node[main] (a) {13};
    \node[main] (c)[below right= 0.45cm and 0.45cm of a] {7};
    \node[main] (b)[below left= 0.45cm and 0.45cm of a] {5};
    \node[main] (d)[below= 1.15cm of a] {2};
    \draw (a) -- (b);
    \draw (a) -- (c);
    \draw (a) -- (d);
    \draw (b) -- (c);
    \draw (b) -- (d);
    \draw (c) -- (d);
\end{tikzpicture}
}

\newcommand{\Szbig}{
\begin{tikzpicture}[main/.style = {draw, circle, scale = 0.8}]
    \node[main] (c) {31};
    \node[main] (a)[above left= 0.45cm and 0.45cm of c] {5};
    \node[main] (b)[below left= 0.45cm and 0.45cm of a] {41};
    \node[main] (d)[below left= 0.45cm and 0.45cm of c] {2};
    \draw (a) -- (b);
    \draw (a) -- (c);
    \draw (a) -- (d);
    \draw (b) -- (c);
    \draw (b) -- (d);
    \draw (c) -- (d);
\end{tikzpicture}
}

\newcommand{\Uthreefour}{
\begin{tikzpicture}[main/.style = {draw, circle, scale = 0.8}]
    \node[main] (c) {13};
    \node[main] (a)[above left= 0.45cm and 0.45cm of c] {3};
    \node[main] (b)[below left= 0.45cm and 0.45cm of c] {2};
    \node[main] (d)[right= 0.45cm of c] {5};
    
    \draw (a) -- (b);
    \draw (a) -- (c);
    \draw (b) -- (c);
    \draw (c) -- (d);
\end{tikzpicture}
}

\newcommand{\Uthreefive}{
\begin{tikzpicture}[main/.style = {draw, circle, scale = 0.8}]
    \node[main] (c) {7};
    \node[main] (a)[above left= 0.45cm and 0.45cm of c] {5};
    \node[main] (b)[below left= 0.45cm and 0.45cm of c] {3};
    \node[main] (d)[right= 0.45cm of c] {2};
    
    \draw (a) -- (b);
    \draw (a) -- (c);
    \draw (b) -- (c);
    \draw (c) -- (d);
\end{tikzpicture}
}

\newcommand{\Uthreeseven}{
\begin{tikzpicture}[main/.style = {draw, circle, scale = 0.8}]
    \node[main] (c) {43};
    \node[main] (a)[above left= 0.45cm and 0.45cm of c] {7};
    \node[main] (b)[below left= 0.45cm and 0.45cm of c] {3};
    \node[main] (d)[right= 0.45cm of c] {2};
    
    \draw (a) -- (b);
    \draw (a) -- (c);
    \draw (b) -- (c);
    \draw (c) -- (d);
\end{tikzpicture}
}

\newcommand{\Uthreeeight}{
\begin{tikzpicture}[main/.style = {draw, circle, scale = 0.8}]
    \node[main] (c) {19};
    \node[main] (a)[above left= 0.45cm and 0.45cm of c] {7};
    \node[main] (b)[below left= 0.45cm and 0.45cm of c] {2};
    \node[main] (d)[right= 0.45cm of c] {3};
    
    \draw (a) -- (b);
    \draw (a) -- (c);
    \draw (b) -- (c);
    \draw (c) -- (d);
\end{tikzpicture}
}

\newcommand{\Ufourthree}{
\begin{tikzpicture}[main/.style = {draw, circle, scale = 0.8}]
    \node[main] (c) {7};
    \node[main] (a)[above left= 0.45cm and 0.45cm of c] {3};
    \node[main] (b)[below left= 0.45cm and 0.45cm of a] {5};
    \node[main] (d)[below left= 0.45cm and 0.45cm of c] {2};
    
    \draw (a) -- (b);
    \draw (a) -- (c);
    \draw (b) -- (c);
    \draw (b) -- (d);
    \draw (c) -- (d);
\end{tikzpicture}
}

\newcommand{\Ufivetwo}{
\begin{tikzpicture}[main/.style = {draw, circle, scale = 0.8}]
    \node[main] (c) {11};
    \node[main] (a)[above left= 0.45cm and 0.45cm of c] {5};
    \node[main] (b)[below left= 0.45cm and 0.45cm of c] {2};
    \node[main] (d)[right= 0.45cm of c] {3};
    
    \draw (a) -- (b);
    \draw (a) -- (c);
    \draw (b) -- (c);
    \draw (c) -- (d);
\end{tikzpicture}
}

\newcommand{\tits}{
\begin{tikzpicture}[main/.style = {draw, circle, scale = 0.8}]
    \node[main] (c) {13};
    \node[main] (a)[above left= 0.45cm and 0.45cm of c] {5};
    \node[main] (b)[below left= 0.45cm and 0.45cm of c] {3};
    \node[main] (d)[right= 0.45cm of c] {2};
    
    \draw (a) -- (b);
    \draw (a) -- (c);
    \draw (b) -- (c);
    \draw (c) -- (d);
\end{tikzpicture}
}

\newcommand{\Gtwothree}{
\begin{tikzpicture}[main/.style = {draw, circle, scale = 0.8}]
    \node[main] (c) {13};
    \node[main] (a)[above left= 0.45cm and 0.45cm of c] {3};
    \node[main] (b)[left= 1.11cm of c] {7};
    \node[main] (d)[below left= 0.45cm and 0.45cm of c] {2};
    
    \draw (a) -- (b);
    \draw (a) -- (c);
    \draw (b) -- (c);
    \draw (b) -- (d);
    \draw (c) -- (d);
\end{tikzpicture}
}

\renewcommand{\arraystretch}{1.2}
\begin{table}[H]
    \begin{centering}
    \begin{tabular}{|c|c||c|c||c|c|}\hline
    $G$ & $\ol\Gamma(G)$ & $G$ & $\ol\Gamma(G)$ & $G$ & $\ol\Gamma(G)$\\\hline\hline
    $A_7$ & \Aseven & $A_8$ & \Aeight & $M_{11}$ & \Meleven\\\hline
    $M_{12}$ & \Mtwelve & $\PSL(3,4)$ & \PSLthreefour & $\PSL(3,5)$ & \PSLthreefive\\\hline
    $\PSL(3,7)$ & \PSLthreeseven & $\PSL(3,8)$ & \PSLthreeeight & $\PSL(3,17)$ & \PSLthreeseventeen\\\hline
    $\PSL(4,3)$ & \PSLfourthree & $Sz(8)$ & \Szsmall & $Sz(32)$ & \Szbig\\\hline
    $U_3(4)$ & \Uthreefour & $U_3(5)$ & \Uthreefive & $U_3(7)$ & \Uthreeseven\\\hline
    $U_3(8)$ & \Uthreeeight & $U_4(3)$ & \Ufourthree & $U_5(2)$ & \Ufivetwo\\\hline
    $^2F_4(2)'$ & \tits & $G_2(3)$ & \Gtwothree & \multicolumn{2}{c|}{}\\\hline
    \end{tabular}
    \caption{$K_4$ groups and their prime graphs, labeled. Only groups not in the family of $\PSL(2,q)$ groups and groups whose prime graphs have triangles are included.}
    \label{table:5}
\end{centering}
\end{table}

\newpage
\begin{center}
{\bf Acknowledgements}
\end{center}
\vspace{3pt}
The authors thank Eitan Marcus for useful discussions on the topic of this paper and Dr. Alexander Hulpke for assistance with GAP in constructing of groups and extensions. This research was conducted under NSF-REU grant DMS-1757233 and DMS-2150205 by the second, third, and fourth authors during the summer of 2023 under the supervision of the first author. The authors gratefully acknowledge the financial support of NSF for this REU. The second, third, and fourth authors thank the first author and mentor, Dr. Thomas Michael Keller, for his invaluable guidance and expertise throughout the program. Those authors also thank Texas State University for the great working environment and support provided. 


\begin{thebibliography}{99}
    \bibitem{aschbacher}
    Aschbacher, M.: Finite Group Theory. Corrected reprint of the 1986 original. Cambridge Studies in Advanced Mathematics, 10. Cambridge University Press, Cambridge, (1993)

    \bibitem{baechle}
    B{\"a}chle, A., Kiefer, A., Maheshwary, S., del R\'io, \'Angel: Gruenberg-Kegel graphs: cut groups, rational groups and the prime graph question, Forum Math. \textbf{35} (2023), no. 2, 409–429. 
    
    \bibitem{K4}
    Bugeaud, Y., Zhenfu, C., Mignotte, M.: On Simple $K_4$-groups. J. Algebra \textbf{241}, 658-668 (2001)

    \bibitem{burichenko}
    Burichenko, V.P.: Extensions of abelian 2-groups by $L_2(q)$ under irreducible action. Algebra and Logic \textbf{39}, 160-183 (2000)

    \bibitem{atlas}
    Conway, J. H. and Curtis, R. T. and Norton, S. P. and Parker,
    R. A. and Wilson, R. A.: Atlas of Finite Groups. Oxford University Press, (1985)

    \bibitem{Curtis+Reiner}
    Curtis, C.W., Reiner, I.: Representation theory of finite groups and associative algebras. John Wiley \& Sons Inc, (1962)

    \bibitem{debon} Deb\'on, Sara C.; Garc\'ia-Lucas, Diego; del R\'io, \'Angel: The Gruenberg-Kegel graph of finite solvable rational groups. 
    J. Algebra \textbf{642} (2024), 470–479. 

    \bibitem{2022}
    Edwards, T.J., Keller, T.M., Pesak, R.M., Sellakumaran Latha, K.: The prime graphs of groups with arithmetically small prime factors.  Ann. Mat. Pura Appl. (4) \textbf{203} (2024), no. 2, 945–973

    \bibitem{Flavell}
     Flavell, P.: A Hall-Higman-Shult type theorem for arbitrary finite groups. Invent. Math. \textbf{164}(2), 361–397 (2006)

    \bibitem{REU}
    Florez, C., Higgins, J., Huang, K., Keller, T. M., Shen, D., Yang, Y. The prime graphs of some classes of finite groups. J. Pure Appl. Algebra \textbf{226}, (2022)

    \bibitem{GAP4}
    GAP–Groups, Algorithms, and Programming, Version 4.12.2. The GAP Group. (2022)

    \bibitem{siberian}
    Gavrilyuk, A. L., Khramtsov, I. V., Kondrat'ev, A. S., Maslova, N. V.: On realizability of a graph as the prime graph of a finite group, Sib. Èlektron. Mat. Izv. 11 (2014), 246–257. 

    \bibitem{2015}
    Gruber, A., Keller, T.M., Lewis, M.L., Naughton, K., Strasser, B.: A characterization of the prime graphs of solvable groups. J. Algebra \textbf{442}, 397–422 (2015)

    \bibitem{A5}
    Huang, Z., Keller, T.M., Kissinger, S., Plotnick, W., Roma, M., Yang, Y.: A classification of the prime graphs of pseudo-solvable groups.  J. Group Theory \textbf{27}, 89–117 (2024) 

    \bibitem{FGT}
    Isaacs, I.M.: Finite Group Theory. American Mathematical Society, (2008)

    \bibitem{lucido}
    Lucido, M. S.: Groups in which the prime graph is a tree, Boll. Unione Mat. Ital. Sez. B Artic. Ric. Mat. (8) \textbf{5} (2002), no. 1, 131–148.

    \bibitem{Manz_Wolf_1993}
    Manz, O., Wolf, T.R.: Representations of Solvable Groups. Cambridge University Press, (1993)

    \bibitem{Maslova}
    Maslova, N.V.: On the Gruenberg–Kegel graphs of finite groups. In: Proceedings of the 47th International Youth School Conference-Modern Problems in Mathematics and its Applications. Ed. by S.F. Pravdin A.A. Makhnev. Jan. (2016)

    \bibitem{Meierfrankenfeld}
    Meierfrankenfeld, U.: Perfect Frobenius complements. Arch. Math. \textbf{79}, no. 1, 19-26 (2002)

    \bibitem{PSL}
    Simpson, W. A., Frame, J. S: The Character Tables for $\SL(3, q)$, $\operatorname{SU}(3, q^2)$, $\PSL(3, q)$, $\operatorname{PSU}(3, q^2)$.  Canadian Journal of Mathematics {\bf 25(3)}, 486–494 (1973)
    
    \bibitem{1981}
    Williams, J.S.: Prime graph components of finite groups. J. Algebra \textbf{69.2}, 487–513 (1981)
    
\end{thebibliography}
\end{document}